\documentclass[english,letterpaper]{article}

\usepackage{hyperref}
\hypersetup{
    colorlinks=true,
    linkcolor=blue,
    filecolor=red,      
    urlcolor=red,
	citecolor=red,
}
\usepackage{graphicx} 
\usepackage{bm}

\usepackage{tikz}
\tikzset{
    axis/.style={thick,->},
    wedge/.style={thick,blue},
    airycurve/.style={thick,blue}
 }

\newcommand{\anglabel}[3]{
    \node at ({#1*cos(#2)},{#1*sin(#2)}) {#3};
}
\usepackage[T1]{fontenc}
\usepackage[utf8]{inputenc}
\usepackage{amsmath}
\usepackage{amsthm}
\usepackage{amssymb}
\usepackage{enumitem}

\global\long\def\re{\text{Re }}%
\global\long\def\im{\text{Im }}%
\global\long\def\arccosh{\text{arccosh}}%

\usepackage[margin=3cm]{geometry}

\usepackage[style=alphabetic]{biblatex}
\renewbibmacro{in:}{}

\DeclareFieldFormat[article,inbook,incollection,inproceedings,patent,thesis,unpublished]{title}{#1} 

\newcommand{\ab}[1]{\left[ #1 \right]}
\newcommand{\abs}[1]{\left| #1 \right|}
\newcommand{\norm}[1]{\left\lVert #1 \right\rVert}
\newcommand{\rb}[1]{\left( #1 \right)}
\newcommand{\set}[1]{\left\{ #1 \right\}}
\newcommand{\wt}[1]{\widetilde{#1}}

\newcommand{\one}{\mathbf{1}}

\renewcommand{\d}{\mathrm{d}}

\newcommand{\rh}{\rho}
\newcommand{\si}{\sigma}
\newcommand{\Si}{\Sigma}
\newcommand{\ps}{\psi}
\newcommand{\ph}{\phi}
\renewcommand{\th}{\theta}
\newcommand{\ld}{\ldots}
\newcommand{\given}{\vert}

\newcommand{\uind}[2]{{#1}^{(#2)}}

\newcommand{\empmeas}[2]{\mu_{#1,#2}}

\makeatletter
\theoremstyle{definition}
\newtheorem*{defn*}{\protect\definitionname}
\theoremstyle{remark}
\newtheorem*{rem*}{\protect\remarkname}
\theoremstyle{plain}
\newtheorem*{thm*}{\protect\theoremname}
\theoremstyle{plain}
\newtheorem*{prop*}{\protect\propositionname}

\theoremstyle{plain} 
\newtheorem{theorem}{Theorem}[section]
\theoremstyle{definition} 
\newtheorem{definition}[theorem]{Definition}
\newtheorem{example}[theorem]{Example}
\newtheorem{proposition}[theorem]{Proposition}
\theoremstyle{remark} 
\newtheorem{remark}[theorem]{Remark}
\newtheorem{lemma}[theorem]{Lemma}
\newtheorem{corollary}[theorem]{Corollary}

\makeatother

\usepackage{babel}
\usepackage{csquotes}
\providecommand{\definitionname}{Definition}
\providecommand{\propositionname}{Proposition}
\providecommand{\remarkname}{Remark}
\providecommand{\theoremname}{Theorem}

\title{A Gaussian integral formula for the Hermite polynomials: Combinatorics, Asymptotics and Applications}

\author{%
  Mihai Nica \\
  Department of Mathematics and Statistics\\
  University of Guelph\\
  \texttt{nicam@uoguelph.ca} \\
  \and
  Janosch Ortmann \\
  D\'epartement AOTI\\
  Universit\'e du Qu\'ebec \`a Mont\'eal\\
  \texttt{ortmann.janosch@uqam.ca} \\
}
\date{} 

\newcommand{\de}{\delta}
\newcommand{\ep}{\epsilon}

\newcommand{\bC}{\mathbb{C}}
\newcommand{\bE}{\mathbb{E}}
\newcommand{\bN}{\mathbb{N}}
\newcommand{\bP}{\mathbb{P}}
\newcommand{\p}{\mathbb{P}}

\newcommand{\bR}{\mathbb{R}}
\newcommand{\bZ}{\mathbb{Z}}

\newcommand{\cF}{\mathcal{F}}

\newcommand{\cN}{\mathcal{N}}

\newcommand{\cP}{\mathcal{P}}

\newcommand{\fH}{H\!\!\operatorname{e}} 
\newcommand{\Id}{\operatorname{Id}} 

\newcommand{\A}{A\!\operatorname{i}} 

\newcommand{\vp}{\varphi}
\newcommand{\AiryCurve}{\langle}
\newcommand{\RotatedAiryCurve}{\wedge}

\newcommand{\mg}[2]{M^{(#1)}_{#2}}
\newcommand{\half}{\frac{1}{2}}

\newcommand{\sgn}{\operatorname{sgn}}
\newcommand{\dysonkernel}{K^{\text{DBM}}}
\newcommand{\rescdysonkernel}{\wt{K}^{\text{DBM}}}
\newcommand{\spikedkernel}{K^{(c)}}
\newcommand{\conjspikedkernel}{\wt K^{(c)}}

\newcommand{\npart}{N_{n,\sigma}}
\newcommand{\npartspiked}{N^{(c)}_{n,\sigma}}

\begin{document}

	\maketitle

	\begin{abstract}
	    The Hermite polynomials are ubiquitous but can be difficult to work with due to their unwieldy definition in terms of derivatives. To remedy this, we showcase an underappreciated Gaussian integral formula for the Hermite polynomials, which is especially useful for generalizing to multivariable Hermite polynomials. Taking this as our definition, we prove many useful consequences, including:
        \\ 1. Combinatorial interpretations for the Hermite polynomials, including a proof of orthogonality. \\ 2. A more elementary proof of Plancherel--Rotach asymptotics that does not involve residues. \\ 3. Limit theorems for GUE random matrices and Dyson's Brownian motion, including bulk convergence to the semi-circle law and edge convergence to the Airy limit/Tracy-Widom law. \\ 4. An analysis of a phase transition in the spiked GUE random matrix as the top eigenvalue goes from well-separated to attached to the bulk, analogous to the BBP phase transition. \\ 5. Elementary derivations of Edgeworth expansions and multivariable Edgeworth expansions. \\ This article is primarily expository and features many illustrative figures.  
	\end{abstract}
    
    \tableofcontents

    \listoffigures

\section{Introduction to the formula and its uses}
The Hermite polynomials, $\fH_n(x)$\footnote{Note that $\fH_n(x)$ are the \emph{probabilist}'s Hermite polynomials, which are not to be confused with the \emph{physicist}'s Hermite polynomials $H_n(x)$. See Section \ref{sec:physicists} for details.}, are the most well known of all the orthogonal polynomials due to the importance of the Gaussian distribution, $\frac{1}{\sqrt{2\pi}} e^{-\half x^2}$, which is the underlying weight function that makes them orthogonal,     \begin{equation}\label{eq:orthog_1} \intop_{-\infty}^\infty \fH_n(x)\fH_m(x) \frac{1}{\sqrt{2\pi}} e^{-\half x^2} \d x = \begin{cases} 0 &\text{ if }n\neq m \\ n! &\text{ if }n=m\end{cases}.\end{equation}
Their most common definition is through the derivative formula $\fH_n(x)=(-1)^n e^{\half x^2} \frac{\d ^n}{\d x^n}e^{-\half x^2}$ or the recursion $\fH_{n+1}(x)=x\fH_n(x)-n\fH_{n-1}(x)$. However, these classical definitions can be inconvenient! In contrast, we champion here a Gaussian integral formula for the Hermite polynomials. Equivalently, this integral formula can be thought of as an expectation over Gaussian random variables.
    \begin{definition}\label{def:Gaussian_def}
    The Gaussian integral formula, or equivalently the Gaussian expectation formula, for the Hermite polynomials $\fH_n$ is 
    \begin{equation} \label{eq:formula} \fH_n(x)=\mathbb{E}_{Z\sim\mathcal{N}(0,1)}[(x+\imath Z)^n]= \int_{-\infty}^\infty (x+\imath z)^n \frac{1}{\sqrt{2\pi}}e^{-\half z^2} \d z, \end{equation}
where $\imath = \sqrt{-1}$ is the imaginary unit\footnote{Note that despite the appearance of $\imath = \sqrt{-1}$, the Hermite polynomials are purely real because the expected odd powers of a Gaussian vanish, $\mathbb{E}[Z^k]=0$ when $k$ is odd.}, and $Z \sim \mathcal{N}(0,1)$ is a standard (mean 0, variance 1) Gaussian random variable.
\end{definition}
This formula has appeared in many places, for example Exercise 6.8 in \cite{Temme} or (4.9) in \cite{janson}. Integral formulas for other orthogonal polynomials also exist, such as (4.5.17) in \cite{ismael}. 

For probabilists, who are adept at manipulating expected values with Gaussians, this formula can be exceptionally useful. For example, the identity for the value of $\fH_n$ at $x=0$ follows from $\mathbb{E}[Z^n]=(n-1)!!$ when $n$ is even, \begin{equation}\label{eq:zero} \fH_n(0)= \mathbb{E}[(0+\imath Z)^n] = \begin{cases}(-1)^{n/2}(n-1)!!& \text{ for }n\text{ even} \\ 0& \text{ for }n\text{ odd} \end{cases}.\end{equation} 
See also Section \ref{sec:combinatorics} for a combinatorial interpretation of $\mathbb{E}[Z^n] = (n-1)!!$.

Another example is the identity for $\fH_n(x+y)$, which easily follows by the binomial theorem and linearity of expectation,
    \begin{equation} \label{eq:binomial} \fH_n(y+x)= \mathbb{E}[(y+x+\imath Z)^n] =\mathbb{E}\left[ \sum_{k=0}^n \binom{n}{k}y^k(x+\imath Z)^{n-k}\right] = \sum_{k=0}^n \binom{n}{k} y^k \fH_{n-k}(x).  \end{equation}
One can also generalize this approach to more summands to get expressions for $\fH_n(x_1+\ldots+x_m)$, see Theorem 5.5. of \cite{MollVignat}. Note also that by plugging in $x=0$ into \eqref{eq:binomial} and using \eqref{eq:zero}, one obtains an explicit sum formula for $\fH_n$. (See again Section \ref{sec:combinatorics} for a combinatorial interpretation.) 

These properties are much more difficult to extract from the traditional definitions. Even simple facts like $\frac{\d}{\d x} \fH_{n}(x)=n\fH_{n-1}(x)$ are difficult to see with the traditional definition, but immediate from the Gaussian integral formula (in this case by differentiating under the integral sign; see Lemma \ref{lem:deriv-of-H}).

\subsection{Roadmap of this article}\label{sec:roadmap}

The ethos of this article is to use the Gaussian integral formula \eqref{eq:formula} of Definition \ref{def:Gaussian_def} as our starting definition for the Hermite polynomials, and then to showcase alternate derivations of many of the properties of Hermite polynomials and obtain several applications. We also show throughout how this formula allows one to easily generalize the Hermite polynomials and their properties to multiple dimensions where possible. Below is a brief summary of each section.

\begin{enumerate}[font=\bfseries,leftmargin=2cm]
\item[\S \ref{sec:combinatorics} Combinatorics] Starting from the Wick formula/Isserilis theorem, which gives a combinatorial iterpretation to expectation of Gaussian powers, we show a combinatorial meaning for the Hermite polynomials. This is then used to prove that Definition \ref{def:Gaussian_def} actually satifies the orthogonality \eqref{eq:orthog_1} by a clever enumeration. We also generalize all of this to higher dimensional Hermite polynomials. Figure \ref{fig:h4_combinatorics} illustrates the combinatorics of this section. 
\item[ \S \ref{sec:asymptotics-hermite} Asymptotics ] By using the Laplace method for approximating integrals, we provide several different approximation formulas for $\fH_n$ that hold when $n$ is large. The most powerful of these approximations are equivalent to famous Plancherel–Rotach asymptotics for the Hermite polynomials, but because we start with the Gaussian integral formula instead of the derivative formula, our proofs are simpler (no residues needed!) and the final result does not involve $n!$. Figure \ref{fig:Hermite_approx} illustrates the approximations in this section.

\item [\S \ref{sec:applications-oscillator} $\Psi_n(x)$ and $K_n(x,y)$] We define the oscillator wave function $\Psi_n(x)$ and the Hermite kernel $K_n(x,y)$, which are a common way the Hermite polynomials appear in applications. We prove several properties of these using the results of Section \ref{sec:asymptotics-hermite}, including convergence (after appropriate rescaling) to the arcsine law (see Figure \ref{fig:arcsine}), semi-circle law (see Figure \ref{fig:semicircle}) and the Airy limit.  
\item [\S \ref{sec:Dyson} Random matrices I] The eigenvalues of the Gaussian Unitary Ensemble (GUE) Random Matrix have a distribution that is intimately connected to the Hermite polynomials. This is also the distribution of Dyson's Brownian motion or equivalently non-intersecting Brownian motions (see Figure \ref{fig:Watermelon}). We derive these probabilistic connections and use the results of Section \ref{sec:applications-oscillator} to obtain limit theorems as $n \to \infty$ for these random objects, including bulk convergence to the semi-circle law and convergence at the edge to the Airy kernel/Tracy-Widom law.

\item [\S \ref{sec:spiked} Random matrices II] It turns out that a rank one perturbation to the GUE (sometimes called a ``spiked random matrix'') also has explicit formulas given purely in terms of Hermite polynomials. By analyzing these in detail, we find a phase transition where the top eigenvalue is well separated from the bulk and then another phase where it is glued to the bulk. This is analogous to the so-called ``BBP phase transition'' \cite{BBP} that appears in the sample covariance matrix.

\item[\S \ref{sec:fourier} Fourier/Edgeworth] By simple change of variables, we transform the Gaussian integral formula into other useful integral formulas. These are related to the effect of the Hermite polynomials in Fourier space and the Gaussian integration by parts formula. The most notable application of this perhaps is the Edgeworth expansion for approximating sums of random variables more accurately than the simple the Central Limit Theorem approximation (see Figure \ref{fig:1d-edgeworth} for an example).

\item [\S \ref{sec:misc} Miscellaneous] This section collects several other miscellaneous results where the Gaussian integral formula is useful. The subsection titles give a good summary of what is here.

\end{enumerate}
       
Each section can be understood without having understood the detailed proof of the other section. However, there is a logical progression within $\S$\ref{sec:asymptotics-hermite} $\to \S$\ref{sec:applications-oscillator}  $ \to \S$\ref{sec:Dyson} $\to \S$\ref{sec:spiked}, where each subsequent section uses pre-packaged results from previous sections. In $\S$\ref{sec:asymptotics-hermite} the actual detailed asymptotic analysis of $\fH_n$ is calculated ``by hand'', then in $\S$\ref{sec:applications-oscillator}, the asymptotics are repackaged to give nice results for $\Psi_n(x)$ and $K_n(x,x)$. In $\S$\ref{sec:Dyson}, those results for convergence of functions are used to prove limit theorems for random matrices and Dyson Brownian motion and extended to spiked Dyson Brownian motion in \S \ref{sec:spiked}.

    \subsection{Generalization to arbitrary variance \texorpdfstring{$\wt H_n(x,t)$}{} and its dual \texorpdfstring{$H_n(x,t)$}{}}
    
    One place in particular where the formula particularly shines is generalizing the Hermite polynomials to work with Gaussians of arbitrary variance. A particularly useful case is to add a time dependence using a Gaussian of variance $t$. This is denoted by $\wt H_n(x,t)$ and defined as follows.
    \begin{definition}    
    $$\wt{H}_n(x,t) := \mathbb{E}_{ B_t \sim \mathcal{N}(0,t) } \left[(x+\imath B_t)^n \right] = \mathbb{E}_{Z\sim\mathcal{N}(0,1)}\left[(x+ \imath \sqrt{t} Z)^n \right].
$$ \end{definition}
  \begin{example}
   The first few polynomials $\wt H_n(x,t)$ are 
$$\wt H_1(x,t) = x, \quad \wt H_2(x,t) = x^2 - t, \quad \wt H_3(x,t) = x^3 - 3xt,\quad \wt H_4(x,t) = x^4 - 6x^2t + 3t^2$$
\end{example} $\wt{H}_n(x,t)$  has many nice probabilistic properties. For example if $B_t$ is a Brownian motion in time, then $\wt{H}_n(B_t,t)$  is a martingale for any $n$, see Section \ref{subsec:martingale}. The reason these are denoted by $\wt{H}$ in the literature instead of $H$ is a historical preference for the derivative formula instead of the integral formula. The notation $H$ is reserved for the ``dual'' to this $\wt{H}$  which are defined below
\begin{definition}
\label{def:dual-H}
 \begin{align} H_n(x,t) :=\bE_{Z\sim N(0,\frac{1}{t})}\left[\left(\frac{x}{t}+\imath Z\right)^{n}\right] =\bE_{Z\sim N(0,1)}\left[\left(\frac{x}{t}+\imath \frac{Z}{\sqrt{t}}\right)^{n}\right].
	\end{align}    
\end{definition}
 These are dual to each other in the sense that 
 $$ \intop_{-\infty}^\infty \wt H_n(x,t)  H_m(x,t) \frac{1}{\sqrt{2 \pi t}} e^{-\half \frac{x^2}{t}} \d x = \begin{cases} n! &\text{ if }n_{} = m_{} \\ 0 &\text{ if }n_{} \neq m_{}\end{cases}, $$
Note also by factoring out a factor of $t$ that $ H_n(x,t) = t^{-n} \wt H_n(x,t).$ By pulling out a factor of $\sqrt{t}$  from the bracket it is also possible to write both  in terms of the classic $\fH_n$ as follows
		\begin{align}
			\label{eq:time-space-scaling-1d}
			\wt H_n(x,t) = t^{n/2} \wt H_n\rb{\frac x{\sqrt{t}},1} = t^{n/2} \fH_n\left(\frac{x}{\sqrt{t}}\right) 
,\quad
			H_n(x,t) = t^{-n/2} H_n\rb{\frac x{\sqrt{t}},1}=t^{-n/2} \fH_n\left(\frac{x}{\sqrt{t}}\right)
		\end{align}
Note also that when $t=1$, everything is the same $\wt H_n(x,1)=H_n(x,1)=\fH_n(x)$.

\subsection{Generalization to multi-variable Hermite polynomials  \texorpdfstring{$\wt H_{\vec{n}}(\vec{x},V)$}{} and the dual \texorpdfstring{$H_{\vec{n}} (\vec{x},V)$}{}}

Hermite polynomials naturally generalize to higher dimensions, see for example \cite{Appell, AKF}.
As suggested in \cite{Withers2000}, the Gaussian expectation formula \eqref{eq:formula} can also be used here. The generalization to $d$ dimensions is to simply replace the single Gaussian $Z \sim \mathcal{N}(0,1)$ with a multi-variate Gaussian $\vec{Z} \sim \mathcal{N}(\vec{0},V)$ for some $d \times d$ covariance matrix $V$. The index of the Hermite polynomial is now a $d$-tuple $(n_1,\ldots,n_d)$, and the Hermite polynomial is the the expectation.

\begin{definition}\label{def:multi-variable-hermite}
   Let $d \in \mathbb{N}$ and let $V$ be a $d \times d$ covariance matrix. Define $\wt H_{n_1,\ldots,n_d}(\cdot;V):\mathbb{R}^d \longrightarrow \mathbb{R}$  by 
$$\wt{H}_{n_1,\ldots,n_d}(x_1,\ldots,x_d;V) :=\mathbb{E}_{\vec{Z}\sim\mathcal{N}(0,V)}\left[ \prod_{a=1}^d (x_a + \imath Z_a)^{n_a} \right].$$
The dual polynomial $H_{n_1,\ldots,n_d}$ is given by 
\begin{align}
		\label{eq:HExpectation}
		H_{\vec{n}}\left(x_1,\ldots,x_d;V\right) & :=\bE_{\vec Z\sim N(0,V^{-1})}\left[\prod_{a=1}^{d}\Big((V^{-1}\vec{x})_{a}+\imath Z_{a}\Big)^{n_{a}}\right].
	\end{align}
\end{definition}
\begin{example}
    Suppose $d=2$ and $V = \begin{bmatrix} 1 & \varrho \\ \varrho & 1 \end{bmatrix}$ where $\varrho \in (-1,1)$ is a parameter. In Example \ref{ex:4-hermite} we use a combinatorial method to explicitly calculate the following polynomials
    \begin{align}
    \wt H_{4,0}(x_1,x_2;V) &= x_1^4 - 6x_1^2 + 3 \\ 
\wt H_{3,1}(x_1,x_2;V) &= x_1^3 x_2 - 3\varrho x_1^2 - 3x_1x_2 +  3 \varrho \\ 
\wt H_{2,2}(x_1,x_2;V) &= x_1^2 x_2^2 - x_1^2-x_2^2 - 4 \varrho x_1x_2 +  2 +\varrho^2
\end{align}
\end{example}
$H$ and $\wt H$ are dual to each other in the sense that
$$\bE_{X \sim \mathcal{N}(\vec{0},V)}\left[H_{\vec{n}}(X,V)\widetilde{H}_{\vec{m}}(X,V)\right]=\begin{cases}  \prod_{j=1}^{d}n_{j}! &\text{ if }\vec{n}=\vec{m} \\ 0&\text{ if }\vec{n}\neq \vec{m} \end{cases}.$$
We provide a combinatorial proof of this orthogonality in Section \ref{sec:combinatorics}.

\subsection{A warning regarding the Physicist's Hermite polynomials \texorpdfstring{$H_n(x)$}{}}\label{sec:physicists}

Other authors use the notation $H_n(x)$ to mean the so-called ``\emph{physicist's} Hermite polynomials'', which are \emph{not} monic polynomials and which are orthogonal to the weight function $e^{-x^2}$ (no half!), namely $H_n(x) = (-1)^n e^{x^2}\frac{\d^n}{\d x^n}e^{-x^2}$. The first few are $H_1(x)=2x, H_2(x)=4x^2 - 2, H_3(x) = 8x^3 - 12x$. 

These are related to the (probabilist's) Hermite polynomials $\fH_n$ by $H_n(x) = \sqrt{2}^n \fH_n(\sqrt{2}x)$ or conversely $\fH_n(x) = H_n(x/\sqrt{2})/\sqrt{2}^{n}$. A Gaussian integral formula for these can be derived by using the one for $\fH$, for example $H_n(x) = 2^n \intop_{-\infty}^\infty (x+\imath y)^n \frac{1}{\sqrt{\pi}}e^{-y^2} \d y $ by changing variables $y = \sqrt{2}z$.

\textbf{WARNING:} Do not confuse the notation for the multivariable (probabilists) Hermite polynomials $H_{\vec{n}}(\vec{x};V)$ from Definition \ref{def:multi-variable-hermite} with the single variable physicists Hermite polynomials $H_n(x)$. Everything in this article is written using only the probabilist's Hermite polynomials $\fH_n(x)$, $\wt H_{\vec{n}}(\vec{x};V)$ or $H_{\vec{n}}(\vec{x};V)$. We do \emph{not} mention the physicists Hermite polynomials $H_n$ in this article except in this section in order to warn the reader that they are different!

\section{Combinatorics of the Hermite polynomials}\label{sec:combinatorics}

The Wick theorem/Isserlis formula gives a combinatorial meaning to the expected value of Gaussian polynomials. Using this, the Gaussian integral formula for the Hermite polynomials can be interpreted combinatorially. For example, the normalizing constant in the Hermite polynomials is $n!$, namely
$$\int_{-\infty}^{\infty} \fH_n(x)\fH_n(x) \frac{1}{\sqrt{2\pi}}e^{-\half x^2} \d x = n!$$
and our combinatorics method here will allow us to realize this $n!$ as counting permutations on $n$ elements, see Subsection \ref{subsec:orthog} for this.

Note that previous authors have derived many combinatorial interpretations of the Hermite polynomials, see \cite{HermiteCombinatorics, PolynomialCombinatorics,janson} and references therein. However, the starting point in those presentation is the recurrence relation for the Hermite polynomials, rather than the Gaussian integral formula. Our presentation here gives an extra level of meaning to the combinatorics. For example, in those papers, the authors declare that certain edge weight should be $-1$, and prove that this recovers the Hermite polynomials by induction with the recurrence relation. In contrast, our starting point realizes the $-1$ as arising from an expectation of imaginary Gaussians  squared: $-1 = \mathbb{E}[(\imath Z)(\imath Z)]$. This gives a new probabilistic meaning to the combinatorial methods. We also generalize the proofs to the multivariate Hermite polynomials. 

\subsection{Tool: Wick formula/Isserlis' theorem for Gaussian expectations }

The entry point into combinatorics is the Wick formula/Isserlis' theorem, which states that a Gaussian expectation for mean zero Gaussians can be seen as a sum over pairings. (These also sometimes go by the name of ``Feynman Diagrams'' see e.g. \cite{janson} Sections 1.5 and 3.1. Be warned that ``Feynman Diagrams'' can also refer to an entirely different set of diagrams from subatomic physics.) 

Consider as follows: if $Z_1, \ldots, Z_n$ are a collection of mean zero, jointly Gaussian random variables, then the expectation of the product $Z_1 \cdots Z_n$ can be written in terms of the individual correlations $\mathbb{E}[Z_a Z_b]$ as:
\begin{equation}\label{eq:wick} \mathbb{E}\left[ Z_1 \cdots Z_n \right] = \sum_{\pi \in \cP(\{1,2,\ldots,n\})}  \prod_{ \{a,b\} \in \pi } \mathbb{E}[Z_a Z_b ], \end{equation}
where $\cP(\{1,2,\ldots,n\})$ denotes the set of \textit{pairings} (also called \textit{pair partitions} or \textit{perfect matchings}) on the set $S = \{1, \ldots , n\}$. That is,  $\cP(S)$ denotes the set of ways to divide $S$ into disjoint pairs
$$\mathcal{P}(S) := \left\{ \left\{ \{a_1, b_1\}, \{a_2, b_2\}, \dots, \{a_{n/2}, b_{n/2}\} \right\} \,\middle|\, \bigcup_{i=1}^{n/2} \{a_i, b_i\} = S \text{ and } \{a_i, b_i\} \cap \{a_j, b_j\} = \emptyset \text{ for } i \ne j \right\}$$

\begin{figure}[!ht]
    \centering
    \includegraphics[width=0.75\linewidth]{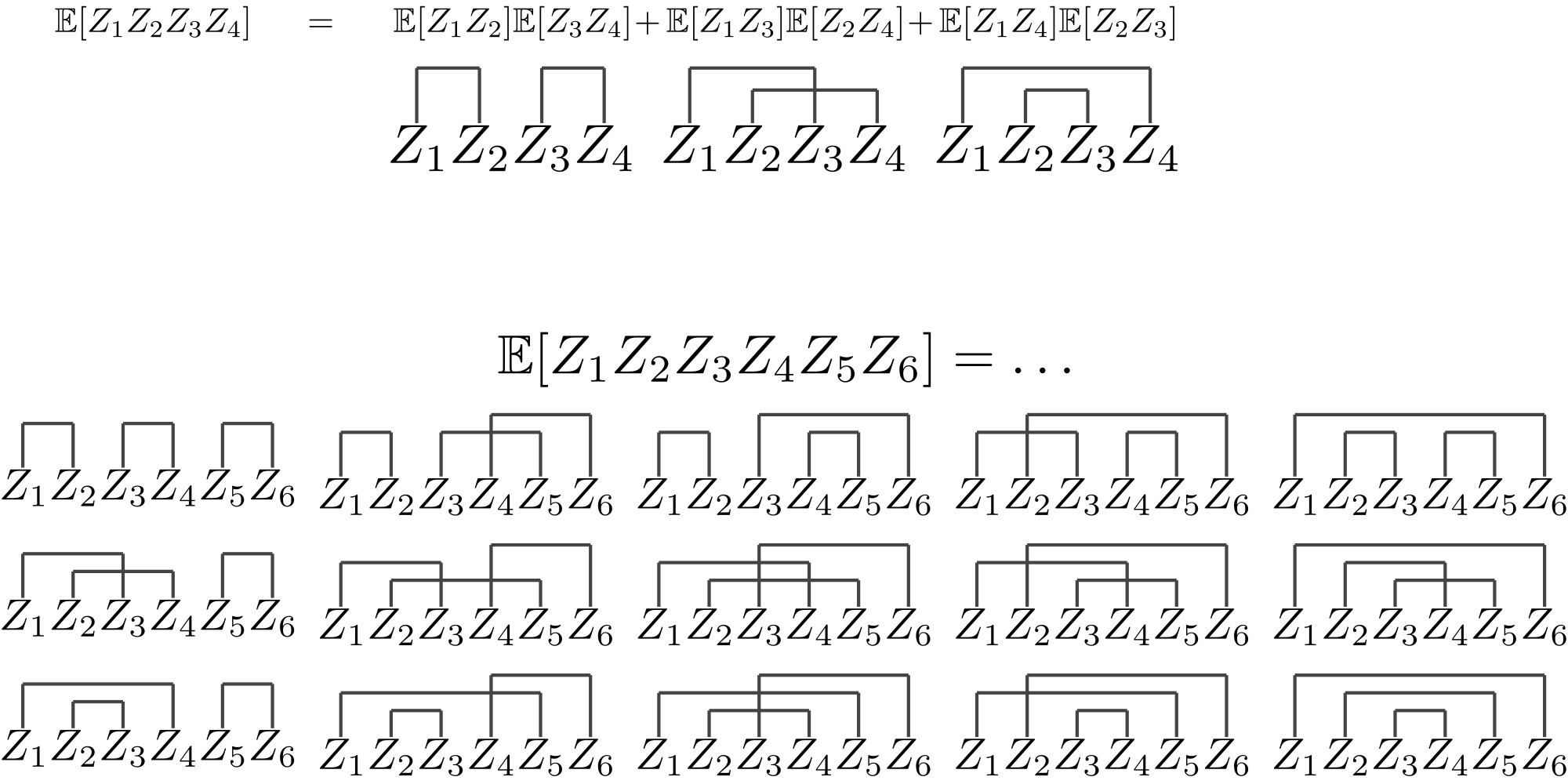}
    \caption[Illustration of Wick Theorem/Isserlis Formula when $n=4$ and $n=6$]{Two examples of the Wick/Isserlis theorem when $n=4$ and $n=6$. When $n=4$, $\mathbb{E}[Z_1 Z_2 Z_3 Z_4]$ is a sum of $3$ terms, each of which is a product of 2 expected values, and when when $n=6$ , $\mathbb{E}[Z_1 Z_2 Z_3 Z_4 Z_5 Z_6]$ is a sum of 15 terms, each of which is a product of 3. }
    \label{fig:Wick-example}
\end{figure}

Figure \ref{fig:Wick-example} shows an illustration of this in the case $n=4$ and $n=6$. Note that the $Z$'s can have any correlation structure whatsoever and can even be repeated. When one puts in the same $Z$ repeated $n$ times, one arrives at the identity for the moments of a centred Gaussian in terms of the number of such pairings, which are given by the double factorials

\begin{equation} \label{eq:powers-of-Z} \mathbb{E}[Z^n] = | \mathcal{P}(\{1,\ldots,n\}) | \cdot \mathbb{E}[Z^2]^{n/2} = \begin{cases} (n-1)!! \cdot \mathbb{V}ar{[Z]}^{n/2} & \text{ if }n\text{ is even} \\ 0& \text{ if }n\text{ is odd} \end{cases} \end{equation}

\subsection{Combinatorics of one-dimensional Hermite polynomials \texorpdfstring{$\fH_n(x)$ and $\wt H_n(x,t)$}{}} \label{sec:Hermite_1d_combinats}

To apply the Wick/Isserlis theorem to the Hermite polynomial  $\wt H_n(x,t) = \mathbb{E}_{Z\sim \cN(0,1)}[(x+  \imath \sqrt{t}Z)^n]$ (or equally well for $\fH_n(x)=\wt H_n(x,1)$ by setting $t=1$), one expands the product choosing either the term $x$ or $\imath \sqrt{t} Z$ in each of the $n$ brackets, yielding a sum over $2^n$ monomials. Then, taking the expectation over $Z$ amounts to adding up the pairings for each term as per \eqref{eq:wick}. Note that the number of pairings for each monomial depends on the number of times $\imath \sqrt{t} Z$ was chosen: there are $(m-1)!!$ such pairings with $m$ is even and no pairings at all when $m$ is odd. Therefore, only the choices where $\imath \sqrt{t} Z$ is chosen an even number of times can contribute. Figure \ref{fig:h4_combinatorics} illustrates these non-zero contributors when $n=4$.

\begin{figure}[!ht]
    \centering
    \includegraphics[width=1\linewidth]{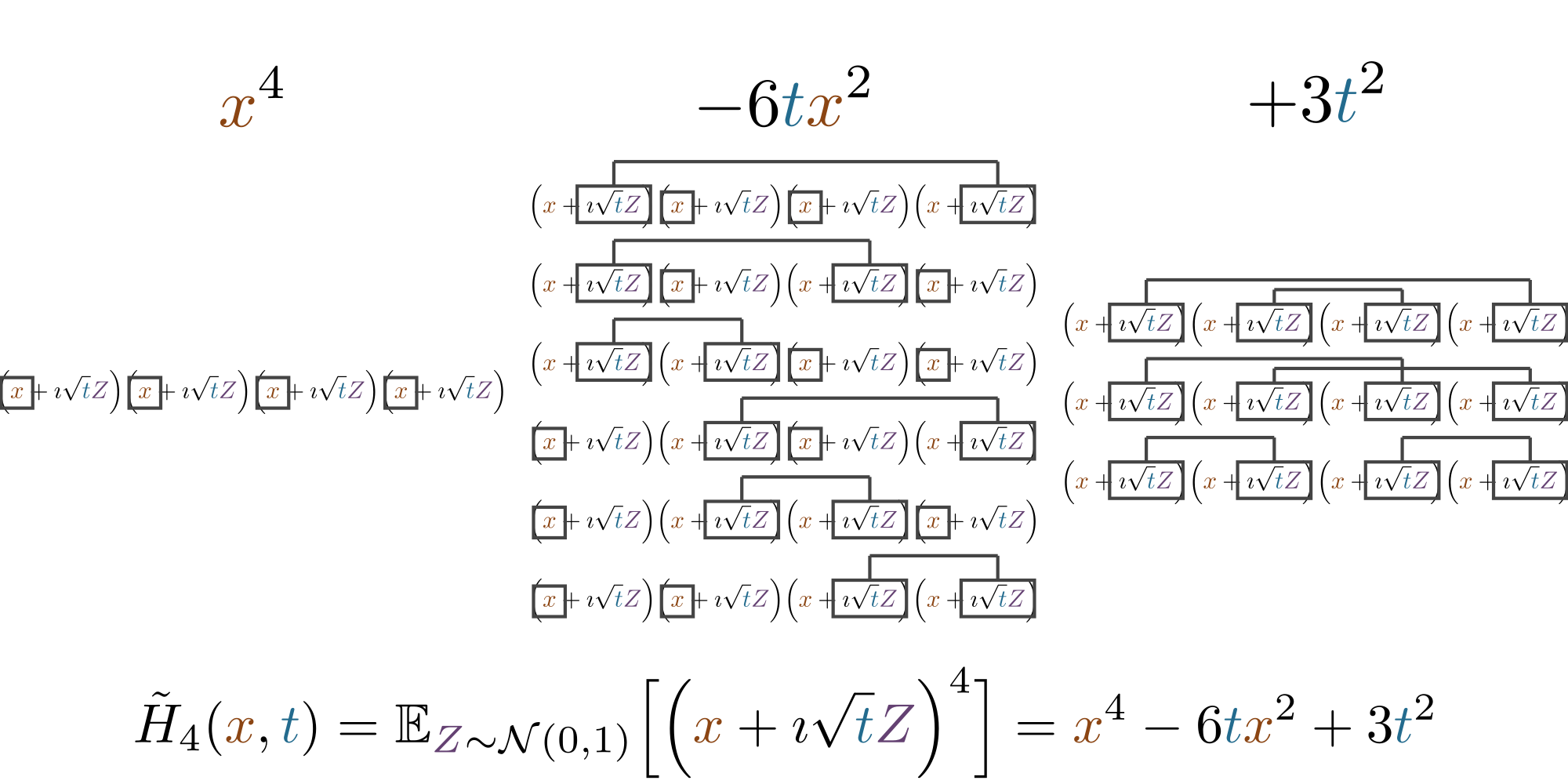}
    \caption[The combinatorics of the Hermite polynomial when $n=4$]{The combinatorics of the Hermite polynomial $\wt{H}_4(x,t)$. There are $2^4 = 16$ possible monomials in the expansion of $(x+\imath \sqrt{t} Z)^4$, but only $1+6+1=8$  monomials have an even power of $Z$, and therefore have a non-zero contribution. The monomial with $Z^4$  becomes 3 pairings due to the Wick/Isserlis theorem for the expectation of Gaussian monomials, yielding 10 total terms. Each pair contributes a multiplicative factor of $(\imath \sqrt{t})^2=-t$ .  }
    \label{fig:h4_combinatorics}
\end{figure}
Some authors refer to the diagrams generated in this way as ``incomplete pairings" or ``partial pairings'', since some of the $\imath \sqrt{t} Z$'s are paired up, but not all $n$ of them (More specifically: whenever $x$ is chosen in a bracket, the corresponding $\imath \sqrt{t} Z$  is left out). We prefer to reserve the word pairing to always mean a complete pairing of the Gaussian variables involved, and think of choosing a subset of the $x$'s as a separate operation.

When exactly $m$ of factors of $\imath \sqrt{t} Z$ that are chosen, there must be $n-m$  copies of $x$ chosen, and so the resulting contribution from each is $(\imath \sqrt{t})^m x^{n-m}$. Fortunately, since $m$ must be even, the result is actually an integer power of $t$ with a real coefficient, either $+1$  or $-1$ depending on the parity of $m/2$. Moreover there are $\binom{n}{m}(m-1)!!$ such terms, from the $m$ choices of $\imath \sqrt{t} Z$ out of the $n$ brackets and the subsequent $(m-1)!!$ ways to pair them up. This yields the combinatorial sum formula for the Hermite polynomials, where we use $S\subset\{1,\ldots,n\}$ to denote the indices where a pair was chosen:
\begin{align}\label{eq:H_sum} \wt H_n(x,t) &= \sum_{ S \subset \{1,\ldots,n\}, \pi \in \mathcal{P}( S )} (\imath \sqrt{t})^{|\pi|} x^{n-|S|} \\ &= \sum_{\substack{m=0 \\ m \text{ even}}}^n \binom{n}{m}(m-1)!! (\imath \sqrt{t})^m x^{n-m}\\ &= \sum_{k=0}^{\lfloor{n/2}\rfloor} \binom{n}{2k}(2k-1)!!(-1)^k t^k x^{n-2k} \\  &= x^n - \binom{n}{2}1!!\cdot tx^{n-2} + \binom{n}{4}3!!\cdot t^2 x^{n-4} - \ldots  \end{align}
This formula can also be derived by the binomial theorem from equations \eqref{eq:binomial} and \eqref{eq:zero}. Note that some authors rewrite this without an explicit double factorial using $\binom{n}{2k}(2k-1)!! = \frac{n!}{2^k \cdot k! \cdot (n - 2k)!}$

\subsection{Combinatorics of the multivariable Hermite polynomials \texorpdfstring{$\wt H_{\vec{n}} (\vec{x},V)$}{}}

Starting from the Gaussian integral formula $\wt H_{\vec{n}}(\vec{x};V)=\mathbb{E}_{Z \sim \mathcal{N}(0,V)}[\prod_{a=1}^d(x_a + \imath Z_a)^{n_a}] $ we can again expand out and use the Wick/Isserlis theorem. There are now $n_1$ copies of the bracket $\big(x_1 + \imath Z_1\big)$,  $n_2$  copies of the bracket $\big(x_2 + \imath Z_2 \big)$ and so on. One must choose either the variable $x$  or the random variable $\imath Z$  in each one. Then, one chooses a pairing $\pi$ of the chosen $\imath Z$  variables and computes the weight $\prod_{\{a,b\}\in \pi}\mathbb{E}[Z_aZ_b]$. By the Wick formula, summing over all these choices gives the Hermite polynomial. In symbols, this is

		\[
		\wt H_{\vec{n}}(x;V)=\sum_{\substack{C\in\{x_1,\imath Z_1\}^{n_1} \times \cdots \times\{x_d,\imath Z_d\}^{n_d}\\\pi\in\cP( \{j : C_j =\imath Z \})}} \prod_{ \{j : C_j = x_j\}} x_j \cdot  \prod_{\{a,b\}\in\pi}\bE\left[(\imath Z_{a})(\imath Z_b)\right]
		\]
where we notice that $\mathbb{E}[\imath Z_a \imath Z_b ] = - V_{ab}$. While this formula seems quite complicated at first, once understood it gives a quick way to calculate the multivariable Hermite polynomials in practice, as shown in the example below.

\begin{remark} These combinatorial expansions are intimately related to the combinatorial expansions for Wick products, see e.g. \cite{janson} Section 3.1. 
\end{remark}

\begin{example}\label{ex:4-hermite}
    Suppose $d=2$  and we have $\vec{Z}=(Z_1,Z_2)\sim \mathcal{N}(0,V)$ with covariance matrix $V = \begin{bmatrix} 1 & \varrho \\ \varrho  &1 \end{bmatrix} $, or in other words $Z_1,Z_2$ are marginally standard mean-0 variance-1 Gaussians with correlation $\mathbb{E}[Z_1 Z_2] = \varrho $. Let us use the combinatorial formula to calculate the Hermite polynomials of order $4$, $\wt H_{4,0}(x_1,x_2), \wt H_{3,1}(x_1,x_2), \wt H_{2,2}(x_1,x_2)$  in this case. Note that all of these examples follow the same combinatorial diagram for the pairings as Figure \ref{fig:h4_combinatorics}, just that the $x$'s/$\imath Z$'s are now distinguished as either $x_1\text{ vs. }x_2$ or $\imath Z_1 \text{ vs }\imath Z_2$, and the edge weights are not always $-1$ but instead are $\mathbb{E}[(\imath Z_a)(\imath Z_b)]$ where $a,b$ are the labels of the chosen $Z$ variables.

    $\mathbf{\wt H_{4,0}}$: $(4,0)$  means to take 4 copies of $(x_1 + \imath Z_1)$ and nothing else. Since $Z_1$  is marginally Gaussian, this is identical to the ordinary Hermite polynomial $\fH_4$ . We have:
    $$\wt H_{4,0}(x_1,x_2;V) = \mathbb{E}[(x_1+\imath Z_1)^4] = x_1^4 - 6x_1^2 + 3$$
    
$\mathbf{\wt H_{3,1}}:$ $(3,1)$ means to take $3$ copies of $(x_1 + \imath Z_1)$ and one copy of $(x_2+\imath Z_2)$. From these, there is one way to choose all $x$'s which yields $x_1^3 x_2$. There are 6 ways to choose two $x$'s and two $\imath Z$'s: three of them choose $x_1,x_2$ with coefficient $\mathbb{E}[(\imath Z_1)(\imath Z_1)]=-1$ and the other three choose $x_1$ twice with coefficient $\mathbb{E}[(\imath Z_1)(\imath Z_2)]=-\varrho $. Lastly, there are $3$ ways to choose all $\imath Z$'s. Every choice has coefficient $\mathbb{E}[(\imath Z_1)(\imath Z_1)]\mathbb{E}[(\imath Z_1)(\imath Z_2)]=\varrho $. Hence, all together we have:
$$\wt H_{3,1}(x_1,x_2;V) = \mathbb{E}[(x_1+\imath Z_1)^3(x_2+\imath Z_2)] = x_1^3 x_2 - 3\varrho x_1^2 - 3x_1x_2 +  3 \varrho$$

$\mathbf{\wt H_{2,2}}:$ (2,2)  means to take $2$ copies of $(x_1 + \imath Z_1)$ and $2$ copies of $(x_2+\imath Z_2)$ . There is one way to choose all $x$'s which yields $x_1^2 x_2^2$. There are 6 ways to choose two $x$'s and two $\imath Z$'s: four ways to choose $x_1,x_2$ with coefficient $\mathbb{E}[(\imath Z_1)(\imath Z_2)]=-\rho$, one way to choose $x_1$  twice with coefficient $\mathbb{E}[(\imath Z_2)(\imath Z_2)] = -1$, or one way to choose $x_2$ twice with with coefficent $\mathbb{E}[(\imath Z_2)(\imath Z_2)] = -1$. Lastly, there are $3$ ways to choose all $\imath Z$'s. One way is to pair $\imath Z_1$ with itself and $\imath Z_2$  with itself,  $\mathbb{E}[(\imath Z_1)^2]\mathbb{E}[(\imath Z_2)^2]=1$ and two ways to pair with the opposte variable giving $\mathbb{E}[(\imath Z_1)(\imath Z_2)]^2 =\varrho^2$   Hence, all together we have:
$$\wt H_{2,2}(x_1,x_2;V) = \mathbb{E}[(x_1+\imath Z_1)^2(x_2+\imath Z_2)^2] = x_1^2 x_2^2 - x_1^2-x_2^2 - 4 \varrho x_1x_2 +  2 +\varrho^2$$

\end{example}

    \subsection{Combinatorial proof of orthogonality}\label{subsec:orthog}

	In this section, we begin with the Gaussian integral formula \eqref{eq:formula} and use it to prove the orthogonality (Theorem \ref{thm:og}) and recursive (Proposition \ref{prop:recursive}) properties of the Hermite polynomials from \eqref{eq:HExpectation}.  
    
 The inner product between two Hermite polynomials $\fH_{n_{}},\fH_{m_{}}$, $n_{},m_{}\in \mathbb{N}$ is again a Gaussian expectation, where now the input to the polynomial is a standard Gaussian random variable,
$$ \intop_{-\infty}^\infty \fH_{n_{}}(x) \fH_{m_{}}(x) \frac{1}{\sqrt{2 \pi}} e^{-\half x^2} \d x = \mathbb{E}_{X\sim \mathcal{N}(0,1)}\left[ \fH_n(X) \fH_m(X)\right]$$
    By the Gaussian integral formula \eqref{eq:formula} and the tower rule for expectation, we can write this as an expectation over three independent Gaussian variables taken in any order. For notational clarity, we use $L \sim \mathcal{N}(0,1)$ for the \emph{left} Hermite polynomial and use $R \sim \mathcal{N}(0,1)$ for the \emph{right} Hermite polynomial. 
    $$\mathbb{E}_{X\sim \mathcal{N}(0,1)}\left[ \fH_{n_{}}(X) \fH_{m_{}}(X)\right] = \mathbb{E}_{X,L,R \sim \mathcal{N}(0,1)} \left[ (X+\imath L)^{n_{}} (X+\imath R)^{m_{}}\right]$$
    Once again, by expanding the brackets out and using the Wick/Isserlis theorem on $X, L, R$ separately this can be interpreted as a sum over certain pairings. To formalize this with notation we make the following definition.

    \begin{figure}[!ht]
        \centering
        \includegraphics[width=0.75\linewidth]{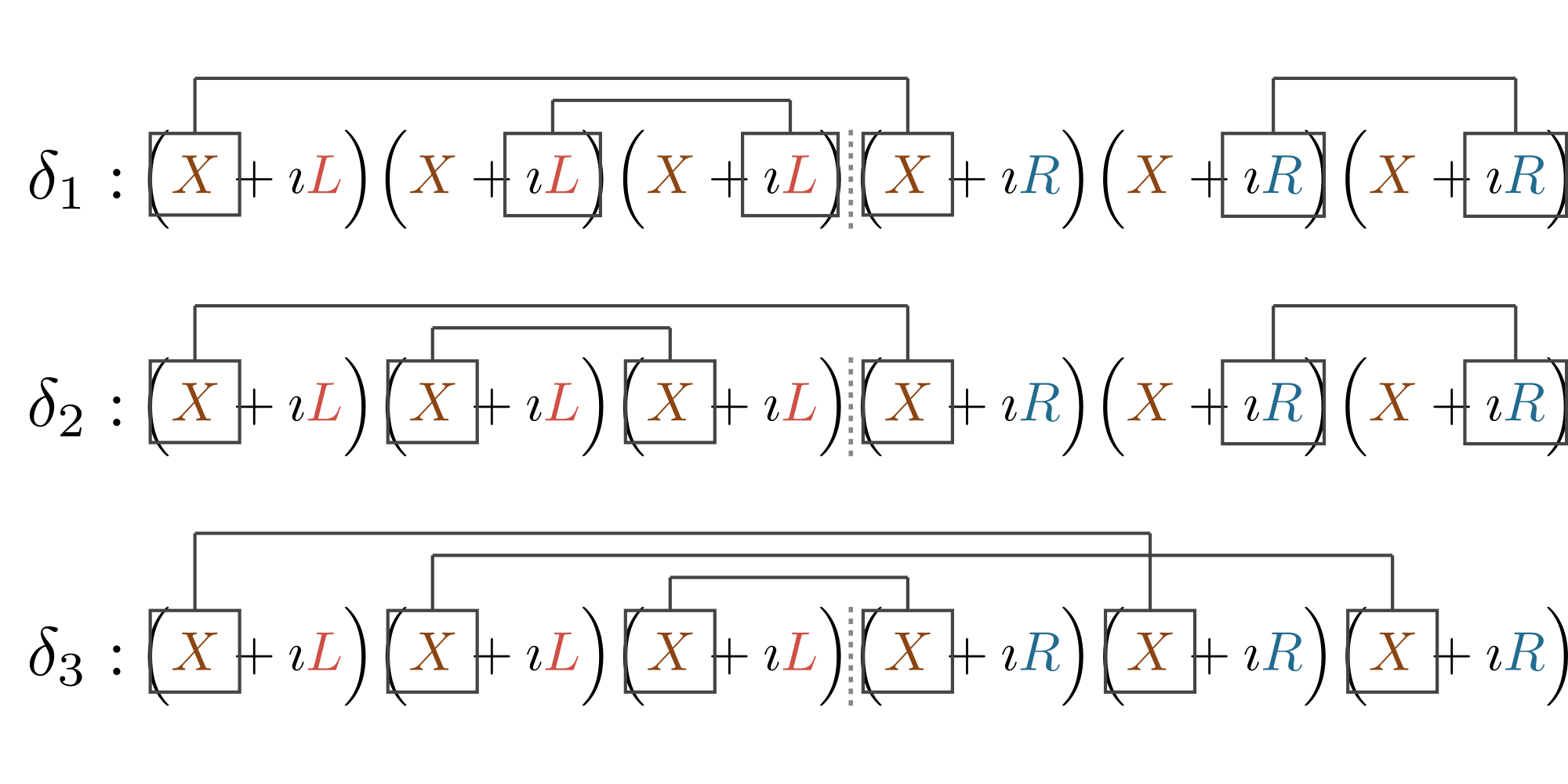}
        \caption[Three examples of diagrams in the combinatorial proof of orthogonality]{Three examples of diagrams for $\mathbb{E}[\fH_3(X)\fH_3(X) ]$  as defined in Definition \ref{def:diagrams} when $n_{} = 3$, $m_{}=3$. The product of the 6 terms in $(X+\imath L)^3(X+\imath R)^3$ is illustrated here with a dotted line for greater visual separation of the left/right halves. In this example: $\delta_1 = (C_1,\pi_1)$ has $C_1 = (X,\imath L,\imath L,X,\imath R, \imath R)$ and $\pi_1 = \{ \{1_L, 1_R\}, \{2_L,3_L\},\{2_R,3_R\}\}$ ; $\delta_2 = (C_2, \pi_2)$ where  $C_2 = (X,X,X,X,\imath R, \imath R)$ and $\pi_2 = \{ \{1_L, 1_R\}, \{2_L,3_L\},\{2_R,3_R\}\}$ ;  $\delta_3 = (C_3,\pi_3)$ has $C_3 = (X,X,X,X,X,X)$ and $\pi_3 = $  $\{ \{1_L,2_R\},\{2_L,3_R\},\{3_L,1_R\}\}$ . The involution $\tau$ defined in Definition \ref{def:involution} swaps the labels of the first entry which is entirely on one half. In this example, $\tau(\delta_1)=\delta_2$ and $\tau(\delta_2)=\delta_1$ . If there are no pairs entirely on one half, then $\tau$ does nothing, so $\tau(\delta_3)=\delta_3$. When $\tau(\delta)=\delta$, since every pair has one end on the left half and one end on the right half, the diagram can be interpreted as a bijection. For example, the permutation for $\delta_3$ is $(2, 3, 1)$ when written in one line notation. }
        \label{fig:example-diagrams}
    \end{figure}
    \begin{definition}\label{def:diagrams}
    Given $n_{}, m_{} \in \mathbb{N}$,  a \textbf{diagram} $\delta$ for the expectation $\mathbb{E}[\fH_{n_{}}(X)\fH_{m_{}}(X)]$ is a choice of variables and a pairing, $\delta = (C,\pi) $ where $C$ is a choice of variables $C \in \{X,\imath L\}^{n_{}} \times \{X, \imath R\}^{m_{}}$ and $\pi$ is a pairing of the indices $\pi \in \cP\left( \{1,\ldots,n_{}\}\dot\cup\{1,\ldots,m_{}\} \right)$ which is \textbf{compatible} with $C$ in the sense that every pair in $\pi$ connects to two \emph{alike} variables from $C$ : $$ \pi\text{ is compatible with }C \iff  C_{a} = C_b \; \: \forall \{a,b\}\in \pi.$$
In other words, every pairing from $\pi$ must be one of three possible connections from $C$: either $\imath L \leftrightarrow \imath L$ ,  $X\leftrightarrow X$, or $\imath R \leftrightarrow \imath R$.  We denote the set of such diagrams by $\Delta(n_{},m_{})$.  Figure \ref{fig:example-diagrams}
shows 3 examples of diagrams.

We define the \textbf{weight} of diagram $w(\delta) \in \mathbb{R}$ by taking the product of all the expected value of the variables in all pairs from $\pi$,
$$w(\delta) := \prod_{ \{a,b\} \in \pi } \mathbb{E}[C_a C_b]$$
Note that pairings of $\imath L$ or $\imath R$ contribute a factor of $-1$ (since $\mathbb{E}[(\imath L)(\imath L)]=\mathbb{E}[(\imath R)(\imath R)]=-1$) while pairings of $X$ contribute a $+1$ (since $\mathbb{E}[X \cdot X]=1$). So an equivalent way to write the weight is:
    \begin{align} \label{eq:weight} w(\delta) &= \prod_{\substack{\{a,b\} \in \pi \\ C_a = C_b = \imath L}} \mathbb{E}[(\imath L)(\imath L)] \cdot \prod_{\substack{\{a,b\} \in \pi \\ C_a = C_b = \imath R}} \mathbb{E}[(\imath R)(\imath R)] \cdot \prod_{\substack{\{a,b\} \in \pi \\ C_a = C_b = X}} \mathbb{E}[X\cdot X] \\ &=  (-1)^{\Big|  \{a,b \} \in \pi : C_a = C_b = \imath L  \Big| + \Big|  \{a,b \} \in \pi : C_a = C_b = \imath R  \Big| }  \end{align}
    \end{definition}
By expanding out the product and applying the Wick/Isserlis theorem with this notation, the expectation can be written as a sum over diagrams:
    \begin{equation}\label{eq:sum-of-diagrams}
\mathbb{E}_{X,L,R \sim \mathcal{N}(0,1)} \left[ (X+\imath L)^{n_{}} (X+\imath R)^{m_{}}\right] = \sum_{\delta \in \Delta(n_{}, m_{})} w(\delta).
    \end{equation}
It seems at first to be quite annoying to evaluate the expectation this way, as the space of diagrams is rather wild. However, noticing that $w(\delta)$ is $\pm 1$ valued, there is a clever way to match many of the $+1$ terms to a corresponding $-1$ term in order to cancel out almost all of the diagrams and remain with a much smaller more manageable sum. To do this we define the following involution on $\Delta(n_{},m_{})$

\begin{definition} \label{def:involution}
The \textbf{involution}  on diagrams $\tau : \Delta(n_{},m_{}) \to \Delta(n_{},m_{})$ is defined as follows. First, fix an order the set $S = S_{n_{}} \dot\cup S_{m_{}} = \{1,\ldots,n_{}\}\dot\cup\{1,\ldots,m_{}\}$ e.g.  the first $n_{}$ elements of $S_{n_{}}$ are declared to be lower than the next $m_{}$ elements of $S_{m_{}}$. Now, given a diagram $\delta = (C,\pi)$,  look at the elements in order and find the \textbf{first} element $a^\ast \in S=S_{n_{}} \dot\cup S_{m_{}}$ which is paired by $\pi$ so that both ends are in the \textbf{same half of} $S$: in other words, the first element $\{a^\ast, b^\ast\} \in \pi$ such that either  $\{a^\ast, b^\ast\} \subset  S_{n_{}}$  or $\{a^\ast, b^\ast\} \subset  S_{m_{}}$ . The involution $\tau(\delta) = (C^\prime,\pi)$  is defined by keeping the same pairing $\pi$  but \textbf{swapping the variable choice} at the indices of $C_{a^\ast}$ and $C_{b^\ast}$ from $C \in \{X,\imath L \}^{n_{}} \times \{X,\imath R\}^{m_{}}$. This means swapping $ \imath L \leftrightarrow X$ or $\imath R \leftrightarrow X$ depending on which half we are in, namely:

\begin{align} C^\prime_{a^\ast}  = C^\prime_{b^\ast} :=& \begin{cases} \imath L &\text{ if } C_{a^\ast} = \phantom{\imath}X\text{ and }{a^\ast}\in S_{n_{}} \\ X &\text{ if } C_{a^\ast} = \imath L\text{ and }{a^\ast}\in S_{n_{}} \\ \imath R &\text{ if } C_{a^\ast} = \phantom{\imath}X\text{ and }a^\ast\in S_{m_{}} \\ X &\text{ if } C_{a^\ast} = \imath R\text{ and }{a^\ast}\in S_{m_{}} \end{cases} \\
C^\prime_j :=& C_j \forall j \notin \{a^\ast,b^\ast\}
\end{align}

Figure \ref{fig:example-diagrams} shows an example of the swapping action that the involution does.
 
Note also in the case that there are \textbf{no pairs} from $\pi$  paired to the same half (i.e. all pairings of $\pi$  have one side in $S_{n_{}}$  and one side in $S_{m_{}}$), then $\{ a^\ast,b^\ast \}$ does not exist whence $C^\prime = C$ exactly and $\tau(\delta)=\delta$ leaves the diagram entirely unchanged. We say $\delta$ is a \textbf{fixed point} of $\tau$  in this case. Figure \ref{fig:example-diagrams} shows an example of this.

It is evident that $\tau$  is an involution because applying $\tau$  twice will first swap the choice of variables $C_{a^\ast},C_{b^\ast}$ and then swap them back to their original setting, so $\tau(\tau(\delta))= \delta$ for all diagrams $\delta$.

\end{definition}

\begin{remark}\label{rem:involution}
The involution $\tau$ is created in such a way to make dealing with the weights $w(\delta)$ convenient. Notice that when $\{a^\ast,b^\ast\}$ exists and $\tau(\delta)\neq\delta$, then the swapping of $\tau$ has the effect of replacing a pair of imaginary variables with $\imath$'s to a pair of $X$s or vice versa, and therefore flips the sign of the weight $w(\delta)$ from \eqref{eq:weight} . This establishes:
\begin{equation}\label{eq:sign-flip}
w(\tau(\delta)) = \begin{cases} -w(\delta)&\text{ if } \tau(\delta) \neq \delta \\ \phantom{-}w(\delta) &\text{ if } \tau(\delta)=\delta \end{cases}. \end{equation}
Notice also that in the situation where $\{a^\ast,b^\ast\}$ does not exist and $\tau(\delta)=\delta$ is a fixed point, we have by definition that $\pi$ has no pairs on the same half of $S$. Since only the variable $X$ can be paired from one half to the other half (both $\imath L$ and $\imath R$ live on exactly one half), it must be that $C_j = X$ for all $j \in \{1,\ldots,n_{}\}\dot\cup\{1,\ldots,m_{}\}$, and therefore the weight of the diagram is $+1$. This establishes:

\begin{equation} w(\delta) = 1\text{ whenever }\tau(\delta)=\delta \end{equation}
\end{remark}

With this setup of diagrams $\delta \in \Delta(n_{},m_{})$ and the involution $\tau$, we are ready to prove the orthogonality of the Hermite polynomials. 

\begin{proposition} \label{prop:orthog-1d}
  Let $n_{},m_{} \in \mathbb{N}$. Then: $$ \intop_{-\infty}^\infty \fH_{n_{}}(x) \fH_{m_{}}(x) \frac{1}{\sqrt{2 \pi}} e^{-\half x^2} \d x = \begin{cases} n! &\text{ if }n_{} = m_{} = n \\ 0 &\text{ if }n_{} \neq m_{}\end{cases} $$
\end{proposition}
\begin{proof}
As in the setup from Definition \ref{def:diagrams}, we write the integral as a sum of weights over diagrams by application of the Wick/Isserlis theorem as in \eqref{eq:sum-of-diagrams}
\begin{align}
\intop_{-\infty}^\infty \fH_{n_{}}(x) \fH_{m_{}}(x) \frac{1}{\sqrt{2 \pi}} e^{-\half x^2} \d x &= \mathbb{E}_{X\sim\mathcal{N}(0,1)}[\fH_{n_{}}(X)\fH_{m_{}}(X)] \\
&= \mathbb{E}_{X,L,R \sim \mathcal{N}(0,1)}[ (X+\imath L)^{n_{}} (X+\imath R)^{m_{}} ] \\
&= \sum_{\delta \in \Delta(n_{},m_{}) } w(\delta).
\end{align}
Now, by using the involution $\tau$  from Definition \ref{def:involution}, we can pair every diagram $\delta$  which has $\tau(\delta)\neq \delta$ with its image $\tau(\delta)$ and observe they cancel out in the sum because $w(\tau(\delta))=-w(\delta)$ for such diagrams by Remark \ref{rem:involution}!  Therefore the sum over all diagrams $\delta$ with $\phi(\delta)\neq \delta$ is $0$. (Formally, one could write $2\sum_{\{\delta : \tau(\delta)\neq \delta\}} w(\delta) = \sum_{\{\delta : \tau(\delta)\neq \delta\}} w(\delta) + \sum_{\{\delta : \tau(\delta)\neq \delta\}} w(\tau(\delta)) = 0$. )  We remain with only the fixed points of the involution $\tau$:
\begin{align}
\intop_{-\infty}^\infty \fH_{n_{}}(x) \fH_{m_{}}(x) \frac{1}{\sqrt{2 \pi}} e^{-\half x^2} \d x &=  \sum_{\substack{\delta \in \Delta(n_{},m_{}) \\ \tau(\delta)=\delta }} w(\delta) + 0\\
&=  \Big| \Big\{ \delta \in \Delta(n_{},m_{}):\tau(\delta)=\delta \Big\} \Big|,
\end{align}
where we have used the fact that $w(\delta)=1$ whenever $\phi(\delta)=\delta$  as in Remark \ref{rem:involution}.  The fixed points of the involution occur precisely when the diagram $\delta = (C,\pi)$ consists of $C_j = X \; \forall j \in  \{1,2\ldots,n_{}\}\dot\cup\{1,2\ldots,m_{}\}$ and all the pairings $\{a,b\}\in \pi$ of have one side on the left half in $a\in \{1,2,\ldots,n_{}\}$ and the other side on the right half $b\in \{1,2,\ldots,m\}$. But such a pairing can be interpreted as a bijection between $\{1,2,\ldots,n\}$ and $\{1,2,\ldots,m\}$  by simply associating the pair $\{a,b\} \in \pi$ with the mapping  $f(a)=b$ where  $f:\{1,2,\ldots,n_{}\}\to \{1,2,\ldots,m_{}\}$.  $f$  defined in this way must be a bijection because $\pi$  is a pairing with every pair having one end on each half. Thus, we can count the fixed points by counting bijections:
\begin{align}
\int_{-\infty}^\infty \fH_{n_{}}(x) \fH_{m_{}}(x) \frac{1}{\sqrt{2\pi}} e^{-\frac{1}{2} x^2} \, \d x 
&= \left| \left\{ \delta \in \Delta(n_{}, m_{}) : \tau(\delta) = \delta \right\} \right| \\
&= \left| \left\{ f : \{1, 2, \ldots, n_{}\} \to \{1, 2, \ldots, m_{}\} : f\text{ is a bijection} \right\} \right| \\
&= \begin{cases} n! &\text{ if }n_{} = m_{} = n \\ 0 &\text{ if }n_{} \neq m_{}\end{cases},
\end{align}
where we have enumerated the set of bijections as either $n!$ or $0$ to give the desired result.
\end{proof}
    The theorem and the combinatorial argument generalize to arbitrary covariances and multiple variables as follows in the Theorem below. In higher dimensions the Hermite polynomials are orthogonal to their dual, which are conjugated by the covariance matrix $V$. (This does not come up in one dimension since $V = I$!)
    \begin{theorem}
		\label{thm:og}
		Let $d \in \mathbb{N}$  be a dimension, and let $X \in \mathbb{R}^d$ be a centred Gaussian random variable with covariance matrix $V \in \mathbb{R}^{d \times d}$. Recall the Hermite polynomial and its dual given by the Gaussian integral formula indexed by integer vectors from $\bZ_{\geq0}^{d}$ in Definition \ref{def:dual-H}

\begin{align}
		H_{n_1,\ldots,n_d}\left(x_1,\ldots,x_d;V\right) & =\bE_{\vec Z\sim \mathcal{N}(0,V^{-1})}\left[\prod_{a=1}^{d}\left((V^{-1}x)_{a}+\imath Z_{a}\right)^{n_{a}}\right].\\
        \wt{H}_{m_1,\ldots,m_d}(x_1,\ldots,x_d;V) & =\mathbb{E}_{\vec{Z}\sim\mathcal{N}(0,V)}\left[ \prod_{b=1}^d (x_b + \imath Z_b)^{m_b} \right].
	\end{align}
        For any  $\vec{m},\vec{n}\in\bZ_{\geq0}^{d}$, we have the orthogonality relation
        \begin{align*}
			\bE_{X \sim \mathcal{N}(\vec{0},V)}\left[H_{\vec{n}}(X,V)\widetilde{H}_{\vec{m}}(X,V)\right]=\begin{cases}  \prod_{j=1}^{d}n_{j}! &\text{ if }\vec{n}=\vec{m} \\ 0&\text{ if }\vec{n}\neq \vec{m} \end{cases}\end{align*}
	\end{theorem}
	\begin{proof}
		We first notice that when $\vec {Y} \sim \mathcal{N}(0,V)$, then the linear transformation $V^{-1}\vec {Y} \sim \mathcal{N}(0,V^{-1})$. This allows us to rewrite the polynomial $H_{\vec{n}} (\vec{x};V)$ as
       \begin{align}
H_{\vec{n}}\left(x_1,\ldots,x_d;V\right) & = \bE_{\vec Z\sim \mathcal{N}(0,V^{-1})}\left[\left(V^{-1} \vec x+\imath \vec Z\right)^{\vec{n}}\right]\\
        &= \bE_{\vec Y\sim \mathcal{N}(0,V)}\left[\left(V^{-1} \vec x+ \imath V^{-1} \vec Y\right)^{\vec{n}}\right].
	\end{align}
  where we are using the notation $\vec x^{\vec{n}} = \prod_{a=1}^d x_a^{n_a}$ for vectors of integers $\vec n$ and of real numbers $\vec x$. This allows us to write both $H_{\vec{n}}$ and $\wt H_{\vec{n}}$ as expectations over $\mathcal{N}(0,V)$ so we have
\begin{align}
			\bE_{\vec{X}\sim \mathcal{N}(0,V)} \left[H_{\vec{n}}(X,V)\widetilde{H}_{\vec{m}}(X,V)\right] & =\bE_{\vec{X},\vec{L},\vec{R}\sim \mathcal{N}(0,V)}
			\left[\left(V^{-1} \vec X+ \imath V^{-1} \vec L\right)^{\vec{n}}\left(\vec X+ \imath  \vec R\right)^{\vec{m}}\right] \\
            & =\bE_{\vec{X},\vec{L},\vec{R}\sim \mathcal{N}(0,V)}
			\left[\left(\vec X^\#+ \imath \vec L^\# \right)^{\vec{n}}\left(\vec X+ \imath  \vec R\right)^{\vec{m}}\right] \\
    \end{align}
    where we introduce the shorthand $$\vec{v}^\# := V^{-1} \vec{v},$$ for notational convenience. Now let $n=\abs{\vec{n}}$ and $m=\abs{\vec{m}}$. We will rewrite the product to be a factor with $n$ terms of the left and $m$ on the right. Let $a_{1},...,a_{n}$ be defined as follows: $a_{1}=\ldots=a_{n_{1}}=1$, $a_{n_{1}+1}=\ldots,a_{n_{1}+n_{2}}=2$ and so on. For example, if $\vec{n}=(2,3,0,2,\ldots)$, then $a$ is the sequence 1,1,2,2,2,4,4,$\ldots$ Similarly define $b_{1},\ldots,b_{m}$ in terms of $\vec{m}.$ Then we can rewrite the product as
		\begin{align*}
			\bE\left[H_{\vec{n}}(X,V)\widetilde{H}_{\vec{m}}(X,V)\right] & =\bE
			\left[\prod_{j=1}^{d}\left(X^\#_{j}+\imath L^\#_{j}\right)^{n_{j}}\prod_{j=1}^{d}\left(X_{j}+\imath R_{j}\right)^{n_{j}}\right]\\
			& =\bE
			\left[\left(X^\#_{a_{1}}+\imath L^\#_{a_{1}}\right)\cdots\left(X^\#_{a_{n}}+\imath L^\#_{a_{n}}\right)\left(X_{b_{1}}+\imath R_{b_{1}}\right)\cdots\left(X_{b_{m}}+\imath R_{b_{m}}\right)\right]
            \end{align*}
Proceeding akin to the 1D setup of diagrams in  Definition \ref{def:diagrams}, we now define a diagram for this expectation as $\delta = (C,\pi)$  where $C \in \{X^\#_{a_1},\imath L^\#_{a_1}\} \times \cdots \times \{X^\#_{a_n},\imath L^\#_{a_n}\} \times \{X_{b_1},\imath R_{b_1} \} \times \cdots \times  \{X_{b_m},\imath R_{b_m} \}  $ is a choice of variable in each bracket and $\pi \in \cP(\{1,\ldots,n\}\dot \cup \{1,\ldots,m\})$ is a pairing which is \emph{compatible} with the choice $C$ in the sense that two ends of any pairing always pair the same \emph{letter} of pairing, either $\imath L^\# \leftrightarrow \imath L^\#$ or $\imath R \leftrightarrow \imath R$ or some combination of $X$'s, $X \leftrightarrow X$, $X^\# \leftrightarrow X$ or $X^\# \leftrightarrow X^\#$. We denote the set of such diagrams as $\Delta(\vec n,\vec m)$. We define the weight of a diagram to be the product of the expectations:
$$w(\delta) = \prod_{ \{\alpha,\beta \} \in \pi } \mathbb{E}[C_\alpha C_\beta ]$$
(Note that unlike the 1D case, because the random variables $X^\#,L^\#,X,R$ are not simply standard Gaussians, the weights are no longer $\pm 1$ valued), By the Wick/Isserlis theorem, the expectation is the weighted sum over all the diagrams 
$$\bE\left[H_{\vec{n}}(X,V)\widetilde{H}_{\vec{m}}(X,V)\right]  = \sum_{\delta \in \Delta(\vec n,\vec m)} w(\delta).$$
Proceeding analogously to the 1D involution of the involution in Definition \ref{def:involution}, we now define the involution $\tau : \Delta(\vec{n},\vec{m}) \to \Delta(\vec{n},\vec{m})$ as follows. Given a diagram $\delta = (C,\pi)$, order the set of indices from left to right, and find the lowest ordered pair $\{ \alpha^\ast, \beta^\ast \}$ that is entirely on the same half. Then define the action of $\tau$  as swapping the choice $C$ at those indices, i.e. swapping $\imath L^\#\leftrightarrow X^\#$ when both indices are on the left, and swapping $\imath R \leftrightarrow X$ when both are on the right.  If no such $\{\alpha^\ast,\beta^\ast\}$ exists (e.g. when all the pairings of $\pi$  have one side on the left and one side on the right), then declare that $\tau$ nothing and $\tau(\delta)=\delta$ 

Since $L^\#,X^\#$ and $R,X$ have the same distribution, and since we always swap a pair of variables with an $\imath$ to one without or vice versa, the effect of the swapping is always a multiplicative factor of $-1$, i.e.:
$$\mathbb{E}[(\imath L^\#)(\imath L^\#)] = (-1) \cdot \mathbb{E}[ X^\# \cdot X^\#], \: \mathbb{E}[(\imath R)(\imath R)] = (-1) \cdot \mathbb{E}[ X \cdot X] $$
Hence we observe that the effect of the involution $\tau$  is to multiply the weight by $-1$ whenever $\delta$ is not a fixed point.

\begin{equation}\label{eq:sign-flip-md}
w(\tau(\delta)) = \begin{cases} -w(\delta)&\text{ if } \tau(\delta) \neq \delta \\ \phantom{-}w(\delta) &\text{ if } \tau(\delta)=\delta \end{cases}. \end{equation}

We hence see that the sum over all diagrams $\delta$ for which $\tau(\delta)\neq \delta$ is zero, since every diagram $\delta$ can be paired with its image $\tau(\delta)$ and the resulting weights cancel out! We hence realize the sum over all diagrams as a sum over only the fixed points 

\begin{align} \bE\left[H_{\vec{n}}(X,V)\widetilde{H}_{\vec{m}}(X,V)\right]  = \sum_{\substack{\delta \in \Delta(\vec n,\vec m) \\ \tau(\delta) \neq \delta}}  w(\delta) + \sum_{\substack{\delta \in \Delta(\vec n,\vec m) \\ \tau(\delta) = \delta}} w(\delta) = 0 + \sum_{\substack{\delta \in \Delta(\vec n,\vec m) \\ \tau(\delta)=\delta}} w(\delta). \end{align}
It remains to enumerate the effect of the diagrams which are fixed points $\tau(\delta)=\delta$. The key observation is that if$\delta=(C,\pi)$ has $\tau(\delta)=\delta$, then every pairing $\pi$  has one side on the left, with endpoint in $\{1,\ldots,n\}$  and one side on the right, with endpoint in $\{1,\ldots,m\}$ .  If $n\neq m$,  there are no such pairings: by the pigeonhole principle it is impossible to make such a pairing. This already shows that $\bE\left[H_{\vec{n}}(X,V)\widetilde{H}_{\vec{m}}(X,V)\right]  = 0$ when $\abs{\vec{n}}\neq\abs{\vec{m}}$.

When $n=m$ , these pairings can be interpreted as bijections $\sigma \in S_n$ by the $\sigma(a)=b$  whenever $\{a,b\}\in \pi$ and $a \in \{1,\ldots,n\}$   and $b\in \{1,\ldots,m\}$ .  The choice of the variables $C$ is also forced to be all $X^\#$  on the left and all $X$ on the right (since $\pi$ must always connect letters of the same letter). We hence have from the definition of the weight $w(\delta)$, in this case $|\vec{n}|=|\vec{m}|$ that
\begin{align*}
			\bE\left[H_{\vec{n}}(X,V)\widetilde{H}_{\vec{m}}(X,V)\right] & =\sum_{\sigma\in S_{n}}\prod_{j=1}^{n}\bE\left[X^\#_{a_{j}}X_{b_{\sigma(j)}}\right] \\ &=\sum_{\sigma\in S_{n}}\prod_{j=1}^{n}\bE\left[ (V^{-1}X)_{a_{j}}X_{b_{\sigma(j)}}\right] \\
			& =\sum_{\sigma\in S_{n}}\prod_{j=1}^{n}\left(\bE\left[X\left(V^{-1}X\right)^{T}\right]\right)_{a_{j},b_{\sigma(j)}}.
		\end{align*} 
        Now $V$ is a covariance matrix and hence symmetric. Therefore so
		is $V^{-1}$and we get $X(V^{-1}X)^{T}=XX^{T}V^{-1}.$ Since $X \sim \mathcal{N}(0,V)$, by linearity
		of expectation, we see that the expected value of this is the identity matrix,
		\begin{align*}
			\bE\left[X\left(V^{-1}X\right)^{T}\right] & =\bE\left[XX^{T}\right]V^{-1}=VV^{-1}=I.\\
		\end{align*}
		In other words, $\bE\left[X_{a}\left(V^{-1}X\right)_{b}\right]=\mathbf{1}\{a=b\}$and so
		\begin{align*}
			\bE\left[H_{\vec{n}}(X,V)\widetilde{H}_{\vec{m}}(X,V)\right] & =\sum_{\sigma\in S_{n}}\prod_{j=1}^{n}\mathbf{1}\{{a_{j}=b_{\sigma(j)}}\}.
		\end{align*}
Since the $a_{j}$and $b_{j}$ are ordered increasingly, the only way all the terms are non-zero is if
if $\vec{m}=\vec{n}.$ If so, then for each $n_{j}$, there are $n_{j}!$ ways of mapping each, which gives the result.
	\end{proof}

	\subsection{Combinatorial proof of the recurrence relation}

    The recurrence relation for the Hermite polynomials can be realized as a consequence of the combinatorial sums.  This also generalizes to the multidimensional Hermite polynomials.
    \begin{proposition}
		\label{prop:recursive}
        For any $n\in \mathbb{N}$, the Hermite polynomials satisfy the recurrence
        $$\fH_{n+1}(x) = x \fH_n(x) - n \fH_{n-1}(x)$$
        \end{proposition}
        \begin{proof} Recall the interpretation of the Hermite polynomial $\fH_{n+1}$ as summing over choices of variables $\{x,\imath Z\}^{n+1}$ and pairings on on the set of $\imath Z$ choices. The set of pairings can be divided into two disjoint sets: Either the the variable $x$ is chosen in the last bracket, or the variable $\imath Z$  is chosen in the last bracket. 
        
        When $x$ is chosen in the last bracket, we have a factor of $x$ and the pairings of the first $n$ brackets is simply $\fH_n(x)$. This case yields $x \fH_n(x)$. 
        
        When $\imath Z$ is chosen in the last bracket, there are $n$  ways to choose which other entry this $\imath Z$ will be paired with. For each choice, with this last bracket and its partner pair chosen, there are $n-1$ brackets remaining, thus giving a contribution of $\fH_{n-1}(x)$. All together the contribution from the $n$ ways is $n \mathbb{E}[(\imath Z)(\imath Z)] \fH_{n-1}(x)=-n\fH_{n-1}(x)$.
        
        Adding up the total contribution from these two disjoint cases gives the desired result.
        \end{proof}

        \begin{proposition}
            Let $d \in \mathbb{N}$. For any $\vec{n}\in\bN^d$, $j\in \{1,\ldots,d\}$ and $\vec{x}\in\bR^d$, the multidimensional Hermite polynomials satisfy
		\begin{align*}
			\wt H_{\vec{n}+\vec{e}_{j}}(\vec x;V)=x_{j}\wt H_{\vec{n}}(\vec x;V)- \sum_{a=1}^{d} n_a V_{j a} \wt H_{\vec{n}-\vec{e}_a}(\vec x;V),
		\end{align*}
        where $\vec{e}_k = (0,0,\ldots,1,\ldots,0)$ is the basis vector with a $1$ in the $k$-th position.  
    \end{proposition}
	\begin{proof}

        The identity for the multidimensional follows by the same argument as the proof of Proposition \ref{prop:recursive}, where we now consider the choices for the variable in a particular  bracket $(x_j + \imath Z_j)$ for the choice of index $j$.
        
        If $x_j$  is chosen as the variable in the bracket, we get a contribution of $x_j \wt H_{\vec{n}}(\vec{x};V)$ since all the other brackets yield $H_{\vec{n}}(\vec{x};V)$. 
        
        If $\imath Z_j$  is chosen as the variable in the bracket, then there are $|\vec{n}|=n_1+n_2+\ldots+n_d$ choices for who the pair will be, namely $n_1$ copies of $\imath Z_1$, $n_2$ copies of $\imath Z_2$ and so on. When $\imath Z_a$  is chosen, the coefficient from that pair is $\mathbb{E}[(\imath Z_j)(\imath Z_a)]=V_{ja}$. The rest of the brackets are $\wt H_{\vec{n}-\vec{e}_a}(\vec{x};V)$  in this case since we have ``used up'' a bracket of the form $(x_a + \imath Z_a)$. 
        
        Adding up over all the choices gives the desired result.
        
	\end{proof}
 \begin{remark}
See also Lemma \ref{lem:recurrence-with-integration-by-parts} where we prove the recurrence relation directly from the Gaussian integral formula using the Gaussian integration by parts identity $\mathbb{E}[Zg(Z)] = \mathbb{E}[g^\prime(Z)]$.
    \end{remark}

\section{Asymptotics for the Hermite polynomials\texorpdfstring{as $n \to \infty$}{}}
\label{sec:asymptotics-hermite}
	
	The Gaussian integral formula \eqref{eq:HExpectation} also yields an easy entry point to the large $n$ limit of the Hermite polynomials because it is naturally set up for the Laplace approximation method for integrals. Other existing approaches to the Hermite  asymptotics instead use the derivative formula and then must go through the Cauchy integral formula and residue analysis to the turn the derivative into a complex integral e.g. Section 3 of \cite{AGZ} and \cite[Section 8]{Szego1975} . For this reason, unlike the typical statement,  the asymptotic formulas we derive here does not involve the factorial $n!$ (which appears due to the Cauchy integral formula), thereby giving ``clean'' asymptotics without the need for Stirling's formula. 
    
    We will use in this section the notation $f(n)=g(n)(1+o(1))$ to mean that $\lim_{n\to\infty}\frac{f(n)}{g(n)}=1$. When $f,g$ also depend on other variables, we will state our theorems as holding pointwise with those variables fixed while $n$ grows to infinity.

    \begin{figure}[!ht]
    \centering
    \includegraphics[width=\linewidth]{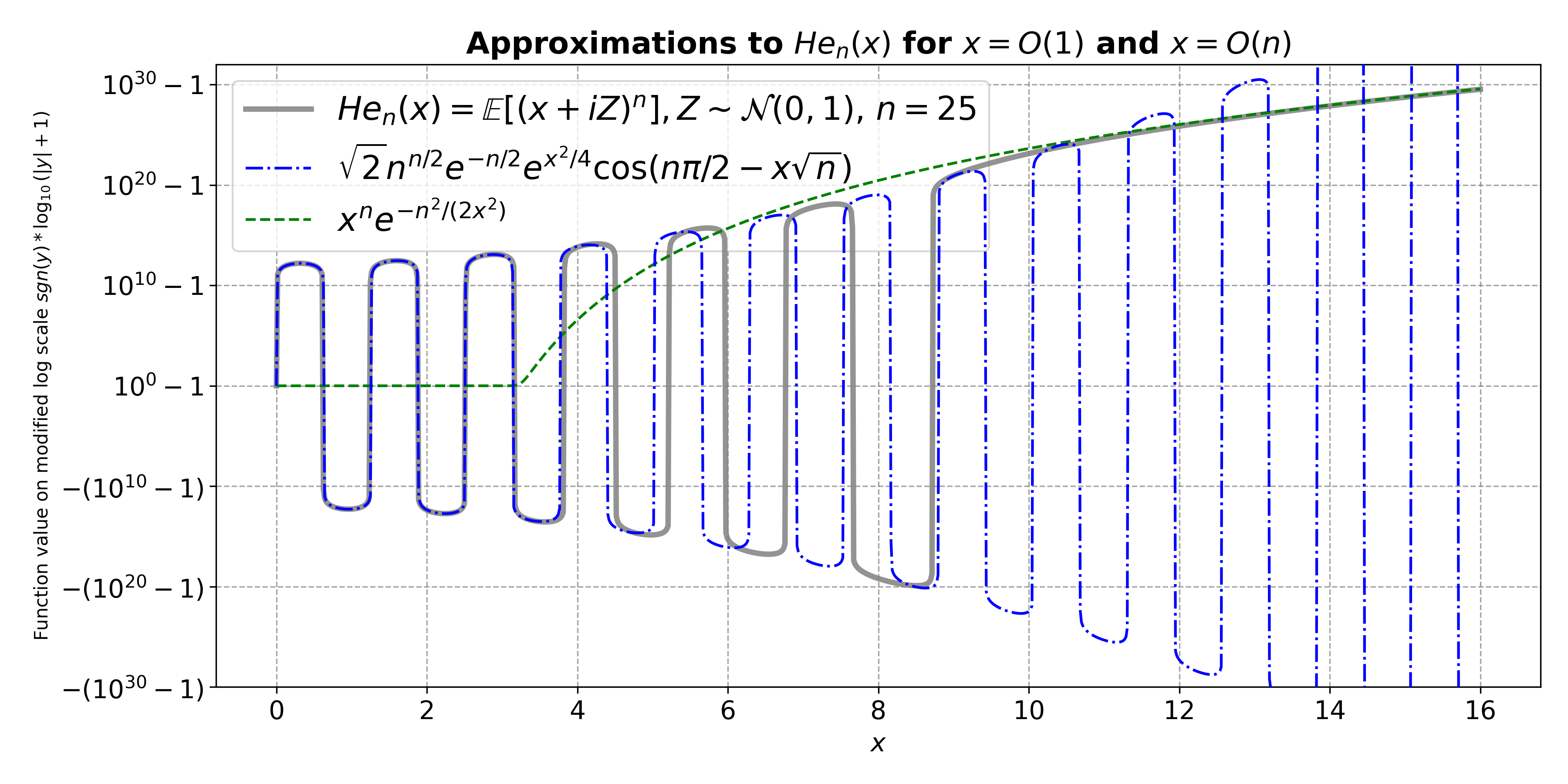}
    \includegraphics[width=\linewidth]{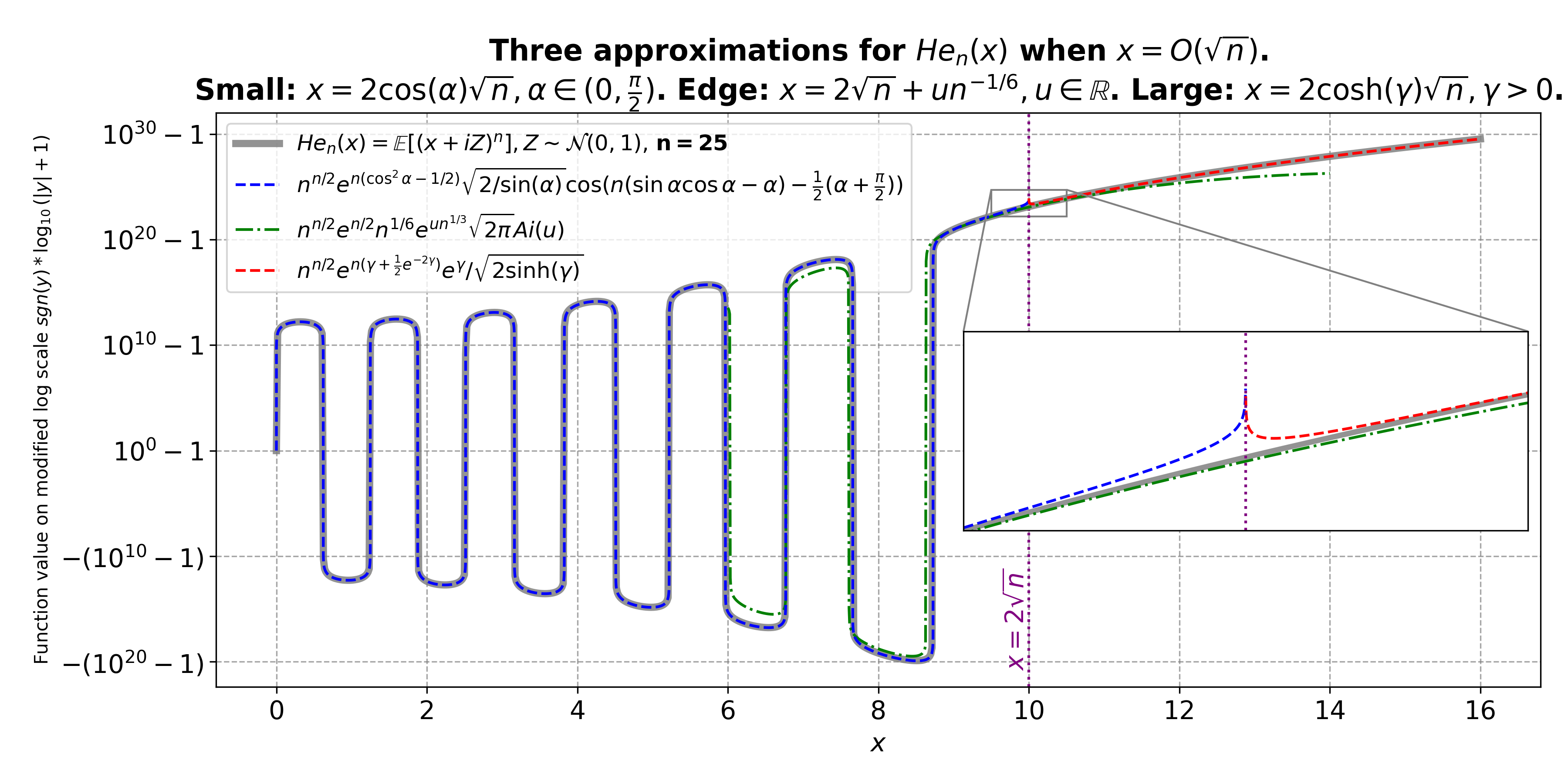}
    \vspace{-3em}
    \caption[Asymptotic approximations to the Hermite polynomial $\fH_n(x)$ when $x=O(1)$, $x=O(n)$ and three formulas when $x=O(\sqrt{n})$]{The approximation formulas for various regimes from  Propositions \ref{prop:unscaled}, \ref{prop:superlarge} and \ref{thm:main-asymptotics} compared to $\fH_n(x)$ in the case $n=25$. The $y$-axis is a modified log scale, so that the sign of each $y$ value is preserved and $0$ is mapped to $0$; the map is $y \to sgn(y) \log_{10}(|y|+1)$. The relatively simple asymptotics when $x=O(1)$ and $x=O(n)$ are only accurate when $x \ll \sqrt{n}$ and $x \gg \sqrt{n}$ (and in fact can also be recovered as limits of the more complicated $x=O(\sqrt{n})$ asymptotics). The $x=O(\sqrt{n})$ asymptotics are divided into three regimes; $x=2\sqrt{n}\cos(\alpha) < 2\sqrt{n}$ for small $x$ values, $x=2\sqrt{n}\cosh(\gamma) > 2\sqrt{n}$ for large $x$ values, and $x=2\sqrt{n} + n^{-1/6}u$ for $x$ values on the edge near $x=2\sqrt{n}$. Notice from the zoomed insert that the small-x and large-x approximations are both singular at the edge $x=2\sqrt{n}$, and this is the regime where the edge approximation with the Airy function is accurate. }	
    \label{fig:Hermite_approx}
    \end{figure}

\subsection{Approximation to \texorpdfstring{$\fH_n(x)$}{H} when \texorpdfstring{$x=O(n)$}{for large values}}
    
\begin{proposition}
\label{prop:superlarge}
Let $c \in \bR$ be fixed, and set $x=cn$. We have the following asymptotic formula as $n\to \infty$:
\begin{align*}
				\fH_n(cn)&=(cn)^ne^{-\frac{1}{2c^2}}(1+o(1)).
			\end{align*}
In other words, $\fH_n(x) \approx x^n e^{-\frac{n^2}{2x^2}}$  for $x=O(n)$.
\end{proposition}
\begin{proof}
Write by our integral formula that $$\fH_n(cn) = \mathbb{E}[(cn+iZ)^n] = (cn)^n \mathbb{E}\left[\rb{1+i\frac{Z}{cn}}^n\right].$$Now observe the limit  $\lim_{n\to\infty}(1+i\frac{Z}{cn})^n = e^{iZ/c}$ almost surely in  $Z$ (see e.g. Theorem 3.4.2 in \cite{durrettPTE}), and in light of the expectation $\mathbb{E}[e^{\beta Z}]=e^{\half \beta^2}$ we find that $\mathbb{E}[(1+i\frac{Z}{cn})^n]=e^{-\frac{1}{2c^2}}(1+o(1))$, whence the result follows. 
    
    \subsection{Tool: Laplace approximation for integrals}

All the other approximations in this section rely on the Laplace approximation method for the integrals of the form $\intop e^{nf(x)} g(x) \d x$ .  Below we state one version of this method and discuss some ideas surrounding it.

\begin{proposition}\label{prop:laplace} [Laplace Approximation] Suppose $g(x)$ is a continuous function, and suppose $f(x)$ is a function with a unique global maximum at $x^\ast$ with zero-derivative there, $f'(x^\ast)=0$, and strictly negative second derivative there $f''(x^\ast) < 0$. Suppose also $x^\ast \in (a,b)$ is some interval containing $x^\ast$ as an interior point. Then, under certain regularity conditions on $f$ (see e.g. Theorem 3.5.3 from \cite{AGZ} for a precise statement), we have the following Laplace approximation of the integral as $n\to \infty$,
\begin{equation}\label{eq:laplace_approx_eq} \intop_{a}^b e^{n f(x)} g(x) \d x  = e^{n f(x^\ast) } g(x^\ast) \sqrt{\frac{2\pi }{n |f''(x^\ast)|}} (1+o(1)).\end{equation}
    \end{proposition}
    \begin{remark}
The proof idea of the Laplace approximation can be more powerful than the mere statement of the result, so we briefly elucidate it here.  The idea is to approximate $f$ around its global maximum by Taylor's theorem $f(x) = f(x^\ast) - \frac{1}{2} |f''(x^\ast)|(x-x^\ast)^2 + \epsilon_2 \rb{x-x^\ast}$, and then do a change of variables $\frac{s}{\sqrt{n}}= x-x^\ast $, so that for any constants $c_a < 0 < c_b$ one has  \begin{equation}\label{eq:laplace_sketch}\intop_{x^\ast+c_a/\sqrt{n}}^{x^\ast+c_b/\sqrt{n}} e^{nf(x)} \d x = e^{nf(x^\ast)}\intop_{c_a}^{c_b} e^{ - \half |f''(x^\ast)|s^2 + n \epsilon_2\rb{\frac{s}{\sqrt{n}}}} \frac{\d s}{\sqrt{n}}.\end{equation}
Then, by using the technical regularity conditions of the function $f$, one argues that Taylor error term $n \epsilon_2(\frac{s}{\sqrt{n}})$ goes to zero fast enough to be negligible (For example, if $f$ has bounded third derivative, one has $\epsilon_2(y)=O(y^3)$ and $n \epsilon_2(\frac{s}{\sqrt{n}})$ goes to zero like $\frac{1}{\sqrt{n}}$ ) , whence one remains with the Gaussian integral $\intop_{c_a}^{c_b} e^{ -\frac{1}{2} |f''(x^\ast)| s^2} \d s$. As $c_a \to -\infty$ and $c_b \to +\infty$ , this converges to $\sqrt{ \frac{2\pi}{|f''(x^\ast)|}}$ which is the term that appears on the RHS of \eqref{eq:laplace_approx_eq}. To complete the argument, one also needs to argue that the integral outside of this neighbourhood of $x^\ast$ is exponentially smaller than $e^{n f(x^\ast)}$ and use continuity of $g$ to bring $g(x)$ along for the ride. (see e.g. Theorem 3.5.3 from \cite{AGZ} where these details are carried out and a rigourous proof is given.)

This proof sketch illustrates a few important concepts:
\begin{enumerate}
    \item More precise error estimates can be obtained by being more careful with the Taylor series expansion of $f$ or by expanding to higher order.  We limit ourselves to proving to $1+o(1)$ for ease of readability here. 
    \item The dominant contribution comes from a window of on the scale $\frac{1}{\sqrt{n}}$ near the maximizer $x^\ast$. Keeping this in mind allows one to apply it to more complicated domains of integration. For example, if the endpoint of the interval of integration coincides with the maximizer, $a=x^\ast$, the approximation will still apply with a factor of $\frac{1}{2}$. It also applies for contour integrals on curves in $\mathbb{C}$ in the same way by zooming in around the maximizer $x^\ast$ because every curve is approximately a tangent line when zoomed in. We will see precisely this idea of applying the Laplace approximation to contour integrals going through a local maximum in Section \ref{subsec:three-scaling}.  
    \item Since the dominant contribution happens near the maximizer $x^\ast$, one can also give limits that give the asymptotic behaviour on windows that zoom in around $x^\ast$. For example, by changing variables $z = \sqrt{|f''(x^\ast)|}s$ in the integral of \eqref{eq:laplace_sketch}, one sees that for $c_a < 0 < c_b$ one has the limit
\begin{equation}\label{eq:zoomed_in_laplace_approx_eq} \intop_{x^\ast + c_a/\sqrt{n}}^{x^\ast + c_b/\sqrt{n}} e^{n f(x)} g(x) \d x  = e^{n f(x^\ast) } g(x^\ast) \sqrt{\frac{2\pi}{n |f''(x^\ast)|}} F_\Phi \Bigg(c_a\sqrt{|f''(x^\ast)|},c_b\sqrt{|f''(x^\ast)|}\Bigg)   (1+o(1)),\end{equation} where $F_\Phi(u,v) = \frac{1}{\sqrt{2\pi}}\intop_{u}^{v} e^{-\frac{1}{2}z^2}\d z = \mathbb{P}_{Z\sim\mathcal{N}(0,1)}(u < Z < v)$ is the cumulative probability function for the standard Gaussian random variable.
    \item For maximizers $x^\ast$ which have $f''(x^\ast)=0$ (aka a double critical point), the Laplace approximation idea can still be used by now expanding to 3rd order in the Taylor series expansion $f(x) = f(x^\ast) - \frac{1}{6} f'''(x^\ast)(x-x^\ast)^3 + \epsilon_3 \rb{x-x^\ast}$ and making a change of variable $\frac{s}{n^{1/3}} = x-x^\ast$ instead. This is precisely what will occur in Section \ref{subsec:edge-scaling} and explains the factors of $n^{1/3}$ that appear in these edge scalings.
\end{enumerate} 

\end{remark}    
\begin{example}\label{ex:stirling}
    The Laplace expansion can be used to derive a version of Stirling's formula for approximating $n!$. We begin by writing the $n!$ as the Gamma function integral    $$n! = \intop_0^\infty x^n e^{-x} \d x = n^{n+1} \intop_0^\infty y^n e^{-ny} \d y = n^{n+1} \intop_0^\infty e^{n (\ln y - y)} \d y, $$
where we have used a change of variables $x = n y$ so that the maximizer to the integral occurs at the fixed value $y^\ast = 1$ (as opposed to the maximum $x^\ast =n$  growing with $n$). The function $f(x) = \ln y - y$ satisfies our criteria with the maximum at $y^\ast = 1$, $f(y^\ast)=-1$, $f''(y^\ast)=-1$ and so we obtain:
    $$n! = n^{n+1} \intop_0^\infty e^{n (\ln y - y)} \d y = n^{n+1} e^{n (-1)} \sqrt{\frac{2\pi}{n \cdot  1}} (1+o(1)) = \sqrt{2\pi n}\left(\frac{n}{e}\right)^n (1+o(1))   $$\end{example}

\begin{remark}
\label{rmk:Stirling-bound}
    While we do not prove this here, the formula given for the limiting behaviour of $n!$ is in fact a lower bound. It will be helpful for us to use the following upper and lower bounds that follow from \cite{Robbins-Stirling}: for any $n\in\bN$,
    \begin{align*}
        \sqrt{2\pi n} \rb{\frac n e}^n < n!\leq e \sqrt{n}\rb{\frac n e}^n .
    \end{align*}
\end{remark}

\subsection{Approximation to \texorpdfstring{$\fH_n(x)$}{H} \texorpdfstring{when $x=O(1)$}{near 1}}

\begin{proposition}
\label{prop:unscaled}
For $x \in \bR$ fixed, we have the following asymptotics as $n\to \infty$
			\begin{align*}
				\fH_n(x)&=\sqrt2 \, \rb{\frac ne}^{n/2}\,  e^{x^2/4} \cos\rb{\frac{n\pi}{2} -x\sqrt{n} }\rb{1+o\rb{1}}.
			\end{align*}
\end{proposition}
\end{proof}

\begin{proof}
Starting with the integral formula \eqref{eq:HExpectation}, we proceed as follows:
\begin{align*}
\fH_{n}(x)& =\frac{1}{\sqrt{2\pi}} \intop_{-\infty}^{\infty}(x+it)^{n}e^{-t^{2}/2}\d t =\frac{2}{\sqrt{2\pi}}\Re\left\{ \intop_{0}^{\infty}(x+it)^{n}e^{-t^{2}/2}\d t\right\} \\
 & =\frac{2}{\sqrt{2\pi}}\Re\left\{ \intop_{\Re(z)=x,\Im(z)>0}e^{n\ln(z)+\half(z-x)^{2}}\frac{\d z}{i}\right\} \\ & =\frac{2}{\sqrt{2\pi}}\Re\left\{ \intop_{\Re(z)=0,\Im(z)>0}e^{n\ln(z)+\half(z-x)^{2}}\frac{\d z}{i}\right\}, 
\end{align*}
where we have thought of this as a line integral in $\mathbb{C}$ and moved the contour over to the positive imaginary axis  $\Re(z)=0$. This can be justified by Cauchy's theorem and the exponential decay of the integrand near infinity. Note that the contribution to the integral is dominated by the locations $z$
where $n\ln(z)+\half(z-x)^{2}$ is maximum along this contour. Taking a derivative, we see the maximum occurs at $z^\ast=i\sqrt{n}+o(\sqrt{n})$, which gives us the idea to change variables by a factor of $\sqrt{n}$ so that this location of maximum contribution is moved to an O(1) location as $n$ grows. Specifically, let us change do a change of variable to a new variable $s$ so that:
\[
z=i\sqrt{n}s,\quad \frac{\d z}{i}=\sqrt{n}\d s, \quad \ln(z)=\ln(\sqrt{n})+\ln(s)+i\frac{\pi}{2}
\]
 and expanding out gives
\[
\frac{1}{2}(z-x)^{2}=\half(i\sqrt{n}s-x)^{2}=-\half ns^{2}+\half x^{2}-i\sqrt{n}sx
\]
so we remain with:
\begin{align*}
\fH_n(x) & =\frac{2}{\sqrt{2\pi}}\Re\left\{ \intop_{0}^{\infty}e^{n\ln\sqrt{n}}e^{n\ln s}e^{ni\frac{\pi}{2}}e^{-\half ns^{2}}e^{\half x^{2}}e^{-i\sqrt{n}sx}\sqrt{n}\d s\right\} \\
 & =\frac{2}{\sqrt{2\pi}} \sqrt{n}e^{n\ln\sqrt{n}}e^{\half x^{2}}\intop_{0}^{\infty}e^{n\left(\ln s-\half s^{2}\right)}\cos\left(n\frac{\pi}{2}-\sqrt{n}sx\right)\d s
\end{align*}
Now by the Laplace approximation \ref{prop:laplace} , the contribution to the integral is coming from the maximizer of $f(s)=\ln s-\half s^{2,}$ which has $f'(s)=\frac{1}{s}-s$ and $f''(s)=-\frac{1}{s^{2}}-1$, and a global maximum at $s^\ast=1$ with $f(1)=-\half$, $f''(1)=-2$, yielding:
\begin{align*}
\fH_n(x) & = \frac{2}{\sqrt{2\pi}} \sqrt{n}e^{n\ln\sqrt{n}}e^{\half x^{2}}\left(e^{nf(1)}\sqrt{\frac{2\pi}{n\abs{f^{''}(1)}}}\right)\cos\left(n\frac{\pi}{2}-\sqrt{n}\cdot1\cdot x\right)(1+o(1))\\
 & =\frac{2}{\sqrt{2\pi}} \sqrt{n} e^{n\ln\sqrt{n}}e^{\half x^{2}}\left(e^{-n\half}\sqrt{\frac{2\pi}{2 n}}\right)\cos\left(n\frac{\pi}{2}-\sqrt{n} x \right)(1+o(1)),
\end{align*}
from which the result follows after simplifying.
\end{proof}

    \subsection{Three approximations to \texorpdfstring{$\fH_n(x)$}{H} \texorpdfstring{when $x=O(\sqrt{n})$}{at square root}}\label{subsec:three-scaling}

It is also natural to obtain asymptotics for $\fH_n(x)$ when $x$ grows proportionally to  $\sqrt{n}$, i.e. when $x = 2a \sqrt{n}$ for some constant $a$.  The results here are equivalent to the so-called Plancherel--Rotach asymptotics for the Hermite polynomials (see e.g. Theorem 8.22.9 of \cite{Szego1975}). However, our results are stated and proven without any reference to $n!$ (which arises from the Cauchy integral formula in the traditional Plancherel--Rotach setup). 

There are three asymptotic regimes depending on the value of  $\lim_{n\to \infty} x/\sqrt{n}$, each of which has qualitatively different behaviour, see Figure \ref{fig:Hermite_approx} for an illustration.
\begin{itemize}
    \item \textbf{Small $x$ values $-2\sqrt{n} < x < 2\sqrt{n}$:} In this regime, the Hermite polynomials are oscillating, and a similar qualitative behaviour to the $x=O(1)$ of Proposition \ref{prop:unscaled} with modifications to the phase and amplitude.  The Hermite polynomials in this region are also intimately connected to the arcsine law and the semi-circle law, see Subsection \ref{subsec:arcsine}  and \ref{subsec:semicircle-limit}.
    \item \textbf{Large $x$ values $|x| > 2\sqrt{n}$}: In this regime the polynomial has no more oscillations and is instead monotone increasing. As $x$ grows larger and larger, this approximation approaches the approximation of Proposition \ref{prop:superlarge}.
    \item \textbf{Values near the edge $x=\pm 2\sqrt{n}+O(n^{-1/6})$:} In this regime right on the edge between the small and large behaviours, the asymptotics are approximated using the Airy function. This interpolates between the oscillatory small-$x$ regime and monotone large-$x$ regimes (corresponding to the oscillations of the Airy function $\A(x)$  for $x<0$ and monotone behaviour for $x>0$). This regime is important to the KPZ universality class and gives rise to the so-called Airy kernel, for example controlling fluctuations of the largest eigenvalue of certain random matrices. See Section \ref{sec:Dyson}. 

\end{itemize}

Mathematically, these three regimes arise from a certain quadratic equation in our integral formula  having either two distinct real roots (the small $x$ regime), a single repeated root (the values near the edge regime) or two complex conjugate roots (the large $x$ regime). See Figure \ref{fig:constant_phase_curves} for how these roots appear in the 3 regimes.

While we state our result for $x>0$, the analogous results for $x<0$ can be determined by the even/oddness of the Hermite polynomials,  $\fH_n(-x)=(-1)^n\fH_n(x)$.

\begin{theorem} \label{thm:main-asymptotics}
The Hermite polynomials $\fH(x)$ have the following three approximations when $x=O(\sqrt{n})$ and $x>0$ 
\begin{itemize}[label={}, leftmargin=*]
\item \textbf{1. Small-$x$ regime $\mathbf{0 < x < 2\sqrt{n}}$ :} For fixed $a\in (0,1)$, set $x=2a\sqrt{n}$ and we obtain:
\begin{equation}
\fH_n\rb{2a\sqrt n} = n^{n/2}\,
				{\frac{\sqrt{2}}{({1-a^2})^{1/4}}}
				e^{n(a^2-1/2)}\,
				\cos\rb{n \rb{a\sqrt{1-a^2} -\arccos(a)} -\frac 12\rb{\arccos(a)+\frac\pi2}}\rb{1+o(1)},
\end{equation}
with the convention $\arccos(a)\in (0,\pi)$. Equivalently, writing $a=\cos(\alpha)$ with $\alpha\in (0,\pi/2)$, we have
			\begin{align}\label{eq:small-x-approx}
				\fH_n\rb{2\cos(\alpha)\sqrt n} &= n^{n/2}
				\sqrt{\frac{2}{\sin(\alpha)}}
				e^{n(\cos^2\alpha-1/2)}\,
				\cos\rb{n \rb{\sin\alpha\cos\alpha -\alpha} -\frac 12\rb{\alpha+\frac\pi2}}\rb{1+o(1)}
			\end{align} See Remark \ref{rem:small-x} for an interpretation of the oscillating factor here.
\item \textbf{2. Large-$x$ regime $\bm{x > 2\sqrt{n}}$:} For fixed $a>1$, write $a=\cosh(\gamma)$ with $\gamma\geq 0$ and set $x=2a\sqrt{n}$. We obtain:
			\begin{align}\label{eq:large-x-approx}
				\fH_n \rb{2a\sqrt n }  = \fH_n( 2\cosh(\gamma)\sqrt{n} )
				& = n^{n/2}e^{n (\gamma+\frac12 e^{-2\gamma})}\, \frac{e^{\gamma/2}}{\sqrt{2\sinh(\gamma)}}
				\rb{1+o(1)}. 
			\end{align}
\item \textbf{3. Edge regime} $\mathbf{x=2\sqrt{n}+O(n^{-1/6})}$: For fixed $u\in\bR$ , set $x = 2\sqrt{n} + n^{-1/6}u$  and we obtain:
			\begin{align}
				\fH_n\rb{2\sqrt n+n^{-1/6}u}
				& = \sqrt{2\pi} e^{n/2} n^{n/2} n^{1/6} e^{un^{1/3}} \A(u) \rb{1+o\rb{1} }
                \end{align}
\end{itemize}

                \end{theorem}
                \begin{proof}
                All three results begin in the same way, by first writing $\fH_n$ as the Gaussian integral:
                \begin{align*}
			\fH_n\rb{2a\sqrt n} &= \mathbb{E}_{Z\sim \cN(0,1)} \left[ \left(2a\sqrt{n} + iZ \right)^n\right] \\
            &= 2^n n^{n/2}\bE_{Z \sim \cN(0,1)} \ab{\rb{a+i\frac{Z}{2\sqrt{n}}}^n} \\
             &= 2^n n^{n/2}\wt H_{n}\rb{a,\frac1{4n}} 
		\end{align*}
where we recall the definition $\wt H_n(x,t) = \mathbb{E}_{W\sim\cN(0,t)}[(x+iW)^n]$. We now manipulate the Gaussian variances to arrange for a power of $n$ to appear as follows:
$$\wt H_n\rb{a,\frac1{4n}} = \mathbb{E}_{W\sim \cN(0,\frac{1}{n})}\left[ \rb{a+i\frac{W}{2}}^n \right] = \sqrt{\frac{n}{2\pi}} \intop_{-\infty}^\infty \rb{a+i\frac{w}{2}}^n e^{-\frac{1}{2}nw^2} \d w$$
which we now write as the following Laplace-style integral:
             \begin{align}
			\label{eq:Hn-integral-Fa}
			\wt H_n\rb{a,\frac1{4n}} = \sqrt{\frac n{2\pi}}\, \int_{-\infty}^\infty \exp\set{n F_a(w)}\, \d w,\quad \quad F_a(w) = \log\rb{a+\frac i2 w}-\frac{w^2}2.
		\end{align}
This is now perfectly ripe for analysis with the Laplace method. The function $F_a$ has derivative $F_a'(w) = \frac1{w-2ai}-w$, whence the critical points are found by the roots of the quadratic equation $(w-2ai)w-1$. Depending on the value of $a$, one of three things can happen:  there can be two conjugate pair roots, two distinct roots or a repeated root; see Figure \ref{fig:constant_phase_curves} for the location of these critical points. This trichotomy is what leads to the three different regimes. 

In each regime, to compute the formula we must find the critical point(s) and then deform the contour of integration from the real axis to a curve in the complex plane passing through the critical point(s) so that the method of Laplace can be used (this shift is sometimes called "the method of stationary phase" since the imaginary part is constant on each curve). One can explicitly compute these curves by finding the real and imaginary parts of $F_a$  as follows
		\begin{align}
			\label{eq:RealFa}
			\re\rb{F_a(x+iy)} & =\frac12\ln\rb{ \rb{a-\frac{y}{2}}^{2}+\rb{\frac{x}{2}}^{2}}-\frac{x^{2}}{2}-\frac{y^{2}}{2}\\
			\label{eq:ImaginaryFa}
			\im\left(F_a(x+iy)\right) & =\arctan\left(\frac{x/2}{a-y/2}\right)-xy.
		\end{align}
This leads to constant phase curves visualized in Figure \ref{fig:constant_phase_curves} . Then, it remains to apply the Laplace method to the integral to get the limiting behaviour of $\wt H(a,\frac{1}{4n})$ in each of the cases. The details of the this are explained in Lemmas \ref{lem:a-less-one}, \ref{lem:a-bigger-one}  and the analogous requited limit for $\wt H_n(1 + \frac{1}{2} u n^{-2/4},\frac{1}{4n}) $ in Lemma  \ref{lem:Hermite-Airy-asymptotics}
                \end{proof}
           \begin{remark}\label{rem:small-x}
The phase term $\theta_n(a)  := n(a\sqrt{1-a^2}-\arccos(a))-\frac{1}{2}\arccos(a)+\frac{1}{4}\pi$ that appears inside the cosine $\cos(\theta_n(a))$ in the small $x$ approximation \eqref{eq:small-x-approx} has a nice interpretation which allows one to analyse the zeros. 

First notice that the function $\theta_n:[-1,1]\to \bR$, is monotone increasing from $f(-1)=-(n+\frac{1}{4})\pi$ to $f(1)=\frac{1}{4}\pi$ . Therefore $\cos(\theta_n(a))$ will have exactly $n$ zeros. This captures the general behaviour of the $n$ zeros of the $n$-th Hermite polynomial. (But note that the proof we give here only gives convergence of the function values; not the location of the zeros. More care would be needed to control the location of the zeros rigorously.)

Moreover,  notice that $\frac{\d}{\d a}\theta_n(a)=2n\sqrt{1-a^2}+\frac{1}{2\sqrt{1-a^2}}$ , which is a mixture of the semi-circle density $\sqrt{1-a^2}$ on $[-1,1]$ and the arcsine density $\frac{1}{\sqrt{1-a^2}}$ on $[-1,1]$. The weighting of the mixture is $n$ parts semi-circle law to $\frac{1}{2}$ a part arcsine law. By normalizing, we can create a cumulative distribution function (CDF) from $\theta_n$ which is this weighted mixture of the semi-circle CDF and the arcsine CDF:
$$CDF_{\theta_n}(a) := \frac{\theta_n(a)+(n+\frac{1}{4})\pi}{(n+\frac{1}{2})\pi} = \frac{n}{n+\frac{1}{2}} CDF_\text{semi-circle}(a)+\frac{ \frac{1}{2} }{n+\frac{1}{2}} CDF_\text{arcsine}(a)$$
Using the zeros for $\cos(\theta)$, we find that the location of the zeros of $\cos(\theta_n(a))$ occur at the locations $a_k, k=1,\ldots,n$ that satisfy $CDF_{\theta_n}(a_k) = \displaystyle\frac{k-\frac{1}{4}}{n+\frac{1}{2}} $. In other words, the zeros are certain quantiles of this CDF, with the space between any two consecutive zeros always capturing a mass of $\frac{1}{n+\half}$ . As $n$ grows large, the contribution from the arcsine law is asymptotically diminished and we recover the known fact that the $n$ zeros of the Hermite polynomial are approximately located according to the semi-circle density, with fewer zeros near the edges and more zeros in the centre.    
\end{remark}
            
\begin{figure}[!ht]
    \centering
    \includegraphics[width=1\linewidth]{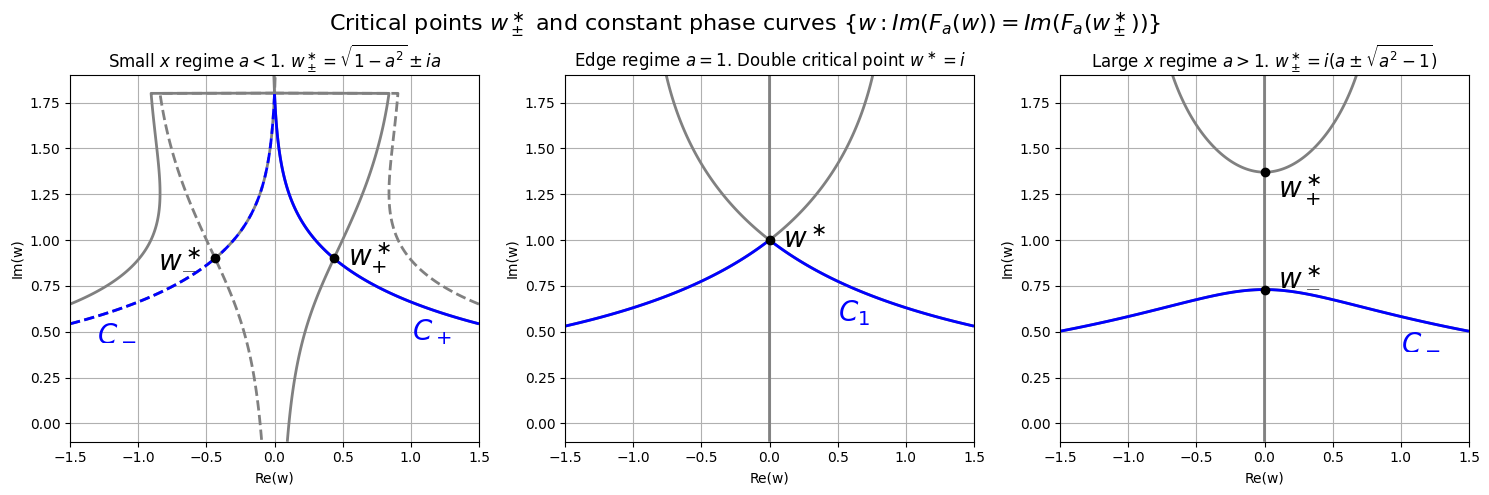}
    \caption[The location of critical points $w^\ast_\pm$ and constant phase curves]{The location of critical points $w^\ast_\pm$ and constant phase curves $\{w : Im(F_a(w)) = Im(F_a(w^\ast_\pm)) \}$ shown in grey. The location of the critical points changes from conjugate pair of roots when $a<1$ (small x regime), to a double critical point when $a=1$, to two distinct roots when $a>1$ (large $x$ regime). The blue curves $C_\pm$ and $C_1$ are used in the proofs: we deform the integral from the real axis to these curves so that we can apply the Laplace appropriation method through the appropriate critical points. }
    \label{fig:constant_phase_curves}
\end{figure}

            \subsubsection{Proof of small-\texorpdfstring{$x$ and large-$x$}{x and large-x} asymptotics}
            
            Both the small and large asymptotics are proved by applying the ordinary Laplace approximation idea from Proposition \ref{prop:laplace}. 

\begin{lemma}
		\label{lem:a-bigger-one}
		For $a>1$, writing $a=\cosh(\gamma)$,
		\begin{align}
			\wt H_n \rb{a,\frac1{4n}} = \wt H_n \rb{\cosh(\gamma),\frac1{4n}} & = e^{n (-\log(2)+\gamma+\frac12 e^{-2\gamma})}\, \frac{e^{\gamma/2}}{2\sqrt{\sinh(\gamma)}}
			\rb{1+o\rb{1}}
		\end{align}
	\end{lemma}
		\begin{proof}
		Recall \eqref{eq:Hn-integral-Fa}:
		\begin{align*}
			\wt H_n\rb{a,\frac1{4n}} = \sqrt{\frac n{2\pi}}\, \int_{-\infty}^\infty \exp\set{n F_a(w)}\, \d w,\quad \quad F_a(w) = \log\rb{a+\frac i2 w}-\frac{w^2}2.
		\end{align*}
So the critical points of $F_a$ for $a>1$ are found by the quadratic formula $w^\ast_\pm=w^\ast_\pm(a)=				i(a\pm\sqrt{a^{2}-1})$. By \eqref{eq:ImaginaryFa}, $\im F_a(w)=0$ for any purely imaginary $w\ne 2ai$, in particular for $w^\ast_\pm$. The constant-phase curve $ C_a = \{ w :  \im F_a(w)=0 \}$ is the disjoint union of two parts $C_+$ and $C_-$, which pass through the points $w^\ast_+$ and $w^\ast_-$ respectively, see Figure \ref{fig:constant_phase_curves} on the right.
		
		Thus we deform the contour of integration from $(-\infty,\infty)$ to $C_-$, which can be justified because $F_a$ is entire in any bounded subset of the region between the two curves.   We have now:
      $$  \wt H_n\rb{a,\frac1{4n}} = \sqrt{\frac n{2\pi}}\, \int_{C_-} \exp\set{n F_a(w)}\, \d w $$In preparation for the Laplace method, we now compute the values of the function $F_a$ and its derivatives at the critical point $w_-^\ast$ where $F_a$ has its global maximum in terms of $\gamma$ as follows:
      \begin{align*}
			w_-^\ast &= i(a-\sqrt{a^{2}-1})=i(\cosh(\gamma)-\sinh(\gamma))=ie^{-\gamma},\\
			F_a\rb{w_-^\ast} 
			& =\log\rb{\frac12 e^{\gamma}}+\frac12 e^{-2\gamma}=-\log(2)+\gamma+\frac12 e^{-2\gamma},\\
			F_a''\rb{w_-^\ast} &=\frac{1}{4\,(e^{\gamma}/2)^{2}}-1
			= e^{-2\gamma}-1 
		\end{align*}
Hence, applying the Laplace approximation in the form of Proposition \ref{prop:laplace}, we obtain
            \begin{align*}
			\wt H_n\rb{a,\frac{1}{4n}} & =  e^{n (-\log(2)+\gamma+\frac12 e^{-2\gamma})}\, \frac1{\sqrt{1-e^{-2\gamma}}}\rb{1+O\rb{n^{-1}}}\\
			& =  2^{-n} e^{n (\gamma+\frac12 e^{-2\gamma})}\, \frac{e^{\gamma/2}}{\sqrt{2\sinh(\gamma)}}
			\rb{1+O\rb{n^{-1}}}
		\end{align*}
		as claimed.
	\end{proof}

\begin{lemma}
		\label{lem:a-less-one}
		For $a<1$, setting $a=\cos(\alpha)$ , we have
		\begin{align}
			\wt H_n \rb{\cos(\alpha),\frac1{4n}} & = 2^{-n}\,
			\sqrt{\frac{2}{\sin(\alpha)}}
			e^{n(\cos^2\alpha-1/2)}\,
			\cos\rb{n \rb{\sin\alpha\cos\alpha -\alpha} -\frac 12(\alpha+\pi/2)}\rb{1+O\rb{n^{-1}}}
		\end{align}
        \end{lemma}
		\begin{proof}
			Again \eqref{eq:Hn-integral-Fa} gives us
			\begin{align*}
				\wt H_n\rb{a,\frac1{4n}} = \sqrt{\frac n{2\pi}}\, \int_{-\infty}^\infty \exp\set{n F_a(w)}\, dw
			\end{align*}
			with $F_a$ as in \eqref{eq:Hn-integral-Fa}. For $a<1$, the critical points are $w^\ast_\pm=ia\pm\sqrt{1-a^{2}}$.
			By \eqref{eq:ImaginaryFa}, 
			\begin{align*}
				\im\left(F_a(w_\pm^\ast)\right) & =\arctan\left(\frac{\pm \sqrt{1-a^2}}{a}\right)\mp a\sqrt{1-a^2}.
			\end{align*}
			The constant-phase curves for $F_a$ passing through $w_\pm^\ast$ are each a union of three curves. The real line can be deformed to a union of constant-phase curves $C_- \cup C_+$  passing through both $w^\ast_-$ and $w^\ast_+$ (these two curves meet at $w=2ai$ where the integrand is $0$)  See Figure \ref{fig:constant_phase_curves}.
			We note that $C_+$ and $C_-$ are images of each other under the transformation $z\longmapsto -\overline{z}$, which follows from the fact that $\re F_a$ is an odd and $\im F_a$ an even function in the real part and that $\im F_a(w_+^\ast)=-\im F_a(w_-^\ast)$. The parity of the real and imaginary part of $F_a$ then imply that the integrals along $C_-$ and $C_+$ are complex conjugates of each other. In other words, the integral along $C_- \cup C_+$ is twice the real part of the integral along $C_-$. We have then
            \begin{align*}
				\wt H_n\rb{a,\frac1{4n}} & = 2\sqrt{\frac{n}{2\pi}}\, \re \left\{  \int_{C_-} \exp\set{n F_a(w)}\, \d w \right\}
			\end{align*}
We compute the value at the critical point and its second derivative in terms of in terms of $\alpha\in(0,\pi/2)$ with $a=\cos(\alpha)$ to obtain:
\begin{align*}
				F_a\rb{w_-^\ast} &= -\log(2) + \frac{e^{2i\alpha}}2-i\alpha = -\log(2) + \frac12\cos(2\alpha) + i\rb{\frac{\sin(2\alpha)}2-\alpha},\\
				F_a''\rb{w_-^\ast} &= e^{2i\alpha}-1,
			\end{align*}
			Therefore we obtain by use of the Laplace approximation as in Proposition \ref{prop:laplace} that
             \begin{align*}
				\wt H_n\rb{a,\frac1{4n}} & = 2\sqrt{\frac{n}{2\pi}}\, \re \left\{\rb{\frac12}^n\,\exp\set{ n\rb{\frac{\cos(\alpha)}2+i\rb{\frac{\sin(\alpha)}2-\alpha} }} \sqrt{\frac{2\pi}{n(1-e^{2i\alpha})}}
				\rb{1+o\rb{1}}\right\}.
			\end{align*}
	We need to compute the real part of this last expression. First note that $1-e^{2i\alpha} = 2e^{i\alpha}\sin(\alpha)$, so that (choosing the branch of the square root to preserve continuity in $\alpha$ on $[0,\pi)$,
			\begin{align*}
				\sqrt{\frac{2\pi}{n(1-e^{2i\alpha})}} & = \sqrt{\frac{2\pi}n} \rb{2e^{i\alpha}\sin(2\alpha)}^{-1/2} = e^{-\frac i2(\alpha+\pi/2)} \sqrt{\frac{\pi}{n \sin(\alpha)}}
			\end{align*}
		Hence, we obtain
			\begin{align*}
				& \re \left\{\rb{\frac12}^n\,\exp\set{ n\rb{\frac{\cos(\alpha)}2+i\rb{\frac{\sin(\alpha)}2-\alpha} }} \sqrt{\frac{2\pi}{n(1-e^{2i\alpha})}} \right\} \\
				=& \rb{\frac12}^n\,
				\sqrt{\frac{\pi}{n\sin(\alpha)}}
				e^{n\frac{\cos(2\alpha)}2}\,
				\re\left\{
					\exp\set{i\rb{\frac{n}2\,\sin(2\alpha)-n\alpha -\frac 12(\alpha+\pi/2)}}\right\}\\
				=& \rb{\frac12}^n\,
				\sqrt{\frac{\pi}{n\sin(\alpha)}}
				e^{n\frac{\cos(2\alpha)}2}\,
				\cos\rb{n \rb{\frac{\sin(2\alpha)}2\,-\alpha} -\frac 12(\alpha+\pi/2)}.
			\end{align*}
	Finally we apply the trigonometric identities $\frac{\cos(2\alpha)}2=\cos^2(\alpha)-\frac12$ and $\frac{\sin(2\alpha)}2=\sin\alpha\cos\alpha$ and obtain
			\begin{align*}
				\wt H_n\rb{a,\frac1{4n}} & = 2\sqrt{\frac n{2\pi}}\, \rb{\frac12}^n\,
				\sqrt{\frac{\pi}{n\sin(\alpha)}}
				e^{n\frac{\cos(2\alpha)}2}
				\cos\rb{n \rb{\frac{\sin{2\alpha}}{2} -\alpha} -\frac 12(\alpha+\pi/2)}(1+o(1))\\
				& =
				\rb{\frac12}^{n-1/2}\,
				\sqrt{\frac{1}{\sin(\alpha)}}
				e^{n(\cos^2\alpha-1/2)}\,
				\cos\rb{n \rb{\sin\alpha\cos\alpha -\alpha} -\frac 12(\alpha+\pi/2)}(1+o(1)).
			\end{align*}
			as claimed.

            \end{proof}
        

            \subsubsection{Proof of edge asymptotics }
            \label{subsec:edge-scaling}
The edge scaling occurs because of a double critical point, so the ordinary Laplace method, which features a 2nd order Taylor expansion and the the Gaussian integral, must be replaced by a 3rd order expansion and the integral of a exponential-cubic function. This gives rise to the Airy function, which we denote $\A$  and define below in the form that it appears in our asymptotics.            \begin{definition}[Airy function]\label{def:airy}
		Let $\AiryCurve$ denote the union of the line segments	$\set{t e^{i\pi/3}\colon t\geq 0}$ and $\set{t e^{-i\pi/3}\colon t\geq 0}$, that is the union of the two infinite rays in the complex plane emanating from the origin that make angles $\pi/3$ and $-\pi/3$ with the real axis (see Figure \ref{fig:AiryContourCurves}). For $u\in\bC$, the \emph{Airy function} is defined by the complex contour integral
		\begin{align}
			\label{eq:defAiry-standard}
			\A(u) & = \frac1{2\pi i}\,\int_\AiryCurve \exp\set{\frac{t^3}3-ut}\, \d t
		\end{align}
        Let $\wedge$  be the union of the line segments$\set{t e^{-i\pi/6}\colon t\geq 0}$ and $\set{t e^{-5i\pi/6}\colon t\geq 0}$. Applying the change of variables $t=iz$ which rotates the integration contour $\RotatedAiryCurve$ into $\AiryCurve$, with $dw = \frac{dt} i$ and $-z^3 = t^3$, we see that the Airy function can also equivalently be written as:
        \begin{align*}
			\A(u) &= \frac1{2\pi}\int_{\RotatedAiryCurve}\exp\left(-\frac{i}{3}z^{3}-iuz\right)\, \d z 
		\end{align*}
     See Figure \ref{fig:AiryContourCurves} for an illustration of these curves.

	\end{definition}

    \begin{figure}[!ht]
        \centering
        \begin{tikzpicture}[>=latex,scale=1.25]

\tikzset{
    axis/.style={thick,->},
    wedge/.style={thick,blue},
    airycurve/.style={thick,blue}
 }


\begin{scope}[shift={(-2.5,0)}]
    \draw[axis] (-1.5,0) -- (1.5,0) node[right] {$\Re$};
    \draw[axis] (0,-1.5) -- (0,1.5) node[above] {$\Im$};
    
    \draw[wedge] (0,0) -- ({1.2*cos(-30)},{1.2*sin(-30)});
    \draw[wedge] (0,0) -- ({1.2*cos(-150)},{1.2*sin(-150)});
    
    \anglabel{1.35}{-30}{$-\pi/6$};
    \anglabel{1.35}{-150}{$-5\pi/6$};
    
    \node at (0,-1.8) {The curve $\wedge$};
\end{scope}

\begin{scope}[shift={(2.5,0)}]
    \draw[axis] (-1.5,0) -- (1.5,0) node[right] {$\Re$};
    \draw[axis] (0,-1.5) -- (0,1.5) node[above] {$\Im$};
    
    \draw[airycurve] (0,0) -- ({1.2*cos(60)},{1.2*sin(60)});
    \draw[airycurve] (0,0) -- ({1.2*cos(-60)},{1.2*sin(-60)});
    
    \anglabel{1.35}{60}{$\pi/3$};
    \anglabel{1.35}{-60}{$-\pi/3$};
    
    \node at (0,-1.8) {The curve $\langle$};
\end{scope}

\end{tikzpicture}
        
        \caption[The curves $\wedge$ and $\langle$ that appear in definitions of $Ai(u)$]{The curves $\wedge$ and $\langle$ that appear in two equivalent definitions of $Ai(u)$ from Definition \ref{def:airy}}.
        \label{fig:AiryContourCurves}
    \end{figure}


	\begin{remark}
		\label{rem:sketch-a-zero}
		Before giving a detailed proof, it is instructive to first sketch an explanation of the appearance of the Airy function without the other complicating details. For this, we take $u=0$  and sketch out the limit for $\fH_n(2\sqrt{n})$, the point exactly on the edge. The proof follows three steps:
        
	\paragraph{Step 1:} Write the integral:
    \begin{align*}
			\fH_n\rb{2\sqrt n} = 2^n n^{n/2}\wt H_{n}\rb{1,\frac1{4n}} & =2^n n^{n/2}\bE \ab{\rb{1+i\frac W2}^n} = 2^n n^{n/2}{\sqrt\frac{n}{2\pi}} \int_{-\infty}^\infty \exp\set{n F_1(w)}\d w,
		\end{align*}
		where $F_1(w)=\log\rb{1+\frac i2\,w}-\frac {w^2}2$. We will use steepest descent to take the $n\to\infty$ limit.
		\paragraph{Step 2:} The function $F_1$ has a double critical point at $w^\ast = i$. The steepest descent curve is $C_1$ defined by $\im F_1(w)=0$.  See Figure \ref{fig:constant_phase_curves} . 
	
		\paragraph{Step 3:} Deform the contour of integration from the real axis to $C_1$. Then only the contribution near the critical point contributes to the integral. Computing the Taylor approximation to third order now, and using $F_1(w^\ast) = -\log 2 + \frac{1}{2}$, $F_1'(w^\ast)=F_1''(w^\ast)=0$, $F_1^{'''}(w^\ast) = 2i $ we get the Taylor approximation $F_1(w) \approx -\log 2 + \frac{1}{2} - i \frac{1}{3} (w-w^\ast)^3 $.  Now change variables by $z = \frac{1}{n^{1/3}}(w-w^\ast)$ and we obtain:
		\begin{align}
			\wt H_n\rb{1,\frac1{4n}} &\approx {\sqrt\frac{n}{2\pi}} \int_{C_1} \exp\set{n \rb{-\log 2 + \frac{1}{2} - i \frac{1}{3} (w-w^\ast)^3}}\d w \\ &= \sqrt{\frac 1{2\pi}} \rb{\frac{\sqrt e}2}^n n^{1/2 - 1/3}
			\int_{\RotatedAiryCurve} \exp\set{-\frac i3 z^3}\, \d z.
		\end{align}
		where $\RotatedAiryCurve$ is the union of the two rays emanating from the origin at angles $-\frac\pi 6$ and $-\frac{5\pi}6$ respectively, which are the left and right tangent directions through the original curve $C_1$ at $w^\ast$ .
	
		\paragraph{Step 4:} Recognise $\frac1{2\pi}\int_{\RotatedAiryCurve} \exp\set{-\frac i3 z^3}\, dz$ as $\A(0)$, obtained by rotating the contour in \eqref{eq:defAiry-standard} by $-\frac\pi2$. This gives
		\begin{align*}
			\fH_n\rb{2\sqrt n}= 2^{n}n^{n/2}
			\wt H_n\rb{1,\frac1{4n}} \approx \sqrt{2\pi} e^{n/2} n^{n/2} n^{1/6}  \A(0)\quad\text{ as } n\to\infty.
		\end{align*}
We will see the general proof for $\fH_n(2\sqrt{n} + un^{1/6})$ is essentially a perturbation around this idea.    \end{remark}
		\begin{lemma}
		\label{lem:Hermite-Airy-asymptotics}
		For $u\in\bR$ and $n\in\bN$, as $n\to\infty$
		\begin{align}
			\wt H_n\rb{1+\frac12 n^{-2/3}u,\frac1{4n}} & = \frac{n^{1/6}}{\sqrt{2\pi}}\left(\frac{\sqrt{e}}{2}\right)^{n} e^{un^{1/3}}\intop_{\wedge}\exp\left(-\frac{i}{3}z^{3}-iuz\right) \d z \rb{1+o(1)}. \\
            &= \sqrt{2\pi} n^{1/6} \left(\frac{\sqrt{e}}{2}\right)^{n} e^{un^{1/3}} \A(u)  \rb{1+o(1)}.
		\end{align}
	\end{lemma}
    
	\begin{proof}
		Using \eqref{eq:Hn-integral-Fa} for $x=1+\frac12 u n^{-2/3}$, we have
		\begin{align}
			\wt H_{n}\rb{1+\frac12 n^{-2/3}u,\frac1{4n}} & 
			= \sqrt\frac{n}{2\pi} \int_{-\infty}^\infty \exp\set{n\rb{\log\rb{1+\frac12 n^{-2/3}u+\frac i2\,w}-\frac {w^2}2}}\\
			\notag
			& = \sqrt\frac{n}{2\pi} \int_{-\infty}^\infty \exp\set{G_n(u,w)}\\
			\text{ Where }
			\notag
			G_n(u,w) &:= n\log\rb{1+\frac u2\, n^{-2/3}+i\frac w2} - n\frac{w^2}2\\
			\notag
			&=n\rb{\log\left(1+i\frac{w}{2}\right) - \frac{w^2}2} +n\ln\left(1+\frac{un^{-2/3}}{2+iw}\right)\\
			\label{eq:TaylorGn}
			& = n F_1(w)+n^{1/3}\,\frac{u}{2+iw} + O\rb{n^{-1/3}}.
		\end{align}
		Here, the last equality is due to the Taylor expansion of the logarithm around 1. We have $F_1'(w)=\frac1{w-2i}-w$ and  and $F_1''(w) = -\frac1{(w-2i)^2}-1$, both of which evaluate to 0 at $w=i$. So $F_1$ has a double critical point at $w=i$.  Let $C_1$ be the curve of $w\in\bC$ such that $\im F_1(w)=\im F_1(i)=0$ and $\im w\leq 1$ as shown in Figure \ref{fig:constant_phase_curves}. It has coordinate parametrisation
		\begin{align*}
			C_{1} &= \set{x+iy \colon \arctan\rb{\frac x{2-y}}=xy}.
		\end{align*}
		At the double critical point $w^\ast = i$ we also have $F_{1}(i)=\frac12-\log(2)$. Moreover, $F_1'''(i)=
		-\tfrac{i}{4}\bigl(1+\tfrac{i}{2}w\bigr)^{-3}$, which evaluates to $-2i$ at $w=i$. Therefore, we obtain the 3rd order  Taylor expansion for $w$  near $i$ which will be the basis for our Laplace approximation.
        \begin{align}
			\label{eq:Taylor-a-one}
			F_1(w)
			&=
			\frac12-\log(2) - \frac{i}{3}\,(w-i)^{3} + O((w-i)^4)
		\end{align}
	We deform the contour of integration from the real axis to $C_1$  (see Figure \ref{fig:AiryContourCurves}) to obtain an integral going through this critical point:

		\begin{align*}
			\wt H_n\rb{1+n^{-2/3}\,\frac u2,\frac1{4n}} & = \sqrt{\frac n{2\pi}} \int_{C_1} \exp \set{G_n(x,w)}\, \d w 
		\end{align*}
		We now make the change of variables $w=i+n^{-1/3}z$ with $\d  w = n^{-1/3}\d z$.  Let $\Gamma^\ast$ denote the $z$-curve corresponding to the $w$-curve $C_1$ and write $\Gamma^\ast = \Gamma^\ast_>\cup \Gamma^\ast_\leq$ where   $\Gamma^\ast_\leq =\set{z\in \Gamma^\ast \colon \abs{z} \leq n^{\epsilon}}$  where $\epsilon$  is some small power to be chosen a-posteriori and $\Gamma^\ast_> =\set{z\in \Gamma^\ast \colon \abs{z} > n^{\epsilon}}$ is everything else. We first show that the contribution of $\Gamma_>^\ast$ is negligible: notice that:
		$\re( -\frac{i}{3} z^{3})\le -c_2 |z|^{3}$ on $\Gamma^{\ast}_>$ for $c_2=\frac1{3\sqrt2}$, and therefore $\abs{e^{nG_n(x,w)}}
		\le e^{nF_1(i)}\exp\set{- c_2 n^{3\epsilon}}$ . Thus we can bound 
		\begin{align*}
			\sqrt{\frac{n}{2\pi}}\,\abs{
				\int_{ \Gamma_>^\ast } e^{G_n(x,w)}\,dw 
			}
			\leq C 			\sqrt{\frac{n}{2\pi}}\, e^{-c_{2}n^{3\epsilon}}.
		\end{align*}	
Hence, this part of the integral is going to zero exponentially fast and $\wt H_n$ is well approximated by an integral over only the part near the critical point $\Gamma^\ast_\leq$.  Notice also that since $|z|\leq n^{\epsilon}$ on $\Gamma^\ast_\leq$, so we can control the error term in our series expansion \eqref{eq:Taylor-a-one} by $O((w-i)^4) = O(\frac{z^4}{n^{4/3}}) = O(n^{4\epsilon-4/3})$. Hence, as long as $4\epsilon - 1/3 < 0$, even when multiplying by $n$ we have a good approximation in this region:  
\begin{align}
			\label{eq:Taylor-a-one-with-n}
			n F_1(w)
			&=
			n\rb{\frac12-\log(2)} - \frac{i}{3}\,z^{3} + O(n^{-1/3 + 4\epsilon})
		\end{align}
Near $w=i$, the segment $\Gamma_\leq^\ast$ looks like two rays emanating from the origin at angles $-\frac\pi 6$ and $-\frac{5\pi}6$,  as can be seen in Figure \ref{fig:constant_phase_curves}. On the segment $\Gamma_\leq^\ast$ near the critical point, we have by \eqref{eq:TaylorGn} 
        \begin{align*}
e^{G_n(u,w)} &= \exp\left\{n^{1/3} \frac{u}{2+iw} + O\left(n^{-1/3}\right)\right\} e^{nF_1(w)}
\end{align*}
Now from our change of variables \( w = i + n^{-1/3}z \), so that \( 2 + iw = 1 + in^{-1/3}z \). Then,
\begin{align*}
n^{1/3}\frac{u}{2+iw} &= n^{1/3}\frac{u}{1 + in^{-1/3}z} \\
&= un^{1/3}-i u z + O(n^{-1/3}).
\end{align*}
Together with equation \eqref{eq:Taylor-a-one-with-n}, this gives
\begin{align*}
e^{G_n(u,w)} &= e^{(1/2 - \log 2)n} e^{un^{1/3}} \exp\left\{-\frac{i}{3}z^3 - i u z\right\} \left(1 + o(1)\right).
\end{align*}
Therefore, combing all these estimates, we obtain

\begin{align*}
\wt H_n\left(1+\frac{u}{2}n^{-2/3},\frac{1}{4n}\right) 
&= \sqrt{\frac{n}{2\pi}} \left( \frac{\sqrt{e}}{2} \right)^n n^{-1/3} 
e^{u n^{1/3}} \int_{\Gamma^\ast_\leq} \exp\left(-\frac{i}{3}z^3 - i u z\right)\, \d z  \left(1 + o(1)\right)  .
\end{align*}
Finally, we can extend $\Gamma^\ast_\leq$ to $\RotatedAiryCurve$. To see this, note that for any $z\in \RotatedAiryCurve$, the argument of $z^3$ is  $-\frac\pi2$. In other words, $z^3=-i\abs{z}^3$ and so
		\begin{align*}
			\exp\set{-\frac i3 z^3} &= e^{-\abs{z}^3/3}\quad\forall\,z\in\RotatedAiryCurve,
		\end{align*}
		and therefore by standard estimates
		\begin{align*}
			\abs{\int_{\Gamma^\ast_\leq } \exp\set{ -\frac i3 z^3}\, dz - \int_\RotatedAiryCurve  \exp\set{-\frac i3 {z}^3}\, dz} &\leq C n^{-1/3} e^{-n^{1/2}/3},
		\end{align*}
		and hence we have proved Lemma \ref{lem:Hermite-Airy-asymptotics}
	\end{proof}

    \subsection{Exponentially fast decay \texorpdfstring{for $a>1$}{outside the interval}}
    \label{subsec:exponential-decay-Hermites}

    The asymptotics of Section \ref{subsec:three-scaling} reveal a critical phase transition in the Hermite polynomials $\fH_n(x)$ at the location $x=\pm 2 \sqrt{n}$ . For large $x$ values, $|x| > 2\sqrt{n}$, the growth of the Hermite polynomials is overwhelmed by the Gaussian density. Here we obtain an upper bound that holds for all $n$. 

    \begin{proposition}
    \label{prop:bound-uniform-H}
        For all $a>1$ and all $n\in\bN$,
    \begin{align*}
        \abs{\fH_n(2\sqrt n a)}& \leq \frac{n^{n/2}}{\sqrt{2(1-a^2+a\sqrt{a^2-1}})} \, e^{n(a^2-a\sqrt{a^2-1}-\frac12+\arccosh(a))} 
    \end{align*}
    \end{proposition}

\begin{remark}
    Note that the exponential growth rate in Proposition {prop:bound-uniform-H} matches that of the large$-x$ regime of Theorem \ref{thm:main-asymptotics}: if $a=\cosh(\gamma)$ then
			\begin{align*}
            e^{-2\gamma} & = \frac1{(a+\sqrt{a^2-1})^2}=(a-\sqrt{a^2-1})^2 = 2a^2-2a\sqrt{a^2-1}-1 
        \end{align*}
        and so 
        \begin{align*}
        e^{n (\gamma+\frac12 e^{-2\gamma})} & = e^{n(\arccosh(a) + a^2-a\sqrt{a^2-1}-\frac12 )}.
			\end{align*}
\end{remark}

\begin{proof}[Proof of Proposition \ref{prop:bound-uniform-H}]
The approach is similar to the Laplace asymptotics, but this time we must be careful with lower-order terms, since we are looking for an upper bound that holds for all $n\in\bN$. We begin \eqref{eq:Hn-integral-Fa} and deform the contour of integration from $\bR$ to $w_-^\ast + \bR$. Here, $w_-^\ast=i(a-\sqrt{a^2-1})$ is the critical point that lies on the contour $C_-$ we used for the Laplace asymptotics (in fact $C_-$ lies between $\bR$ and $w_-^\ast + \bR$); see Figure \ref{fig:exponential_contour} for an illustration.

\begin{figure}
    \centering
    \includegraphics[width=0.5\linewidth]{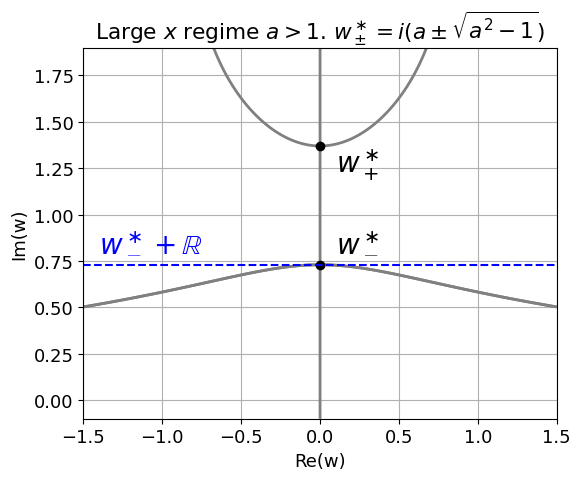}
    \caption[The contour used to prove the exponential bound for large $x$ values]{The contour used to prove the exponential bound for large $x$ values in Proposition \ref{prop:bound-uniform-H} }
    \label{fig:exponential_contour}
\end{figure}

For convenience we write $u_\ast=a-\sqrt{a^2-1}$ so that $w_-^\ast = iu_\ast$. We obtain
     \begin{align*}
			 \fH_n(2a\sqrt n)&=
           n^{n/2} \sqrt{\frac n{2\pi}}\, \int_{-\infty}^\infty \exp\set{n \ab{\log\rb{2a+i \rb{w+iu_\ast}}-\frac{(w+iu_\ast)^2}2}}\, \d w\\
           & =  n^{n/2} \sqrt{\frac n{2\pi}}\, \intop_{-\infty}^\infty \exp\set{n \ab{\log\rb{\frac1{u_\ast}+i w}-\frac{(w+iu_\ast)^2}2}}\, \d w
\intertext{where we have used the fact that $2a-u_\ast=a+\sqrt{a^2-1}=\frac1{u_\ast}$. Taking the absolute value inside the integral we obtain the inequality}
 \abs{\fH_n(2a\sqrt n)} & \leq \frac{n^{(n+1)/2}}{\sqrt{2\pi}} \intop_{-\infty}^\infty \exp\set{n\rb{\frac12\log\rb{\frac1{u_\ast^2}+w^2} -\frac{w^2-u_\ast^2}2 }}\, dw
        \end{align*}
        Writing $\Phi(w)=\frac12\log\rb{\frac1{u_\ast^2}+w^2} -\frac{w^2-u_\ast^2}2 $ and using the inequality $\log(1+x)\leq x$ for $x>0$:
        \begin{align}
            \Phi(w) = -\log u_\ast +\frac12 \log\rb{1+w^2u_\ast^2} + \frac{u_\ast^2}2 - \frac{w^2}2\leq \lambda_1(a) -\lambda_2(a)\frac{w^2}2,
        \end{align}
        where $\lambda_1(a) = -\log u_\ast + \frac {u^2_\ast} 2$ and $\lambda_2(a)=1-u_\ast^2$. 
    Crucially, $\lambda_2(a)>0$ since $u_\ast\in (0,1)
    $ for all $a>1$. More explicitly, $\lambda_1(a)=a^2-a\sqrt{a^2-1}-\frac12+\arccosh(a)$ and $\lambda_2(a)=2 (1-a^2+a\sqrt{a^2-1})$. Putting everything together,
        \begin{align*}
            \abs{\fH_n(2a\sqrt n)} & \leq
            \frac{n^{(n+1)/2}}{\sqrt{2\pi}} \, e^{n\lambda_1(a)} \intop_{-\infty}^\infty e^{-n\lambda_2(a) \frac{w^2}2}\, dw =  \frac{n^{n/2}}{\sqrt{\lambda_2(a)}} \, e^{n\lambda_1(a)} 
        \end{align*}
        as claimed.
        \end{proof}

\section{Applications to the oscillator wave functions \texorpdfstring{$\Psi_n$}{} and the Hermite kernel \texorpdfstring{$K_n$}{}}
\label{sec:applications-oscillator}

\subsection{Definitions \texorpdfstring{of $\Psi_n$ and $K_n$}{}}

One common way the Hermite polynomials appear in applications is through the so called ``oscillator wave functions'', which were named after their appearance as wave functions solution for the quantum harmonic oscillator (See Section \ref{subsec:arcsine}  for this connection.)

\begin{definition}\label{def:oscillator-wave-function}
Define the sequence of functions $\Psi_n:\mathbb{R} \to \mathbb{R}$  of normalized \emph{oscillator wave functions} by 
\begin{equation}
\Psi_n(x) := \frac{1}{\sqrt{n!}} \fH_n( x) \sqrt{ \frac{1}{\sqrt{2\pi}} e^{-\frac{1}{2} x^2 } } 
\end{equation}
\end{definition}
By including ``half'' of the Gaussian density and ``half'' the normalizing constant $n!$, it is seen from the orthogonality of the Hermite polynomials Proposition \ref{prop:orthog-1d} , that the $\Psi_n$ are ortho\emph{normal} in $L^2(\mathbb{R})$:
$$\intop_{-\infty}^\infty \Psi_n(x)\Psi_m(x)\d x = \begin{cases}1 &\text{ if } n=m \\ 0&\text{ otherwise} \end{cases}.$$
Hence the non-negative function $\Psi_n^2(x)$ is a probability density function. Adding together copies of the first $n$ density functions from the previous section gives us another object, this time whose total mass is $n$ in some sense. This combination appears naturally related to Dyson's Brownian motion with $n$ particles. See Section \ref{sec:Dyson} for more on the probabilistic interpretation of this object. 

\begin{definition}
\label{def:hermite-kernel}
    The \emph{Hermite kernel} of order $n$ is the function $K_n\colon\bR^2\longrightarrow\bR$ defined by
    \begin{align}
    \label{eq:Hermite_kernel}
        K_n(x,y):= \sum_{j=0}^{n-1}\Psi_j(x) \Psi_j(y) = \frac{e^{-x^2/4-y^2/4}}{\sqrt{2\pi}} \sum_{j=0}^{n-1} \frac1{j!} \fH_j(x) \fH_j(y) .
    \end{align}
\end{definition}

In this section, we prove various useful facts about $\Psi_n$ and $K_n$. Many of these results themselves have probabilistic applications and are motivated by their use in studying Dyson's Brownian motion, see Section \ref{sec:Dyson}.

\subsection{Exponentially fast decay to zero for \texorpdfstring{$|x| > 2\sqrt{n}$}{large values of x}}
\label{subsec:exponential-decay-Psi}

In Section \ref{subsec:exponential-decay-Hermites} we saw that the Hermite polynomials go to zero exponentially fast outside the interval $[-2\sqrt n, 2\sqrt n]$. From this, we can deduce the behaviour of the wave oscillator functions $\Psi_n$.



\begin{proposition}
\label{prop:bound-uniform-Psi}
    For all $a>1$ and all $n\in\bN$,
    \begin{align*}     
        \Psi_n(2\sqrt n a)^2&\leq  \frac{1}{4\pi\sqrt n (1-a^2+a\sqrt{a^2-1})}  \exp\set{2n\ab{\arccosh(a)-a\sqrt{a^2-1} }}
    \end{align*}
\end{proposition}

\begin{proof}
Recall from Proposition \ref{prop:bound-uniform-H} that
\begin{align*}
  \abs{\fH_n(2\sqrt n a)} \leq \frac{n^{n/2}}{\sqrt{2(1-a^2+a\sqrt{a^2-1}})} \, e^{n(a^2-a\sqrt{a^2-1}-\frac12+\arccosh(a))}  
\end{align*}
 So by the inequality version of Stirling's inequality, Remark \ref{rmk:Stirling-bound},
        \begin{align*}
            \Psi_n(2\sqrt n a)^2 
             &\leq \frac1{\sqrt {2\pi} n!} e^{-2na^2} \frac{n^n}{2 (1-a^2+a\sqrt{a^2-1})} e^{2n(a^2-a\sqrt{a^2-1}-\frac12+\arccosh(a))}\\
           & \leq\frac{1}{2\pi \sqrt n}  \frac{1}{2 (1-a^2+a\sqrt{a^2-1})}  e^{2n\ab{\arccosh(a)-a\sqrt{a^2-1} }}
        \end{align*}
        as claimed.
\end{proof}

\begin{remark}
    \label{rmk:bound-uniform-Psi}
    Later on, specifically in the proof of Proposition \ref{prop:decay-points-unspiked}, it will be useful to bound the rate of exponential decay in terms of a simpler function.
    To do so, note that   
    the function $g$ defined by $g(a) = a\sqrt{a^2-1}-\arccosh(a)$ satisfies $g(1)=0$ and also $g'(a)=2\sqrt{a^2-1}> 2(a-1)$ for $a>1$. Therefore $g(a)>(a-1)^2$ for all $a>1$. Moreover, the function $\lambda_2(a)=2(1-a^2+a\sqrt{a^2-1})$ is strictly increasing on $[1,\infty)$ and satisfies $\lambda_2(1)=0$ . Therefore we can get the following bound: fix $\ep>0$ and let 
    \begin{align}
        \label{eq:def-C-epsilon}
        C_\ep &= \frac1{4\pi (2\ep-\ep^2+(1+\ep)\sqrt{2\ep-\ep^2}}.
        \intertext{Then we have, for all $n\in\bN$ and all $x>1+\ep$,}
    \label{eq:bound-uniform-Psi-remark}
        \Psi_n\rb{2\sqrt n a}^2& \leq\frac{C_\ep }{\sqrt n}e^{2n(\arccosh(a)-a\sqrt{a^2-1})} \leq \frac{C_\ep }{\sqrt n} e^{-n(a-1)^2}.
\end{align}
\end{remark}

\subsection{Cramèr's bound for the global maximum}

The function $\Psi_n(x)$ goes to zero exponentially fast outside of $x = [-2\sqrt{n},2\sqrt{n}]$, and the asymptotics of Section \ref{subsec:three-scaling} suggest it should achieve its global maximum $\max_{x \in \mathbb{R}} |\Psi_n(x)|$ somewhere near the edge of this region $x=\pm 2\sqrt{n}$. However, since all of the asymptotics are proven pointwise, a different approach is needed in order to draw precise conclusions about the actual maximum. In \cite{Indritz1961}, a clever elementary analysis of the zeros of $\Psi^\prime_n$  reveals indeed that the maximum occurs at the closest zero of $\Psi_n^\prime$  to the edge $2\sqrt{n}$   and  finds a bound for the maximum $\max_{x \in \mathbb{R}} |\Psi_n(x)|$. This bound is sometimes called Cramèr's bound. We state this result here for later use.

\begin{proposition}
\textbf{[Cramèr's Bound]} \label{prop:Cramer} The normalized oscillator wave functions 
are bounded by \begin{equation}\label{eq:Cramer-bound} \max_{x\in\mathbb{R}} |\Psi_n(x)| \leq \frac{1}{(2\pi)^{1/4}} \end{equation}
\end{proposition}
\begin{proof}
The proof is given in \cite{Indritz1961}. Note that the function called $E_n(x)$ there is related to $\Psi_n(x)$ by $\Psi_n(x) = \frac{1}{2^{1/4}} E_n(x/\sqrt{2})$ .
\end{proof}

\begin{remark}
    In fact, the bound can be improved to yield 
    \begin{align}
       c n^{-1/12}\leq \max_{x\in\bR} \abs{\Psi_n(x)}\leq C n^{-1/12}
    \end{align}
    for explicit constants $c,C>0$, see \cite{Krasikov}.
    Both \cite{Indritz1961} and \cite{Krasikov} use the fact that for a given $n\in\bN$, the supremum of $\Psi_n$ is after the last zero of $\fH_n$, near $2\sqrt n$, as can be seen in Figure \ref{fig:arcsine}.
\end{remark}

\subsection{Arcsine law limit for \texorpdfstring{$\Psi_n^2$ and the limit of}{} the quantum harmonic oscillator}\label{subsec:arcsine}
	
	\begin{figure}
	    \centering
	    \includegraphics[width=\linewidth]{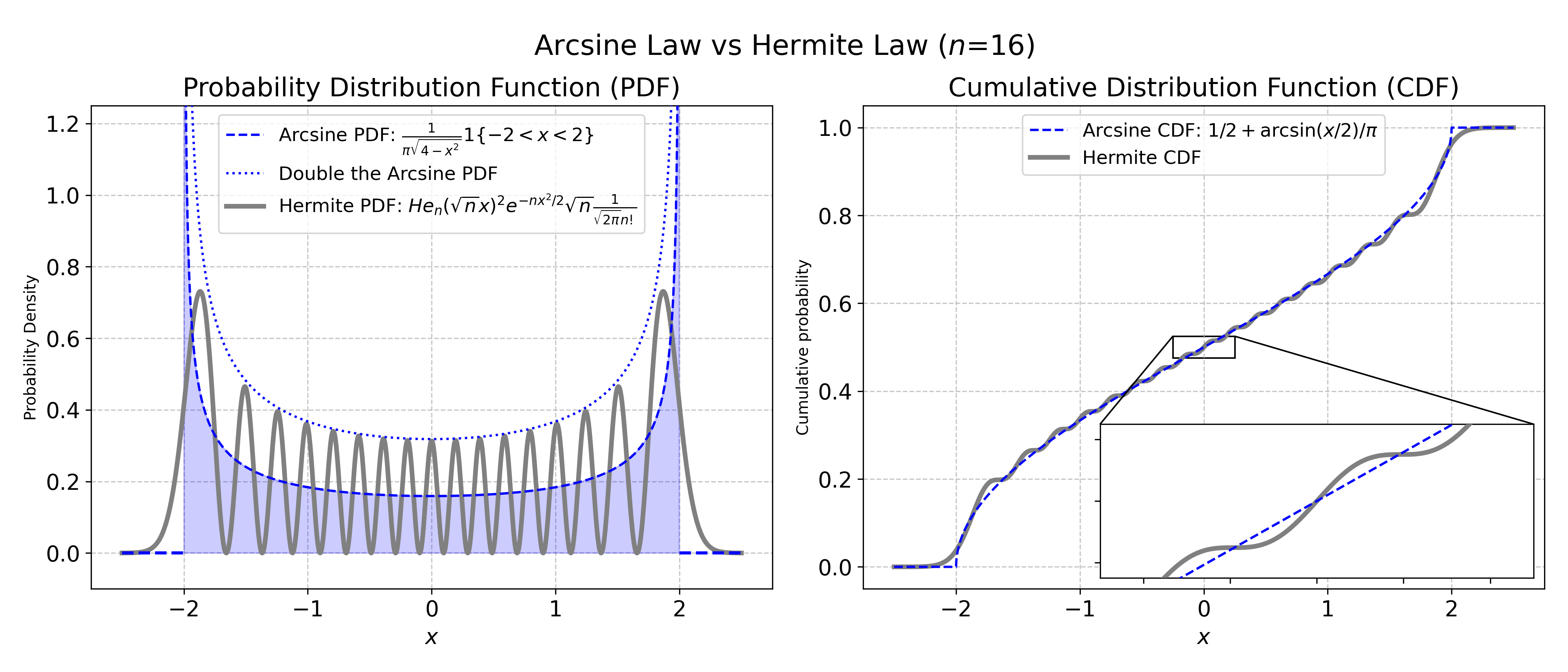}
        \vspace{-3em}
	    \caption[The squared oscillator wave function $\Psi^2_n$ converges weakly to the arcsine law]{By squaring $\fH_n(x)$ and multiplying by a Gaussian density, one gets a probability distribution $\Psi_n(x)^2$ that converges weakly to the arcsine law on $[-2,2]$ after rescaling by $\sqrt{n}$.  This law appears in relation to the wavefunction for the Quantum Harmonic oscillator.}
	    \label{fig:arcsine}
	\end{figure}
	The integral $\intop_{-\infty}^\infty \Psi_n^2(x) \d x = 1 $ means we can interpret the non-negative function $\Psi_n^2(x)$ as a probability density function. Indeed, this probability distribution appears in quantum mechanics: up to appropriately inserting physical constants it is the distribution of the position of a quantum harmonic oscillator, see Proposition \ref{prop:physics}.  The $\sqrt{n}$ scaling of Section \ref{subsec:three-scaling} suggests that this probability distribution needs to be rescaled by $\sqrt{n}$  to have a non-trivial limit as $n\to \infty$, so the correct object to study is the density function $\sqrt{n}\Psi_n(\sqrt{n}x)^2$ . Indeed, this is the case as can be seen in Figure  \ref{fig:arcsine} and the limit is the arcsine distribution, as we prove below.
    
\begin{proposition}
		\label{prop:arcsine}
The sequence of functions $\{\sqrt{n} \Psi_n(\sqrt{n}x)^2 \}_{n=1}^\infty$ converges weakly  as $n\to \infty$ to the arcsine-density function $\rho_{\text{arcsine}}$ on $[-2,2]$, 
		\begin{align*}
			\sqrt{n}\Psi_n(\sqrt{n}x)^2 \Rightarrow \rho_{\text{arcsine}}(x) := \begin{cases} \frac1{\pi \sqrt{4-x^2}}, & x\in [-2,2] \\ 0, & |x|>2\end{cases}.
		\end{align*}
See Figure \ref{fig:arcsine} for an illustration of this convergence.
	\end{proposition}	\begin{proof}
By Proposition \ref{prop:bound-uniform-Psi}, $\Psi_n(\sqrt{n}x) \to 0$ for $|x|>2$ exponentially fast. To see the convergence to the function $\frac1{\pi \sqrt{4-x^2}}$ for $x \in [-2,2]$, we use the approximation \eqref{eq:small-x-approx}  with $a = x/2$ to get 

\begin{align*}
\sqrt{n}\Psi_n(\sqrt{n}x)^2 & = \frac{\sqrt{n}}{\sqrt{2\pi}n!} n^{n}\,
				{\frac{{2}}{\sqrt{1-(x/2)^2}}}
				e^{n(2(x/2)^2-1/2)}\,
				\cos^2\rb{\theta_n(x)}e^{-\frac{1}{2}nx^2}\rb{1+o(1)} \end{align*}
where the phase angle is $\theta_n(x) = n\rb{x/2\sqrt{1-(x/2)^2} -\arccos(x/2)} -\frac 12\rb{\arccos(x/2)+\frac\pi2}$. In this case, the exponential exponents $\frac{1}{2}nx^2$ perfectly cancel, and after simplifying using the Stirling approximation $n! = \sqrt{2\pi n}\left(\frac{n}{e}\right)^n (1+o(1)) $, we are left with:
                
\begin{align}\label{eq:psi-pre-lemma}
\sqrt{n}\Psi_n(\sqrt{n}x)^2  & = \frac{1}{ \pi } \,
				{\frac{{1}}{\sqrt{4-x^2}}}
				\left[ 2\cos^2\rb{\theta_n(x)} \right] \rb{1+o(1)}. \end{align}
Now we note that  the oscillating function $\cos^2(\theta_n(x))\Rightarrow \frac{1}{2}$ weakly as the oscillations increase in frequency, which we state as Lemma \ref{lem:cos-weak-convergence} below. Morally $\sin^2(\theta_n(x))+\cos^2(\theta_n(x))=1$ so they should each converge to $\half$ as the frequency increases $\theta_n(x)=O(n)$, but there is a small wrinkle at the points where $\theta'_n(x)=0$ that needs to be dealt with.  See Figure \ref{fig:arcsine} for an illustration of how the oscillations weakly converge. Once this technical lemma establishes $2\cos^2(\theta_n(x))\Rightarrow 1$, the result follows from \eqref{eq:psi-pre-lemma}.
\end{proof}    

\begin{lemma}
			\label{lem:cos-weak-convergence}
			Let $p<2$. The sequence of functions $h_n$ defined by $h_n(x) = 2\cos^2(\theta_n(x))$ converges weakly to 1 in $L^p([-1,1])$ as $n\to\infty$.
		\end{lemma}
 
The proof of Lemma \ref{lem:cos-weak-convergence} is given in Appendix \ref{sec:proof-lem-cos-weak-convergence}.   
\begin{proposition}
\label{prop:physics}
        Consider the solution to the time independent Schroedinger equation for the quantum harmonic oscillator, namely: $$-\frac{\hbar^2}{2m}\frac{\d  \psi^2}{\d x^2} + \frac{1}{2} m \omega^2 x^2 \psi^2 = E \psi, $$
    where $\hbar$  is the reduced Planck constant, $m$ is the mass of the particle, $\omega$ is the angular frequency of the quadratic potential\footnote{$\omega$ can alternatively be written as $\omega = \sqrt{k/m}$ in terms of the force constant $k$  of the potential}, $E$  is the total energy, and $\psi \in L^2(\mathbb{R})$ is the wave function. Then the energy levels with non-zero solutions are quantized and obey $n = n(E,\omega,\hbar)=\frac{E}{\hbar \omega} - \frac{1}{2} \in \mathbb{N}$ and the normalized solution with this energy, $\psi_n$, can be written in terms of the constant $c=c(m,\omega,\hbar) := \sqrt{2 m \omega/\hbar}$ as:     $$ \psi_{n}(x;E,m,\omega,\hbar) = \sqrt{c} \Psi_n(c x) = \frac{1}{\sqrt{n!}} \fH_n( c x) \sqrt{ \frac{c}{\sqrt{2\pi}} e^{-\frac{1}{2} (cx)^2 } } $$
Moreover, thinking of $\psi_n$ as depending on the physical parameters $E,m,\omega,\hbar$ , consider the situation where $E,m,\omega$ are held fixed, and the parameter $\hbar \to 0$  in a sequence in such a way that $n=\frac{E}{\hbar \omega}-\half \to \infty$ is a sequence of positive integers tending to $\infty$. Then, in this limit we have weak convergence of the probability density $|\psi_n|^2$ to the arcsine law supported on the interval $x \in [-\sqrt{\frac{2E}{m\omega^2}},\sqrt{\frac{2E}{m\omega^2}}]$, namely:
$$|\psi_n(x;E,m,\omega,\hbar)|^2 \Rightarrow  \frac{1}{\pi\sqrt{\frac{2E}{m\omega^2}-x^2}}\text{ as } \hbar \to 0$$
    \end{proposition}
    
\begin{remark}
    Note also that the arcsine law limit is precisely the probability distribution of the spatial location of the classic harmonic oscillator with parameters $E,m,\omega$. In other words, for the classical harmonic oscillator equation $m \ddot{x}(t) + m\omega^2 x(t)=0$ with energy $E$, the solution $x(t)=\sqrt{\frac{2E}{m\omega^2}} \cos(\omega t)$ has this arcsine distribution for its location.  This follows by a simple trigonometric change of variables to see that for any test function $f(x)$ one has 
    $$\intop_0^{2\pi/\omega} f(x(t))\d t=\intop_{-\sqrt{2E/m\omega^2}}^{\sqrt{2E/m\omega^2}} f(x) \frac{1}{\pi\sqrt{\frac{2E}{m\omega^2}-x^2}}\d x .$$
In real life, the value of Planck's constant is fixed at $\hbar = 1.055\ldots \times  10^{-34} \operatorname{J s} \text{(joule seconds)}$ . Compared to human time/energy scales, this is extremely close to zero $\hbar \approx 0$, which explains why humans observe the classical physics of the  $\hbar \to 0$ limit in our day to day experiences. 
\end{remark}

\begin{proof}
The statement of the Schroedinger equation is exactly as in \cite{griffithsQM},  where the quantization for the energy is derived and it is shown that the solution is $\psi_n(x) = \left(\frac{m\omega}{\pi \hbar}\right)^{1/4} \frac{1}{\sqrt{2^n  n!}} H_n(\xi) e^{- \xi^2/2}$ where $H_n$ are the physicists Hermite polynomials ($H_1(x)=2x, H_2(x)=4x^2-2$, etc) and $\xi = \sqrt{ m\omega / \hbar } x$ is rescaled space. This is identical to our stated version by using the translation between physicists and probabilist $H_n(x)= 2^{n/2} \fH_n(\sqrt{2} x)$ and noting that $cx = \sqrt{2}\xi $  by its definition. (One can also directly verify using the change of variable Section \ref{sec:fourier} that this solves the stated Schroedinger PDE)

To prove the convergence to the arcsine law, we notice the following relationship between $|\psi_n|^2$ the functions  $g_n$ is the function from Proposition \ref{prop:arcsine} : 
$$ |\psi_n(x)|^2 = \sqrt{n} \Psi_n\Big( \sqrt{n} \big(\frac{c}{\sqrt{n}} x\big) \Big)^2 \frac{c}{\sqrt{n}} $$
Note also from the quantization $n = E/\hbar\omega -\half$  that the ratio $c/\sqrt{n}$ tends to a constant as $\hbar \to 0$
$$\frac{c}{\sqrt{n}} = \frac{\sqrt{2m\omega/\hbar}} {\sqrt{\frac{E}{\hbar \omega} - \half}} \to \sqrt{\frac{2m\omega^2}{E}}\text{ as }\hbar \to 0 $$
We now invoke Slutsky's theorem, which says that if $f_n(x) \Rightarrow f(x)$ converges weakly, and there are a sequence of constants $a_n$ so that $a_n \to a$ , then we have weak convergence ${a_n} f_n(a_nx) \Rightarrow a f(ax)$ (Note that Slutsky's theorem is usually stated w.r.t to the random variables rather than the densities, where the conclusion is that  $a^{-1}_n X_n \Rightarrow a^{-1} X$. ) In our case, we know that $\sqrt{n}\Psi_n(\sqrt{n}x)^2 \Rightarrow \frac{1}{\pi\sqrt{4-x^2}}$ by Proposition \ref{prop:arcsine}, and we know that $c/\sqrt{n} \to \sqrt{2m\omega^2/E}$  whence the result follows.
	\end{proof}

    \subsection{Semi-circle law limit for \texorpdfstring{$K_n$}{the Hermite kernel}}
    \label{subsec:semicircle-limit}
	\begin{figure}[!ht]
	    \centering
	    \includegraphics[width=\linewidth]{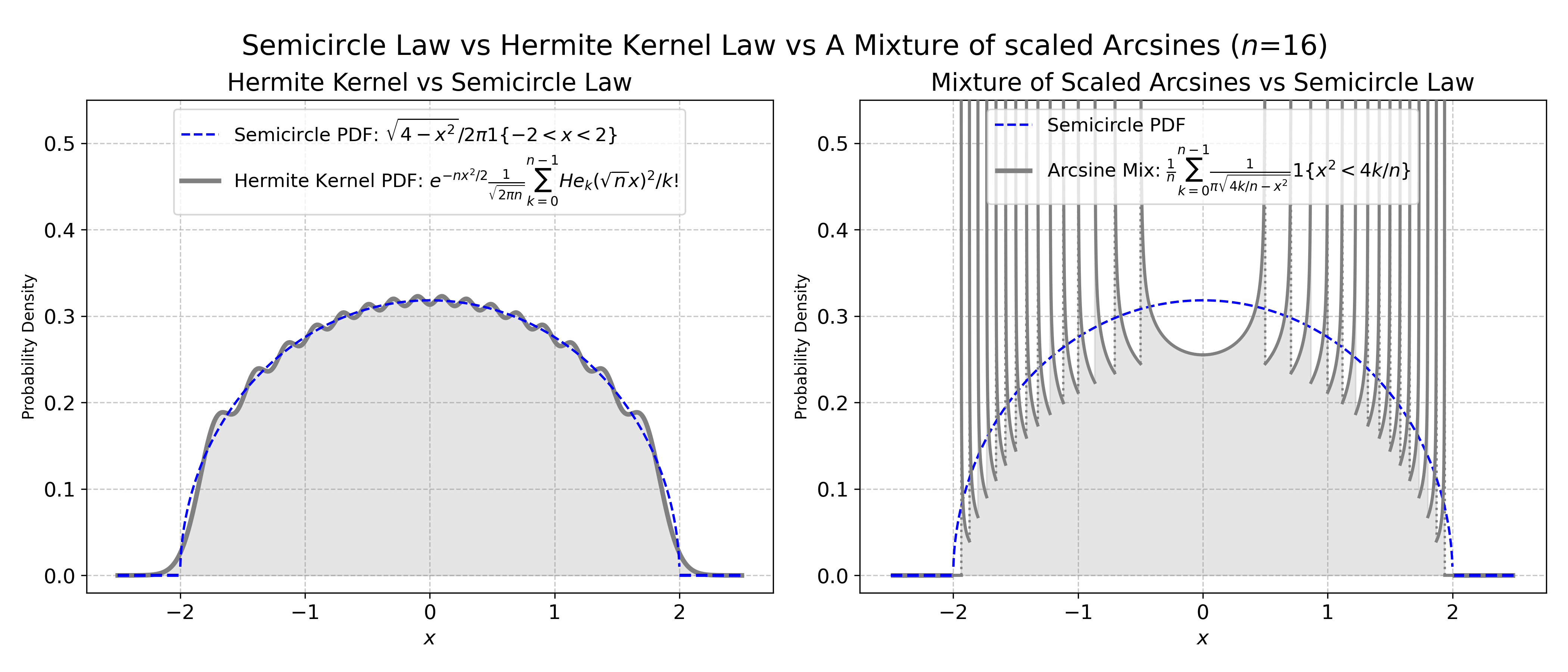}
        \vspace{-3em}
	    \caption[The diagonal of the Hermite kernel $K_n(x,x)$ converges to the semi-circle law]{The combination of Hermite polynomials also known as the diagonal Hermite kernel converges to the semi-circle law. This can be explained piggybacking on the arcsine convergence of the previous section: the same combination of arcsine laws also converges to the semi-circle law. }
	    \label{fig:semicircle}
	\end{figure}

Here, we are interested in a special evaluation of the diagonal, $ K_n(x,x)$, of the Hermite kernel $K_n(x,y)=\sum_{j=0}^n \Psi_j(x)\Psi_j(y)$ from Definition \ref{def:hermite-kernel}.  In Section \ref{sec:Dyson}, it is shown that this is closely related to the counting measure of the particles of Dyson's Brownian motion. In particular, one has that the total integral is $n$: 
$$ \intop_{-\infty}^\infty K_n(x,x) \d x = n$$
Therefore $\frac{1}{n}K_n(x,x)$ is a probability density. As with the arcsine law, to get a non-trivial limit, one must look at the scale $\sqrt{n}$, so the actual object we prove the limit of in this section is $\frac{1}{n} \sqrt{n} K_n(\sqrt{n} x, \sqrt{n} x) = \frac{1}{\sqrt{n}} K_n(\sqrt{n} x, \sqrt{n} x)$. Note that we can write this as:

$$\frac{1}{\sqrt{n}} K_n(\sqrt{n}x,\sqrt{n}x) = \frac{1}{\sqrt{n}}\sum_{j=0}^{n-1} \Psi_j(\sqrt{n}x)^2 =\frac{1}{n} \sum_{j=0}^{n-1} \left( \sqrt{j} \Psi_j\Big(\sqrt{j} \cdot \sqrt{\frac{n}{j}}x\Big)^2 \right) \sqrt{\frac{n}{j}} $$
By Proposition \ref{prop:arcsine}, the functions on the inside, $z \to \sqrt{j}\Psi_j(\sqrt{j}\cdot z))^2$,  are converging as $j \to \infty$ to the arcsine law $\rho_{arcsine}(z)=\frac{1}{\pi\sqrt{4-z^2}}1\{z^2<4\}$, so we might hope (and indeed will shortly prove) that the Hermite kernel to be closely approximated by the same mixture operation applied to arcsine laws, namely:
$$\frac{1}{n} \sum_{j=0}^{n-1} \left( \sqrt{j} \Psi_j\Big(\sqrt{j} \cdot \sqrt{\frac{n}{j}}x\Big)^2 \right) \sqrt{\frac{n}{j}} \stackrel{?}{\approx} \frac{1}{n} \sum_{j=0}^{n-1} \rho_{arcsine}\rb{\sqrt\frac{n}{j}x}\sqrt\frac{n}{j}.$$
See Figure  \ref{fig:semicircle} for an illustration of the resulting function.  This in turn reminds us of a Riemann sum integrating $u=\frac{j}{n}$ and we may hope (and indeed will shortly prove) that as $n\to \infty$ this converges to the following integral, which we can explicitly evaluate

\begin{equation} \label{eq:semi-circ-int}\intop_{0}^1 \rho_{arcsine}\left(\frac{x}{\sqrt{u}}\right)\frac{1}{\sqrt{u}}\d u = \intop_0^1 \frac{\mathbf{1}\{x^2 < 4u \}}{\pi\sqrt{4u-x^2}} \d u = \intop_{x^2/4}^1 \frac{1}{\pi\sqrt{4u-x^2}}\d u = \frac{1}{2\pi}\sqrt{4-x^2}. \end{equation}
This is precisely the semi-circle law! See again Figure \ref{fig:semicircle} for an illustration of this convergence. 

We can make the arguments here rigorous by precisely computing the moments of all of the distribution involved, and proving convergence of the moments. The moments are an ideal tool here because the scaling by $\sqrt{j/n}$  can easily be factored out.

\begin{lemma} \label{lem:moments}
Let $\rho_n(x)$ be a sequence of distributions, and for any fixed $p \in \mathbb{N}$, denote the $p$-th moment by $\mu_p\{\rho_n\}=\intop_{-\infty}^{\infty} x^p \rho_n(x) \d x$. Then the moments of a mixture are given by the mixture of the moments, namely
$$\mu_p\left\{ \frac{1}{n}\sum_{k=1}^n \rho_k\left(\sqrt\frac{n}{k}x\right) \sqrt\frac{n}{k}\right\} = \frac{1}{n}\sum_{k=1}^n \left(\sqrt\frac{k}{n}\right)^p \mu_p\{ \rho_k \}. $$
Moreover, in the case that the moments of $\rho_n$  converge, $\lim_{n\to\infty} \mu_p\{\rho_n\} = \mu_p\{\rho_\ast\} $, to some limiting distribution $\rho_\ast$,  we have the convergence of the moments of the mixture as follows:
$$\lim_{n\to\infty} \frac{1}{n}\sum_{k=1}^n \left(\sqrt\frac{k}{n}\right)^p \mu_p\{ \rho_k \} =  \lim_{n\to \infty}\frac{1}{n}\sum_{k=1}^n \left(\sqrt\frac{k}{n}\right)^p  \mu_p\{ \rho_\ast \}  =  \frac{1}{p/2+1} \mu_p\{\rho_*\}$$
Moreover, these are precisely equal to the moments of integrated new distribution $\intop_0^1 \rho_*(x/\sqrt{u}) /\sqrt{u} \d u$:
$$\frac{1}{p/2+1} \mu_p\{\rho_*\} = \mu_p \left\{ \intop_0^1 \rho_*(\frac{x}{\sqrt{u}}) \frac{1}{\sqrt{u}} \d u \right\} $$
\end{lemma}
    \begin{proof}
   The first statement follows by the simple change of variables $y=cx$ to establish that for any distribution $\rho$ and any  constant $c$, one has: $$\mu_p\{ \rho(cx)c \} = \intop x^p \rho(cx)c \d x = \intop (y/c)^p \rho(y) \d y = \rb{\frac{1}{c}}^p \mu_p\{ \rho  \} $$  To see the convergence to $\frac{1}{p/2+1} \mu_p\{\rho_*\}$, let $\epsilon > 0$ be arbitrary and take $N_\epsilon$  so large so that $|\mu_k(\rho_n) - \mu(\rho)| < \epsilon$ whenever $k > N_\epsilon$ .  But then for $n > N_\epsilon$, the difference between the LHS and RHS is bounded by

\begin{align*} 
   \left| \frac{1}{n}\sum_{k=1}^n \left(\sqrt\frac{k}{n}\right)^p \Big( \mu_p\{ \rho_k \}  -  \mu_p\{ \rho_\ast \} \Big) \right| &\leq \frac{N_\epsilon}{n} \max_{k < N_\epsilon}\Big| \mu_p\{ \rho_k \}  -  \mu_p\{ \rho_\ast \} \Big|  + \frac{n-N_\epsilon}{n} \epsilon\\ 
   &\to \epsilon\text{ as }n\to \infty.
   \end{align*}
   Since this limit equals $\epsilon$, it is eventually less than $2\epsilon$ for $n$ large enough. Since $\epsilon$ is arbitrary, this establishes the limit. Now notice that $\lim_{n\to \infty}\frac{1}{n}\sum_{k=1}^n \left(\sqrt\frac{k}{n}\right)^p  = \intop_0^1 \sqrt{u}^p \d u$ is precisely the Riemmann sum limit for the integral, and evaluating the integral gives $\frac{1}{p/2+1}$ as stated.

   Finally, to evaluate the moments of  $\intop_0^1 \rho_*(\frac{x}{\sqrt{u}}) \frac{1}{\sqrt{u}} \d u $ , we evaluate the integral by doing a change of variable with $y=x/\sqrt{u}$ whose Jacobian is $\d x \d u = \sqrt{u} \d y \d u$, whence the integral factors $$ \int\hspace{-0.5em}\intop_0^1 \rho_*\left(\frac{x}{\sqrt{u}}\right) \frac{1}{\sqrt{u}} x^p \d u \d x =  \intop\hspace{-0.5em} \intop_0^1 \rho_*(y)  (\sqrt{u}y)^p \d u \d y =   \left(\int_0^1 \sqrt{u}^p \d u \right)\hspace{-0.2em}\left( \intop \rho_\ast(y) y^p \d y \right) = \frac{\mu_p\{ \rho_\ast\}}{\frac{p}{2}+1}.$$
    \end{proof}
   
 \begin{corollary}
 \label{cor:hermite-to-sc}
The Hermite kernel distribution $\frac{1}{\sqrt{n}} K_n(\sqrt{n} x,\sqrt{n} x) = \frac{\sqrt{n}}{\sqrt{2\pi}} \sum_{k=0}^{n-1} \frac{\fH_k(\sqrt{n} x)^2}{k!} e^{-\half nx^2}$ from \eqref{eq:Hermite_kernel} converges weakly as $n\to \infty$ to the semi-circle law $\frac{1}{2\pi}\sqrt{4-x^2}$ supported on $[-2,2]$.
\end{corollary}
\begin{proof}
By the convergence of $\sqrt{n}\Psi_n(\sqrt{n}x)^2 =  \frac{\sqrt{n}}{\sqrt{2\pi}n!}\fH_n(\sqrt{n}x)^2 e^{-\frac{1}{2}nx^2}$ to the arcsine law of Proposition \ref{prop:arcsine}, the compact support of the arcsine distribution on $[-2,2]$ and the exponential decay of Proposition \ref{prop:bound-uniform-Psi}, we conclude that the moments of $\sqrt{n}\Psi_n(\sqrt{n}x)^2$ converge to the moments of the arcsine law. 

Now, because the semi-circle law is the corresponding integral of the arcsine law from \ref{eq:semi-circ-int}, by application of Lemma \ref{lem:moments} we consequently see that the moments of $\sqrt{n}^{-1}K_n(\sqrt{n}x,\sqrt{n}x)$ converge to the moments of the semi-circle law. Since the semi-circle law is compactly supported on $[-2,2]$, it is the unique distribution with these moments, and therefore the moment convergence implies the convergence in distribution (See e.g. Section 3.3.5 of \cite{durrettPTE} on the moment problem and in particular Theorem 3.3.12)
\end{proof}
     \begin{remark}
The proof shows convergence in abstract terms without having to actually calculate the moments of the underlying distributions. However, it is amusing to notice that the moment sequences for the arcsine and the semicircle are well known combinatorial sequences. When supported on $[-2,2]$, the odd moments for both distributions are 0, while the even moments for $p=2m$ of the arcine law are $\mu_{2m}\left\{ \rho_{arcsine} \right\} = \binom{2m}{m}$, and for the semi-circle law are the Catalan numbers $\mu_{2m}\left\{ \rho_{semi-circle} \right\} = C_m = \frac{1}{m+1}\binom{2m}{m}$. The factor of $1/(m+1)$ going from the binomial coefficients $\binom{2m}{m}$ to the Catalan numbers is exactly what appears in Lemma \ref{lem:moments}. 
    \end{remark}

\begin{remark}
The connection between the semi-circle law and the arcsine law has been noted in many other contexts too. In \cite{HolcombVirag2019} the convergence is proven using random graphs. In \cite{Ledoux2004}, the convergence to the arcsine law is shown using differential equations. It is subsequently noted there that the mixture converges to $\sqrt{U} \xi_{arcsine}$ where $\xi_{arcsine}$  has the arcsine law, and $U$ is an independent uniform random variable. One can verify that $\sqrt{U}\xi_{arcsine}$  is distributed according to the semicircle law.
\end{remark}

\subsection{The Airy limit}
\label{subsec:Airy-limit}

Zooming in around the edge of the semicircle law at $2\sqrt{n}$ by setting $x=2\sqrt{n} + n^{-1/6}u$, the oscillator wave functions converge to the Airy function $Ai(u)$. This Airy limit is important because drives the behaviour of the \emph{top particle} in an ensemble of $n$  Dyson Brownian walkers, which is related to the Airy kernel defined in \eqref{eq:Airy-kernel} below. See Section \ref{sec:Dyson} for the details.

\begin{proposition} \label{prop:Psi-to-Airy}
Recall the oscillator wave function $\Psi_n(x)=\frac{1}{\sqrt{n!}}\fH_n(x)\sqrt{\frac{1}{\sqrt{2\pi}} e^{-\frac{1}{2} x^2}}$ from Definition \ref{def:oscillator-wave-function}. Let $u$ be fixed and set $x=2\sqrt{n}+n^{-1/6}u$. We have the convergence
$$\lim_{n\to\infty} n^{1/12} \Psi_n(2\sqrt{n} + n^{-1/6}u) = Ai(u)$$
and the derivative converges
$$\lim_{n\to\infty} n^{-1/12} \Psi^\prime_n(2\sqrt{n} + n^{-1/6}u) = Ai^\prime(u)$$
\end{proposition}
\begin{proof}
The convergence to $Ai(u)$  follows by Stirling's formula $n! = \sqrt{2\pi n}\cdot n^n e^{-n}(1+o(1))$ and Proposition \ref{thm:main-asymptotics} for the edge limit $\fH_n\rb{2\sqrt n+n^{-1/6}u}  = \sqrt{2\pi} e^{n/2} n^{n/2} n^{1/6} e^{un^{1/3}} \A(u) \rb{1+o\rb{1} }$ after some serendipitous cancellations. The convergence of the derivative can be obtained by
noting that all the functions $\fH_n,\Psi_n,Ai$  that appear here are  integrals of smooth functions of the arguments, which is how the asymptotics of Section \ref{sec:asymptotics-hermite} were obtained. One can differentiate under the integral sign to realize the derivatives $\fH_n^\prime, \Psi_n^\prime, Ai^\prime$  as integrals with a slightly modified integrand. Then one can check that all the steps of the asymptotics go through with these modified integrands to get convergence of the derivatives.
\end{proof}

The convergence can be extended to the full Hermite kernel as follows.

\begin{proposition}
\label{prop:convergence-hermite-kernel}
Let $u,v\in\bR$ and set $x=2\sqrt{n}+n^{-1/6}u$ and $y=2\sqrt{n}+n^{-1/6}v$. Then \begin{equation}
 \lim_{n\to \infty} n^{-1/6} K_n(2\sqrt{n} + n^{-1/6} u, 2\sqrt{n} + n^{-1/6}v) = K_{\A}(u,v), 
 \end{equation}
 where $K_{\A}$ is the \emph{Airy kernel} defined by
 \begin{equation}
 \label{eq:Airy-kernel}
   K_{\A}(u,v)= \frac{Ai(u)Ai^\prime(v)-Ai^\prime(u)Ai(v)}{u-v} 
 \end{equation}
\end{proposition}
\begin{proof}
Our starting point is the Christoffel-Darboux formula for the Hermite polynomials, Proposition \ref{prop:Christofel-Darboux} , which states that

$$\sum_{j=0}^{n-1} \frac{1}{j!} \fH_j(x)\fH_j(y) = \frac{\fH_n(x)\fH^\prime_n(y) - \fH_n(x)\fH^\prime_n(y)}{x-y}$$
Multiplying this by $\frac{1}{\sqrt{2\pi}}e^{-\frac{1}{4}(x^2-y^2)}$  on both sides gives an identity for $K_n(x,y) = \sum_{j=0}^{n-1}\Psi_j(x)\Psi_j(y) $. The RHS can be manipulated as follows using Lemma \ref{lem:deriv-of-H} which gives $\Psi_j^\prime(x) = \fH_n^\prime(x) \frac{1}{(2\pi)^{1/4}}e^{-\frac{1}{4}x^2} - \frac{1}{2} x \fH_n(x)  \frac{1}{(2\pi)^{1/4}}e^{-\frac{1}{4}x^2}$, and one obtains after some cancellation
$$K_n(x,y) = \frac{\Psi_n(x)\Psi^\prime_n(y) - \Psi^\prime_n(x)\Psi_n(y) }{x-y} - \frac{1}{2} \Psi_n(x) \Psi_n(y)$$
We now add in appropriate powers of $n$  as follows to realize a factor $n^{1/12}\Psi_n$ and $n^{-1/12} \Psi_n$ to set us up for the  convergence of this to $Ai$ as in Proposition \ref{prop:Psi-to-Airy}

$$n^{-1/6} K_n(x,y) = \frac{n^{1/12} \Psi_n(x) n^{-1/12}\Psi^\prime_n(y) - n^{-1/12}\Psi^\prime_n(x)n^{1/12}\Psi_n(y) }{n^{1/6}(x-y)} - \frac{1}{2 n^{1/3}} n^{1/12}\Psi_n(x) n^{1/12}\Psi_n(y)$$ 
and then plugging in $x=2\sqrt{n} + n^{-1/6}u,y=2\sqrt{n}+n^{-1/6}v$, we use the convergence of Proposition \ref{prop:Psi-to-Airy} to conclude the desired result, noting that $n^{1/6}(x-y) = u-v$  and that the second term $\to 0$  as $n\to \infty$.
\end{proof}

%

\section{Random matrix applications part I: Dyson Brownian motion} 
\label{sec:Dyson}

    In this section we introduce Dyson Brownian motion, which can be considered as the eigenvalue process of the matrix analogue of Brownian motion (Dyson Brownian motion) \cite{Dyson} or as Brownian motion in a Weyl chamber, that is the fundamental chamber of the Type A reflection group \cite{JonesOConnell}. We introduce determinantal point processes (DPP) and show that Dyson Brownian motion is a DPP with kernel given by the Hermite kernel from Definition \ref{def:hermite-kernel}. We can then appeal to the results from Sections \ref{subsec:semicircle-limit} and \ref{subsec:Airy-limit} to obtain convergence of the empirical distribution to the semicircle law and of the top particle to the Tracy-Widom distribution.

    \subsection{Two constructions of DBM}
		
	Informally, non-intersecting Brownian motion (NIBM) is the vector-valued process  $\vec X=(X_{1}(t),\ldots,X_{n}(t))$ consisting of $n$
	Brownian motions which are conditioned to not intersect for all time. Alternatively, $\vec X$ is a $n$-dimensional Brownian motion conditioned to stay in the fundamental chamber
	\begin{align*}
		C_\Sigma = \set{\vec x\in\bR^n\colon x_1>x_2>\ldots>x_n}
	\end{align*}
	of the reflection group $W(\Sigma)$ generated by the set $\Sigma$ consisting of the vectors $e_j-e_{j+1}$ for $j\in [n-1]$ (here $e_1,\ldots,e_n$ are the $n$ standard basis vectors of $\bR^n$). When started from $\vec X(0)=0$, the process $\vec X(t)$ represents the $n$ ordered eigenvalues of $\beta(t)$. Here, $\beta$ is the matrix-valued process defined by $\beta(t) = B(t)+\overline{B(t)}$, where the entries of $B$ are independent complex Brownian motions
	
	\paragraph{Killed processes and Doob $h$-transforms.}
	
	There are two ways to make the above rigorous. The first is to start a $n$-dimensional Brownian motion inside $C_\Sigma$ and to kill it\footnote{formally this is done by defining a cemetery state $\Delta$ and setting $B_t=\Delta$ for any $t>T$, where $T$ is the hitting time of $\partial C_\Sigma$ by $X$} as soon as it hits the boundary of $C_\Sigma$. The NIBM process can then be obtained from this killed process by conditioning it to survive forever, via the Doob $h$-transform via the harmonic function
	\begin{align*}
		V(x)=\sum_{\alpha\in\Psi^+} \alpha\cdot x = \prod_{i<j}\rb{x_i-x_j},
	\end{align*}
	where $\Phi^+=\set{e_i-e_j\colon i<j}$ is the positive root system of the reflection group $W(\Sigma)$. Note that the function $V$ is often called the \emph{Vandermonde determinant}. 
    By the general theory of $h$-transforms (see for example \cite{Biane1994}), the NIBM process is the diffusion
	\begin{align}
		\label{eq:SDE-Dyson}
		d X_j(t) =  d B_j(t) + \sum_{\ell \ne j} \frac{dt}{X_j(t)- X_\ell(t)}
	\end{align}
	
	\begin{remark}
		It was shown by C\'epa--L\'epingle \cite{CepaLepingle} that the SDE \eqref{eq:SDE-Dyson} has a unique solution and that $X_1(t)<X_2(t)<\ldots<X_n(t)$ almost surely for all $t>0$. 
		On the other hand, the \emph{Hermitian Brownian motion} is obtained by choosing $2n^2$ independent Brownian motions $B_{j\ell}, \beta_{j\ell}$ ($1\leq j,\ell\leq n$) and then defining the $\bC^{n\times n}$-valued process $H=(H_{j\ell})_{j,\ell=1}^n$ by
		\begin{align*}
			H_{j\ell}(t) & = \begin{cases}
				\frac1{\sqrt2} \rb{B_{j\ell}(t) + i \beta_{j\ell}(t)}\quad & \text{if } j<\ell\\
				B_{jj}(t) & \text{if }j=\ell\\
				\frac1{\sqrt2} \rb{B_{j\ell}(t) - i \beta_{j\ell}(t)} & \text{if } j>\ell.
			\end{cases}
		\end{align*}
		As the name suggests, the matrix $H(t)$ is Hermitian for any $t>0$ and so its eigenvalues $X_1(t),...,X_n(t)$ are real. If we re-order them to be increasing, then \cite{Dyson} they satisfy \eqref{eq:SDE-Dyson} with $\lambda_j(0)=0$.
        \begin{remark}
        \label{rmk:Brownian-scaling}
            Since the eigenvalues respect scalar multiplication of the matrix, this  construction shows that $X$ satisfies \emph{Brownian scaling}: the process $(\frac1{\sqrt\sigma}X(\sigma t))_{t\geq 0}$ is equal in distribution to $X$.
        \end{remark}
		
		Thus, the NIBM process started at zero is the same as the eigenvalue process of Hermitian Brownian motion, and we will from now on refer to either as \emph{Dyson Brownian motion.}
	\end{remark}
	
	\paragraph{Brownian watermelons.}
	We will now consider a second way of rigorously defining Dyson Brownian motion, via watermelons. Given a family of $n$ independent
	Brownian motions $\vec{Z}(t)=\left(Z_{1}(t),\ldots,Z_{n}(t)\right)$
	started from some initial condition $Z_{1}(0)>\ldots>Z_{n}(0)$, define
	the non-crossing event: 
	\[
	NI_{t}=\left\{ Z(s) \in C_\Sigma\quad \forall\,0<s<t\right\} 
	\]
	We want to define the NIBM process to be $Z$ started at 0 and conditioned on $NI_\infty$. However, from this initial condition, the non-intersecting probability
	is $0$.
	
	To get around
	this, we will first define the \emph{$\ep$-of-room Brownian watermelon}
	which starts and ends from $\vec{W}^{\ep}(0)=\vec{0^{\ep}}$
	and $\vec{W}^{\ep}(0)=\vec{0^{\ep}}$, where $\vec{0}^{\ep}=((n-1)\ep,(n-2)\ep,\ldots,\ep,0)$.
	We will then define $\vec{W}$ as the $\ep\to0$
	limit of $\vec{W}^{\ep}$.
	
	\begin{definition}
		A \emph{Brownian watermelon} on the time interval $[0,t^{\ast}]$ consists
		of $n$ non-intersecting Brownian bridges $\vec{W}(t)=\left(W_{1}(t),\ldots,W_{n}(t)\right)$ for $t\in[0,t^{\ast}]$ .
	\end{definition}
	
	The process is called a watermelon because it looks like the stripes of the
	eponymous fruit, see Figure \ref{fig:Watermelon}.
	    \begin{figure}[!ht]
        \centering
        \includegraphics[width=0.5\linewidth]{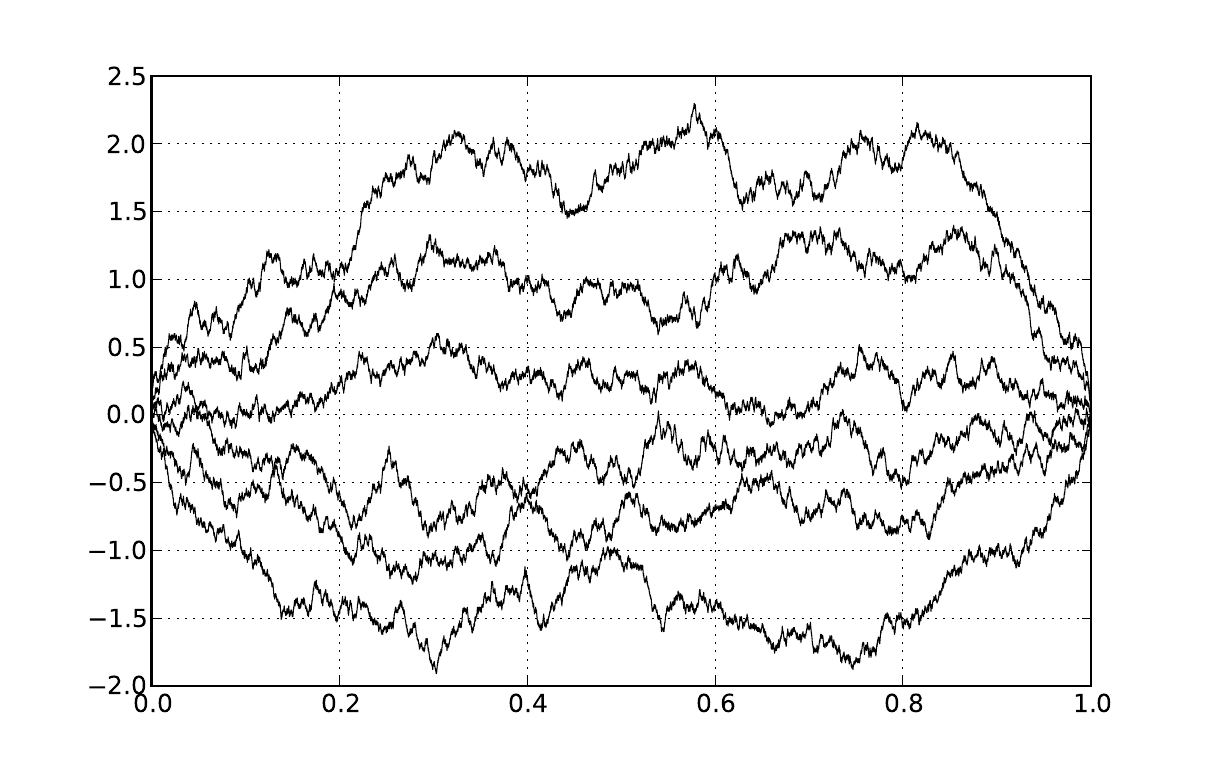}
        \caption{A Brownian watermelon with $n=6$ non-intersecting Brownian bridges.}
        \label{fig:Watermelon}
    \end{figure}

		The proof of the following proposition is a straightforward calculation using the Karlin--McGregor formula. See Appendix \ref{subsec:proof-Watermelon-dist} for details.

	\begin{proposition}
		\label{prop:Watermelon-dist} There exists a unique watermelon process $\uind{\vec{W}}{t^\ast}$ on the interval $[0,t^\ast]$. Moreover the distribution $\uind{\vec{W}}{t^\ast}(t)$ at $t\in[0,t^{\ast}]$ satisfies
		\begin{align}
			\label{eq:prob-dist-watermelon}
			\bP\left(\uind{\vec{W}}{t^\ast}(t)\in d\vec{x}\right)= t^{-n^2/2} c_{n,t,t^{\ast}}V(x_{1},\ldots,x_{n})^{2}\prod_{i=1}^{n}\exp\left(-\frac{x_j^{2}}{2t}\right)\exp\left(-\frac{x_j^{2}}{2(t^{\ast}-t)}\right)    d\vec{x}
		\end{align}
		where 
		\begin{align*}
			c_{n,t,t^{\ast}}=\frac{1}{\sqrt{2\pi}^{n}}\left(\frac{t^{\ast}}{t^{\ast}-t}\right)^{n^{2}/2}\frac{1}{V(0,1,\ldots,n-1)}    
		\end{align*}
	\end{proposition}

	Thus, we can define non-intersecting Brownian motions $\vec{X}(\cdot)$
	conditioned not to intersect for all time as follows: 
	\[
	\vec{X}(\cdot)=\lim_{t^{\ast}\to\infty}\uind{\vec{W}}{t^\ast}(\cdot)
	\]
	and obtain the 
	following formula:

	\begin{proposition}
		\label{prop:density-NIBM}
		The density function of non-intersecting Brownian motions is given by
		\begin{align}
			\label{eq:density-NIBM}
			\bP\left(\vec{X}(t)\in d\vec{x}\right)=\frac{1}{\prod_{k=1}^{n-1}k!}V\left(\frac{x_{1}}{\sqrt{t}},\ldots,\frac{x_{n}}{\sqrt{t}}\right)^{2}\prod_{i=1}^{n}\vp_{t}(x_{i})\, dx
		\end{align}
        where we recall that $\vp_t$ is the density of a centred Gaussian random variable with variance $t$.
	\end{proposition}
	The proof of Proposition \ref{prop:density-NIBM} is straightforward once we note the following property of the Vandermonde determinant.
	
	\begin{lemma}
		\label{lem:Vandermonde}
		For any $c\in\bR$ and $x\in\bR^n$,
		\begin{align*}
			V\rb{c x_1,\ldots,c x_n}= c^{\binom{n}{2}} V\rb{x_1,\ldots,x_n}
		\end{align*}
	\end{lemma}
	
	\begin{proof}
		We can write 
		\begin{align*}
				V\rb{c x_1,\ldots,c x_n}= \prod_{\ell<j} \rb{c x_\ell-cx_j} =\rb{\prod_{\ell<j} c} V\rb{x_1,\ldots, x_n}
		\end{align*}
		and the bracketed product on the right-hand side has $\binom n2$ terms.
	\end{proof}

	\begin{proof}[Proof of Proposition \ref{prop:density-NIBM}]
		As $t^\ast\to \infty$ we have 
		\begin{align*}
			c_{n,t,t^\ast}&\to \frac1{(2\pi)^{n/2}} \frac1{\prod_{k=1}^{n-1}k!}
			\intertext{Moreover, by Lemma \ref{lem:Vandermonde} applied to $c=\frac1{\sqrt{t}}$,}
			t^{-n^2/2} V(x_1,\ldots,x_n)^2    & = t^{n^2-\binom n2} V \rb{\frac{x_1}{\sqrt {t}},\ldots,\frac{x_n}{\sqrt {t}} } = t^{-n/2} V \rb{\frac{x_1}{\sqrt {t}},\ldots,\frac{x_n}{\sqrt {t}} } 
		\end{align*}
Putting these two equations together completes the proof by observing that $\phi(x_j)=\frac1{\sqrt{2\pi t}} e^{-x_j^2/2}$.
	\end{proof}

	\paragraph{(Determinantal) point processes}
	
	\begin{definition}
		A \textbf{simple point process} (SPP) is a collection of random points:
		$S\subset\bR$ such that $S$ has no accumulation points
		and there are never two points at the same location in $\bR$.
	\end{definition}
	
	In this paper, we only consider point processes where the total number of particles is non-random, i.e. $\abs{S}=n$ almost surely for some $n\in\bN$.
	
	\begin{definition}
		We say that the SPP $S$
		has a \emph{$k$-point correlation function} $\rho_{k}:\bR^{k}\to\bR^{+}$ if we have, for every collection of $k$ disjoint sets $A_{1},\ldots,A_{k}\subset\bR$,
		\[
		\bE\left[\prod_{i=1}^{k}\mu_{S}(A_{i})\right]=\intop_{A_{1}}\intop_{A_{2}}\ldots\intop_{A_{k}}\rho_{k}(x_{1},\ldots,x_{k})\d x_{1}\ldots\d x_{k}
		\]
	\end{definition}

	\begin{definition}
		A SPP $S$ is said to be \emph{determinantal} if its correlation kernels $\rho_k$ exist and there is a kernel $K\colon\bR\times\bR\longrightarrow\bR$ such that
		\begin{align}
			\rho_k\rb{x_1\ldots,x_k} & = \det\ab{K\rb{x_j,x_\ell}}_{j,\ell=1}^k
		\end{align}
        for all $k\in\bN$ and $x_1,\ldots,x_k\in\bR$.
	\end{definition}
	In this case we say that $S$ is a \emph{determinantal point process (DPP)} with kernel $K$.
	DPPs make it easier to compute correlations because we have the algebra on our side. This is good news for us, because the NIBM process is determinantal. Its kernel is given in terms of the Hermite kernel, recall Definition \ref{def:hermite-kernel}:
\begin{align*}
K_n(x,y) = \sum_{j=1}^n \Psi_j(x)\Psi_j(y)
\end{align*}
where the functions $\Psi$ defined at the beginning of Section \ref{sec:applications-oscillator} are given by
\begin{align*}
    \Psi_j(x) = \frac{1}{\sqrt{j!}} \fH_j( x) \sqrt{ \frac{1}{\sqrt{2\pi}} e^{-\frac{1}{2} x^2 } } .
\end{align*}

\begin{proposition}
		\label{prop:NIBM-DPP}
		The $n$ points of the non-intersecting Brownian motions form a DPP with kernel $\dysonkernel$ where
		\begin{align}
		    \notag
            \dysonkernel_{n,t}(x,y) = \frac1{\sqrt{t}} K_n\rb{\frac x{\sqrt{t}},\frac{y}{\sqrt t}}
		\end{align}
        
        \begin{remark}
            Recall from the introduction that the Hermite polynomials with arbitrary variance can be written in one dimension as $\wt H_n(x,t) = t^{n/2}  \fH_n\rb{\frac x{\sqrt{t}}}$. So the $\dysonkernel$ can be equivalently written
            \begin{align*}
                \dysonkernel_{n,t}(x,y)=\sqrt{\frac1{2\pi t} e^{-x^2/(2t)-y^2/(2t)}}\sum_{j=0}^{n-1} \frac{1}{t^j j!}\wt H_j(x,t) \wt H_j(y,t).
            \end{align*}

            $$ \frac{1}{\sqrt{t}} \sum_{j=0}^n \Psi_j(x/\sqrt{t})\Psi_j(y\sqrt{t})$$
                  \end{remark}
            

	\end{proposition}
	
	In order to prove Proposition \ref{prop:NIBM-DPP}, we appeal to Theorem 11.4.1 in \cite{Borodin-HbRMT}, which says informally that bi-orthogonal ensembles are determinantal point processes.
	
	\begin{definition}
		A \emph{bi-orthogonal ensemble} is a point process on $\bR$ whose $n$-point correlation function is given by
	\begin{align}
	\rho_{n}(x_{1},\ld,x_{N})=\det_{i,j=1}^{n}\left[\psi_{i}(x_{j})\right]\cdot\det_{i,j=1}^{n}\left[\ph_{i}(x_{j})\right]\label{eq:biorthog_ensemble}
\end{align}
	\end{definition}

	\begin{theorem}[Theorem 11.4.1 in \cite{Borodin-HbRMT}]
		\label{thm:bi-orthog-thm}
		Any biorthogonal ensemble is a DPP with kernel
		\begin{align*}
			K(x,y) & = \sum_{j,\ell=1}^N [G^{-t}]_{j,\ell} \psi_j(x)\phi_\ell(y),
			\intertext{where $G$ is the \emph{Gram matrix} of $\phi$ and $\psi$, defined by}
			G_{j\ell} &= \int_\bR \phi_j(x)\psi_\ell(x)\, dx
		\end{align*}
		and $G^{-t}$ denotes the inverse of the transpose of $G$.
%
%
%
	\end{theorem}

	\begin{proof}[Proof of Proposition \ref{prop:NIBM-DPP}]
		We will re-write \eqref{eq:prob-dist-watermelon} in the form of \eqref{eq:biorthog_ensemble} with $\phi_n =\psi_n= \Psi_n$.
		First, the Vandermonde determinant can be written
		\begin{align*}
			V(x_{1},\ldots,x_{n})=\det\left[\begin{array}{ccccc}
				1 & x_{1} & x_{1}^{2} & \ldots & x_{1}^{n-1}\\
				1 & x_{2} & x_{2}^{2} & \ldots & x_{2}^{n-1}\\
				\vdots & \vdots & \vdots &  & \vdots\\
				1 & x_{n} & x_{n}^{2} & \ldots & x_{n}^{n-1}
			\end{array}\right]
			=\det\left[\begin{array}{ccccc}
				1 & P_{1}(x_{1}) & P_{2}(x_{1}) & \ldots & P_{n-1}(x_{1})\\
				1 & P_{1}(x_{2}) & P_{2}(x_{2}) & \ldots & P_{n-1}(x_{2})\\
				\vdots & \vdots & \vdots &  & \vdots\\
				1 & P_{1}(x_{n}) & P_{2}(x_{n}) & \ldots & P_{n-1}(x_{n})
			\end{array}\right]
		\end{align*}
		where $P_{i}(x)$ is any monic polynomial of degree $i$. The second equality holds because taking linear combinations of the columns does not effect the value of the
		determinant. Letting $P_j(x) = \wt H_j(x,t)$ and factoring $\sqrt{\vp_{t}(x_{i})}$ into the $i$-th row of the
		matrix, and factoring in $\sqrt{\frac{1}{t^{j}j!}}$ into the $j$-th
		column of the matrix, we obtain
		\begin{align}
        \label{eq:density-DBM}
			\rho_{n}(x_{1},\ldots,x_{n})
			& = \rb{\det_{i,j=1}^{n}\left[\frac{\wt H_{i-1}(x_{j},t)\sqrt{\vp_{t}(x_{j})}}{\sqrt{t^{i}i!}}\right]}^{2} = \rb{\det_{i,j=1}^{n} \ab{\Psi_j\rb{\frac x{\sqrt t}}}}^2.
		\end{align}
        where we have again used the relation $\wt H_n(x,t) = t^{n/2}  \fH_n\rb{\frac x{\sqrt{t}}}$.
                In particular, the DBM process is a biorthogonal ensemble where the roles of the $\phi_j$ and the $\psi_j$ are both played by the functions $\Psi_j$ from the quantum harmonic oscillator.
		It remains to compute the Gram matrix. But 
        \begin{align}
            \notag
        G_{j\ell} =\frac{1}{\sqrt{t^{j+\ell} j!\ell!}}\bE\ab{\wt H_j(X,t)\wt H_\ell(X,t)}= \frac{t^\ell}{\sqrt{t^{j+\ell} j!\ell!}}\bE\ab{\wt H_j(X,t)H_\ell(X,t)}=  \delta_{j\ell}
        \end{align} by the comment after Definition \ref{def:dual-H} and then
        Theorem \ref{thm:og}. So $G^{-T}$ is the identity.  Therefore we obtain the rescaled Hermite kernel as claimed.

	\end{proof}

    \subsection{Bulk convergence to the semicircle law}
	The global behaviour of the Dyson Brownian motion process can be obtained by studying the \emph{empirical measure} of the process, defined by 
    \begin{align}
        \label{eq:empirical-DMB}
        \empmeas{n}{t} &= \frac1n \sum_{j=1}^n \delta_{X_j(t)}.
    \end{align}
	In the same way that we saw in Section \ref{sec:applications-oscillator} that the Hermite polynomial $\fH_n$ has a non-trivial behaviour on the interval $[-2\sqrt n,2\sqrt n]$, we will see that the measure $\mu_{n,t}$ tends to concentrate on the interval $[-2\sqrt{nt},2\sqrt{nt}]$. In other words, for fixed time, the eigenvalues will spread out further and further as $n\to\infty$. 
    There are two ways of obtaining a compactly supported limiting measure: first, in the random matrix theory literature, one often rescales $X$ by $\frac1{\sqrt{n}}$. Alternatively, one can let time run slower, at speed $\frac1n$, which will be our approach. Of course, by Brownian scaling, these two approaches are equivalent.
    \begin{proposition}
	    \label{prop:semicircle-DBM}
        Let $t=\frac1n$. The sequence of  random measures $\empmeas{n}{\frac1n}$ converges weakly in expectation, as $n\to\infty$, to the semicircle law from Section \ref{subsec:semicircle-limit}. That is,
        \begin{align*}
            \lim_{n\to\infty} \bE\ab{\int_\bR f(x)\empmeas{n}{\frac1n}(dx)} =\lim_{n\to\infty} \frac1n\sum_{j=1}^n\bE\ab{ f\rb{X_j\rb{\frac1n}}}= \int_\bR f(x)\rho_{\text{semicircle}}(dx)
        \end{align*}
        for any bounded continuous function $f\colon\bR\longrightarrow \bR$.
	\end{proposition}
\begin{proof}
    Since DBM is a determinantal point process kernel $\dysonkernel_{n,t}$ (as in Proposition \ref{prop:NIBM-DPP}), 
    \begin{align}
        \notag
        \bE \ab{\int_\bR f(x)\empmeas nt (d x)} & =
        \frac1n \int f(x)\rho_{1,t}(x)\, dx =  \frac1n \int f(x)\dysonkernel_{n,t}(x,x)\, dx.
    \end{align}
    In other words, the expected empirical measure, $\bE \ab{\empmeas nt}$, is absolutely continuous with respect to Lebesgue measure with density $\frac1n \dysonkernel_{n,t}(x,x)$. Evaluating this at $t=\frac1n$ using  Proposition \ref{prop:NIBM-DPP}, we obtain 
    \begin{align*}
        \frac1n\dysonkernel_{n,1/n}(x,x)=\frac1{\sqrt{n}} K_n(x\sqrt n,x\sqrt n),
    \end{align*}
    which converges to the semicircle law by Corollary \ref{cor:hermite-to-sc}.
    
    \end{proof}

	\subsection{Fluctuations of the top particle}

        Having seen that the empirical measure of the DBM process run up to time $\frac1n$ converges to the semicircle law, whose support is $[-2,2]$, one can ask whether the top particle converges to the value $2$. We now show that it does and compute the fluctuations around this value in the process. Figure \ref{fig:DBM-unspiked} illustrates this. See also \cite{AGZ} and references therein.  
	We will see that the limiting fluctuations are of order $n^{2/3}$ and are distributed according to the \textit{Tracy--Widom GUE distribution}, which is closely related to the Airy function that already appeared in \ref{subsec:Airy-limit}.

\begin{figure}
    \centering
    \includegraphics[width=0.75\linewidth]{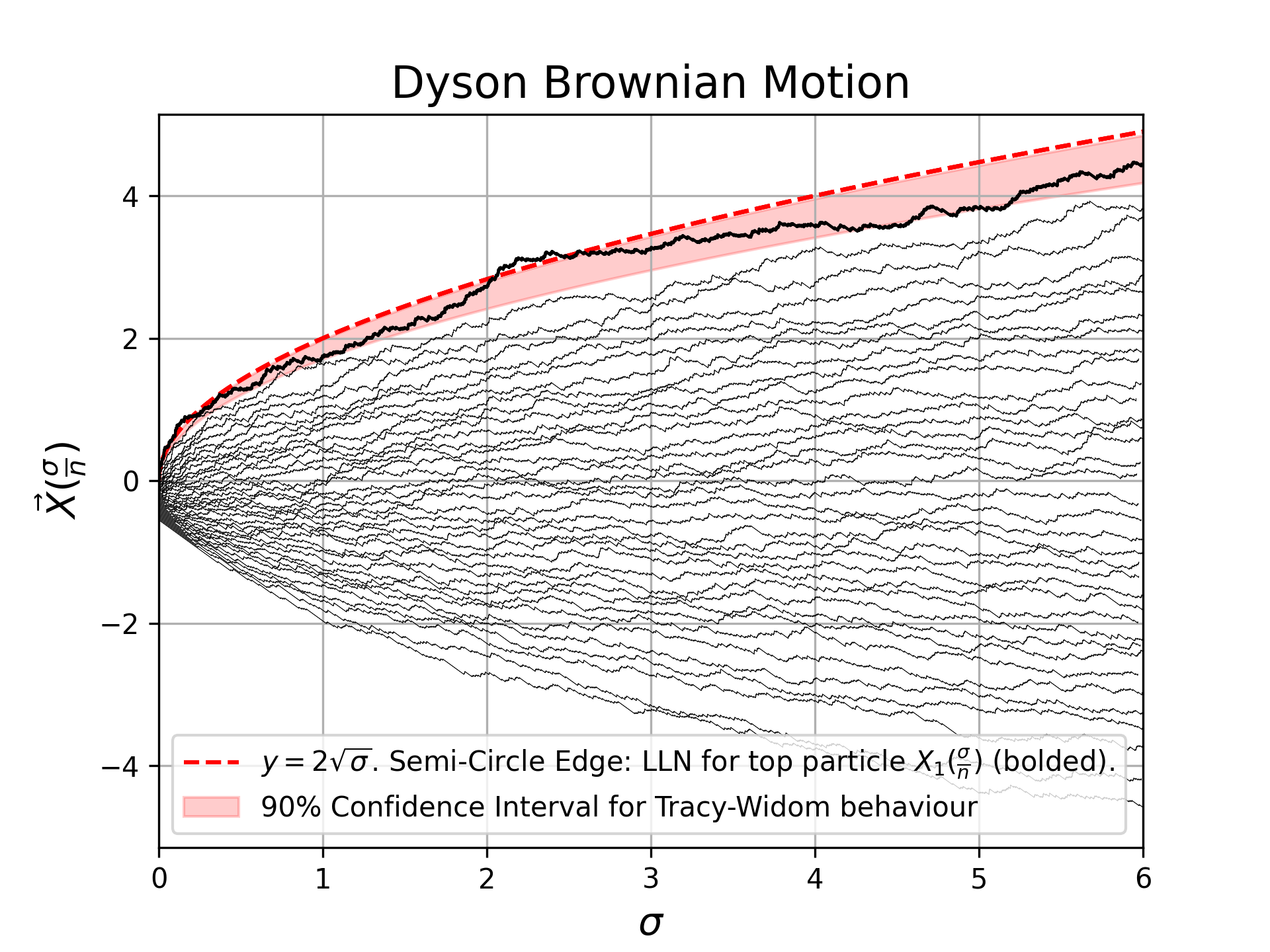}
    \caption[A sample of Dyson Brownian motion with limit theorems for the top particle]{A sample of the Dyson Brownian motion process when $n=36$. The trajectory of the top particle, $X_1(\sigma/n)$ is bolded. The LLN behaviour at $y=2\sqrt{\sigma}$ is illustrated, and a 90\% confidence interval using the result of Proposition \ref{prop:DBM-Airy}, along with the quantiles for the Tracy-Widom distribution $F_2(-3.19)=0.05$, $F_2(-0.23)=0.95$ computed in \cite{Chiani}.}
    \label{fig:DBM-unspiked}
\end{figure}
    
For our calculations, it will be helpful to appeal to theory of Fredholm determinants. We start by  giving a condensed introduction to this theory.

\paragraph{Fredholm determinants}

	One advantage of DPPs is that we can write observables in terms of determinants or Fredholm determinants. For background see for example \cite{McKean}. A \emph{kernel} on a a subset $D\subseteq \bR^d$ is a map $K\colon D\times D\longrightarrow \bR$. Under suitable integrability conditions, every kernel induces a linear map $T\colon L^2(D)\longrightarrow L^2(D)$ defined by $T(f)(x) = \int_D K(x,y)\, f(y)\, dy$. Note that every (suitable) kernel determines a unique operator,but that an operator may have more than one kernel that defines it. For a kernel $K$, we will write $T_K$ for the operator that it gives rise to.

	\begin{definition}
		Let $K\colon D\times D\longrightarrow\bR$ be a kernel. The \emph{Fredholm determinant} of the operator $T_K$ is defined by
		\begin{align}
			\label{eq:def-FD}
			\det(I+T_K) = \sum_{k=0}^\infty \frac{(-1)^k}{k!}\int_{D^k} \det\rb{K\rb{x_i,x_j}}\,\d x_1\ldots\d x_k.
		\end{align}
		provided the right hand side converges.
	\end{definition}
	
\begin{remark}
	Given a DPP with kernel $K$ and number of points $N$, we can always define the Fredholm determinant $\det (I+T_K)$, because the series in \eqref{eq:def-FD} terminates after a finite number ($N$) of terms.
\end{remark}

	\begin{proposition}
		\label{prop:prob_top_particle}
		Let $S$ be a point process that has
		a top particle, and let $X_{1}$ be the location of this top particle.
		Then:
		\begin{align}
			\p\left(X_{1}(t)<a\right)=\det \rb{1-P_{(a,\infty)}KP_{(a,\infty)}}= 
			\sum_{k=0}^{n}\frac{(-1)^{k}}{k!}\intop_{a}^{\infty}\ldots\intop_{a}^{\infty}\rh_{k}(x_{1},\ld x_{k})\d x_{1}\ldots\d x_{k}\label{eq:P_top}
		\end{align}

	\end{proposition}
Note that $P_{(a,\infty)}KP_{(a,\infty)}$ denotes the operator with kernel $(x,y)\longmapsto 1_{(a,\infty)}(x)K(x,y)1_{(a,\infty)}(y)$.
	Proposition \ref{prop:prob_top_particle} follows from setting $\eta=-\one_{[a,\infty)}$ in the following, more general result that can be found in \cite{Borodin-HbRMT}, see (11.2.4). Note that if $S$ has a fixed number of points then it always has a top particle.

	\begin{proposition}
		\label{prop:expectation-determinant}
		Let $S$ be a DPP with $n$ points and let $\eta:\bR\longrightarrow\bR$ be any bounded test function. Then
		\[
		\bE\left[\prod_{x\in S}\left(1+\eta(x)\right)\right]=\det(I+T_{K\eta})
		\]
		where $T_{K\eta}$ is the operator with kernel $(x,y)\longmapsto \sqrt{\eta(x)} K(x,y)\sqrt{\eta(y)} $
	\end{proposition}
	
\paragraph{Proof of convergence}

We now prove convergence of the top particle to the \emph{Tracy--Widom distribution} whose cumulative distribution function is given by a Fredholm determinant:
    \begin{align}
        \label{eq:TW-GUE}
        F_2(s) & = \det\rb{ I - P_{(s,\infty)} K_{\mathrm{Ai}} P_{(s,\infty)}},
    \end{align}
    where $K_{Ai}(\xi,\eta)=\frac{\A(\xi)\,\A'(\eta) - \A'(\xi)\,\A(\eta)}{\xi - \eta}$ is the Airy kernel defined in Proposition \ref{prop:convergence-hermite-kernel}.

\begin{proposition}
	    \label{prop:DBM-Airy}
	Let $F_2$ be the cumulative distribution function of the Tracy--Widom GUE distribution and $s\in\bR$. Let $\sigma>0$, then
    \begin{align*}
    \lim_{n\to\infty} 
\mathbb{P}\!\left(  n^{2/3} \sigma^{-1/2} \rb{X_1\rb{\frac{\sigma}n  } - 2\sqrt{\sigma}}\le s \right) = F_2(s).
    \end{align*}
    \end{proposition}

    \begin{proof}
Recall that the DBM process $X$ is determinantal with kernel $\dysonkernel_{n,\cdot}$. So, by Proposition \ref{prop:prob_top_particle} and a change of variables, the left hand side can be written as the Fredholm determinant $\det\rb{1-\chi_s \rescdysonkernel \chi_s}$ where 
       \begin{align}
       \label{eq:rescdysonkernel}
           \rescdysonkernel_n (u,v) &  =
            n^{-2/3}\sigma^{-1/2} \dysonkernel_{n,\sigma/n} \rb{2\sqrt \sigma+\frac u{n^{2/3}\sigma^{1/2}},2\sqrt \sigma+\frac v{n^{2/3}\sigma^{1/2}}}\\
            & = n^{-2/3}\sigma^{-1/2} \frac{\sqrt{n}}\sigma K_n\rb{2\sqrt{n} + n^{-1/6}\frac u\sigma,2\sqrt{n} + n^{-1/6}\frac v\sigma}\\
            & = n^{-1/6} \frac1\sigma K_n\rb{2\sqrt{n} + n^{-1/6}\frac u\sigma,2\sqrt{n} + n^{-1/6}\frac v\sigma}
       \end{align}
       which converges pointwise to  $\frac1\sigma K_{\A}(u\sigma,v\sigma)$ as $n\to\infty$ by Proposition \ref{prop:convergence-hermite-kernel}. To conclude the convergence of the Fredholm determinant of $\rescdysonkernel_n$ to that of $K_{\A}$ (which is $F_2$), one needs two additional arguments: first, apply the unitary transformation $U_\sigma f(x) = \sigma^{-1/2} f(\sigma x)$, which doesn't affect the Fredholm determinant. Secondly, the pointwise convergence of Proposition \ref{prop:convergence-hermite-kernel} needs to be upgraded, for example to uniform convergence on compact subsets of $\bR\times \bR$. This can be done, for example, by carefully examining the proof of the asymptotics and noting that the bounds hold uniformly on compact sets. See \cite{AGZ} for a rigorous treatment of this point.
    \end{proof}

Due to the exponential decay of the Hermite polynomials (Proposition \ref{prop:bound-uniform-H}) and the the wave oscillator functions $\Psi_n$ (Proposition \ref{prop:bound-uniform-Psi}), the probability of finding a particle to the right of the edge at $2\sqrt\sigma$ goes to zero quite sharply.  More precisely, denote by $\npart(A)$ the number of particles of the DBM process with $n$ particles at time $\frac\sigma n$ in an interval $A$.

\begin{proposition}
\label{prop:decay-points-unspiked}
    Let $\ep>0$ and recall the constant $C_\ep$ defined in \eqref{eq:def-C-epsilon}. For every $n\in\bN$,
    \begin{align}
    \label{eq:bound-exp-spiked}
        \bE\set{\npart ([2\sqrt\sigma +\ep,\infty))\geq 1} \leq \frac{C_\ep\sqrt {8\sigma n}}{\ep } e^{-\frac{\ep^2 n}{4\sigma}}
    \end{align}

A key ingredient in the proof is the following tail bound for Gaussian integrals known as the Mills inequality \cite{MillsRatio1941}. 
For completeness, we give a proof in the appendix, see Section \ref{sec:proof-Mills}.

\begin{lemma}
    \label{lem:mills}
    For any $A,B,c>0$,
    \begin{align}
		\label{eq:Mills-general}
		\frac{1}{2B(A-c)+\frac1{A-c}}\, e^{-B(A-c)^2} \leq \int_A^\infty e^{-B(x-c)^2} \d x &  \leq \frac1{2B(A-c)} \, e^{-B(A-c)^2}
	\end{align}
\end{lemma}

\

    \begin{proof}[Proof of Proposition \ref{prop:decay-points-unspiked}]
        Because the DBM process is determinantal and making the change of variables $x  = 2\sqrt\sigma y$,
 \begin{align}
 \notag
     \bE\ab{\npart (\rb{2\sqrt\sigma+\ep,\infty)}}
     & = 2\sqrt n\int_{1+\hat\ep}^\infty K_n\rb{2y\sqrt n, 2y\sqrt n}  d y\\
    \notag
    & = 2\sqrt n \int_{1+\hat\ep}^\infty \Psi_0\rb{2\sqrt n y}^2\, dy+  2\sqrt n\sum_{j=1}^{n-1}\int_{1+\hat\ep}^\infty \Psi_j\rb{2\sqrt j \rb{\sqrt {\frac n j} y}}^2\, dy\\
   \label{eq:two-terms}
    & \leq \sqrt {\frac{2n}{\pi} } \int_{1+\ep}^\infty e^{-2ny^2}\, dy+ 2 C_\ep \sum_{j=1}^{n-1} \sqrt {\frac nj}\int_{1+\ep/2\sqrt\sigma}^\infty \exp\set{-j\rb{\sqrt{\frac nj}y-1}^2}\, dy
 \end{align}
 where $\hat\ep=\ep/2\sqrt\sigma$ and we have used Remark \ref{rmk:bound-uniform-Psi}.
 By the Mills inequality (Lemma \ref{lem:mills}), we can bound the first term in \eqref{eq:two-terms} as follows:
\begin{align}
    \label{eq:bound-first-term}
    \sqrt {\frac{2n}{\pi} } \int_{1+\ep}^\infty e^{-2ny^2}\, dy 
    \leq  \sqrt {\frac{1}{2\pi} } \frac1{2\sqrt n(1+\hat \ep)} e^{-2n(1+\hat\ep)^2}.
\end{align}
  Turning to the second term of \eqref{eq:two-terms}, 
 \begin{align*}
 \int_{1+\ep/2\sqrt\sigma}^\infty \exp\set{-j\rb{\sqrt{\frac nj}y-1}^2}\, dy
 & =  \int_{1+\ep/2\sqrt\sigma}^\infty \exp\set{-\, \rb{ \sqrt{2n} y - \sqrt {2j}}^2/2}\, dy\\
 & =  \frac1{2\sqrt n}\int_{\sqrt {2n} (1+\hat \epsilon)-\sqrt {2j})}^\infty \exp\set{-z^2/2}\, dz \\ &
 \leq \frac1{\sqrt {2n}} \frac1{\sqrt{2n}(1+\hat\epsilon)-\sqrt {2j})} e^{-[\sqrt{n}(1+\hat\epsilon)-\sqrt {j})]^2}.
 \end{align*}
 Now this last term is increasing in $j$, so we can bound it above by setting $j=n$: 
 \begin{align*}
 \int_{1+\ep/2\sqrt\sigma}^\infty \exp\set{-j\rb{\sqrt{\frac nj}y-1}^2}\, dy
 & \leq \frac1{\sqrt{2n}} \frac1{\sqrt{n}\hat\epsilon} \, e^{-n\hat\epsilon^2} =\frac1{\sqrt{2}n\hat\epsilon}  \, e^{-n\hat\epsilon^2}.
 \end{align*}
 Summing over the $n$ terms, and combining with \eqref{eq:bound-first-term}, we get
\begin{align}
    \label{eq:bound-integrated-Hermite-kernel}
    \sqrt{\frac n\sigma} \intop_{2\sqrt\sigma+\ep}^\infty K_n
    \rb{\frac{x\sqrt n}{\sqrt\sigma}, \frac{x\sqrt n}{\sqrt\sigma}} \, dx & \leq 2 C_\ep \sqrt n\cdot n\  \frac1{\sqrt{2}n\hat\epsilon}  \, e^{-n\hat\epsilon^2}  = \frac{C_\ep\sqrt {2n}}{\hat\ep } e^{-\hat\ep^2 n},
\end{align}
which completes the proof.
 \end{proof}
\end{proposition}

\begin{remark}
	By Chebychev's inequality it follows immediately from \eqref{eq:bound-exp-spiked} that
	\begin{align}
		\label{eq:bound-prob-spiked}
		 \bP\set{\npart ([2\sqrt\sigma +\ep,\infty))\geq 1} \leq \frac{C_\ep\sqrt {8\sigma n}}{\ep } e^{-\frac{\ep^2 n}{4\sigma}}.
	\end{align}
\end{remark}


    \section{Random matrix applications part II: Spiked Dyson Brownian Motion}
    \label{sec:spiked}

In the previous section, we introduced Dyson Brownian motion as the eigenvalues of the GUE process or as non-intersecting Brownian motions.  We now consider a spiked version of this as follows: in the non-intersecting BM picture, start one walker at $c$ instead of 0. Equivalently, perturb the  GUE process by adding a rank-one matrix of the form $c v_1v_1^T$. Here, $v_1$ is any unit vector (and by rotation invariance of the GUE it doesn't matter which one).

    As before, we will scale $t$ to the order of $\frac1n$ to obtain non-trivial asymptotics. On that order, up to time $\frac {c^2}n$, we see a different behaviour of the top particle. After that time, the lower walkers `catch up' to the top particle and we again get the Tracy--Widom limit. See Figure \ref{fig:spiked-DBM} for an illustration.

    \begin{theorem}
    \label{thm:spiked-limit}
        Set $t=\frac{\sigma}n$. Then, as $n\to\infty$, the distribution of the top particle of the spiked DBM process behaves according to the value of $\sigma$ as follows:
        \begin{enumerate}
            \item If $\sigma>c^2$, then as in the unspiked case, the top particle converges to $2\sqrt\sigma$. The fluctuations around $2\sqrt\sigma$ are also of order $n^{-2/3}$ and follow the Tracy--Widom GUE distribution $F_2$:
                \begin{align}\label{eq:convergence-spiked-subcritical}
    \lim_{n\to\infty} 
\mathbb{P}\!\left(  n^{2/3} \sigma^{-1/2} \rb{X_1\rb{\frac{\sigma}n  } - 2\sqrt{\sigma}}\le s \right) = F_2(s).
    \end{align}
            \item If $\sigma<c^2$ then the top particle converges weakly  to $c+\frac\sigma c$ with Gaussian fluctuations:
            \begin{align*}
                n^{1/2} \rb{X_1\rb{\frac\sigma n}- \rb{c+\frac\sigma c}}\Longrightarrow \mathcal{N}\rb{0,\frac{\sigma}{c^2}\rb{c^2-\sigma}}
            \end{align*}
            
        \end{enumerate}
    \end{theorem}

    \begin{figure}
        \centering
        \includegraphics[width=0.75\linewidth]{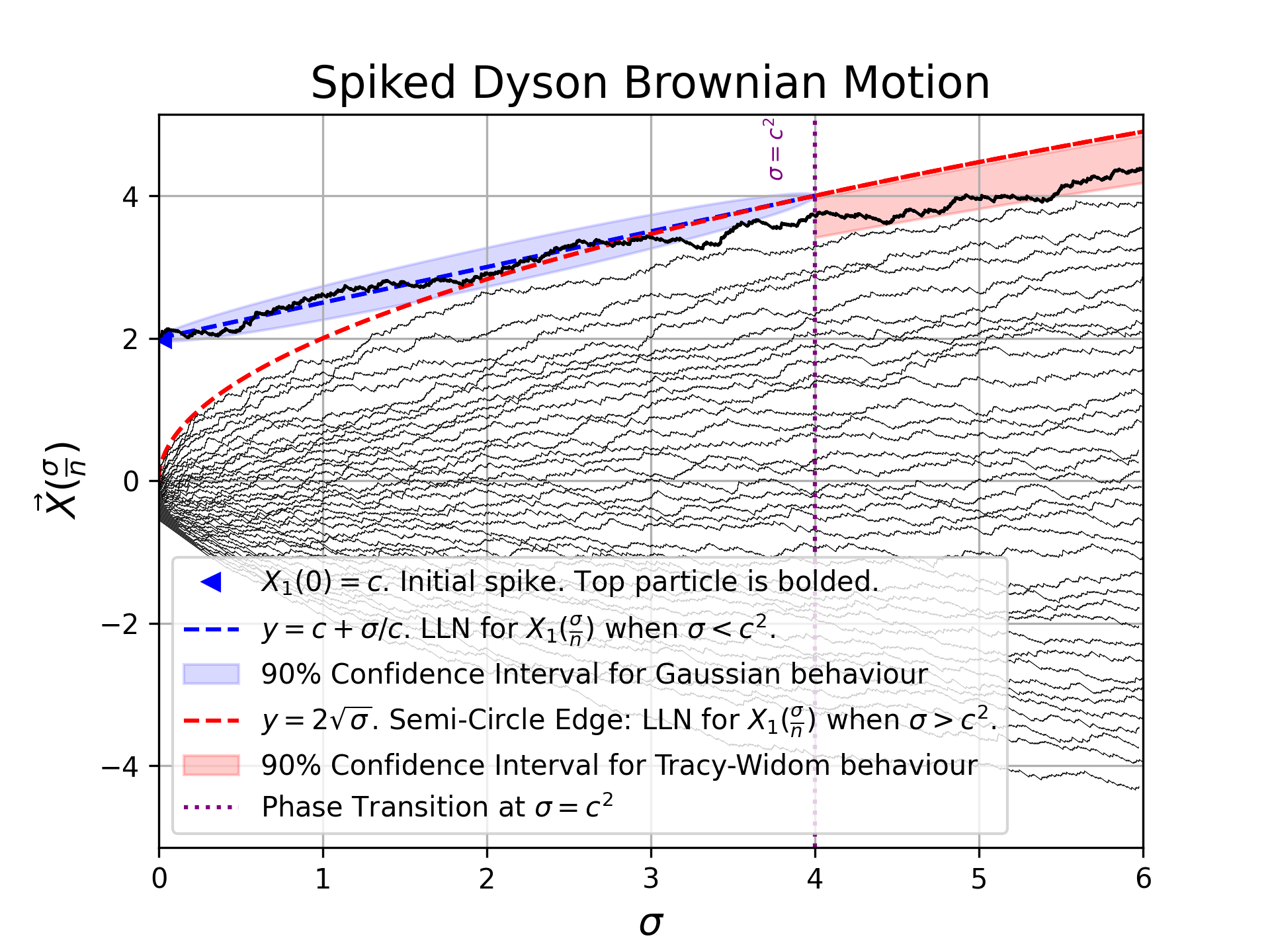}
        \caption[A sample of the Spiked Dyson Brownian motion process]{A sample of the Spiked Dyson Brownian motion process when $n=36$ and intitial spike strength $c=2$. The trajectory of the top particle, $X_1(\sigma/n)$ is bolded. The two regimes, depending on if $\sigma < c^2$ or $\sigma > c^2$, given in Theorem \ref{thm:spiked-limit} are illustrated. Note that the convergence of Theorem applies as $n \to \infty$; the approximation is already good at $n=36$ far away from the phase transition point, but may not be good a approximation near the phase transition.  }
        \label{fig:spiked-DBM}
    \end{figure}
    
\begin{remark}
As can be seen in  Figure \ref{fig:spiked-DBM}, the LLN curves of Theorem \ref{thm:spiked-limit} have an interesting tangency property. The LLN line $y=c+\sigma/c$ from the $\sigma < c$ phase is exactly tangent to the parabola LLN $y=2\sqrt{\sigma}$ from the $\sigma > c$ phase. This tangency happens exactly at the point $\sigma=c^2$ where the phase transition happens. Based on energy minimization considerations, one might intuit that this must be the case for the LLN curves: imagine a rubber band that is stretched from the original point $X_1(0)=c$ to some large final point on the curve $y=2\sqrt{\sigma}$. (A rubber band and Brownian motion both minimize a quadratic ``energy'', so the analogy is apt.) By the laws of physics, the rubber band must be tangent at the point it meets the curve (otherwise the tension in the rubber band would pull it somewhere else). This heuristic argument to explain the tangency was originally suggested by B{\'a}lint Vir{\'a}g 
\end{remark}

    As for the Dyson Brownian motion process above we can show that the spiked Dyson BM is a biorthogonal ensemble and then appeal to Theorem \ref{thm:bi-orthog-thm}. The computations can be found in Appendix \ref{app:proof-spiked-DPP}. 
    
    \begin{proposition}
    \label{prop:spiked-DPP}
        The spiked DBM process is a DPP with kernel $\spikedkernel$ where 
        \begin{align*}
            \spikedkernel_{n,t}(x,y) & = \dysonkernel_{n,t}(x,y)+
            \frac{\wt{H}_n(y,t)}{c^{n}} \sqrt{\vp_t(y)}\ab{ \frac{\phi_t(x-c)}{\sqrt{\vp_t(x)}} - \sum_{j=0}^{n-1}\rb{\frac ct}^j \frac1{j!} \wt H_j(x,t)\sqrt{\vp_t(x)}}
        \end{align*}


    \end{proposition}

    \begin{remark}
    Note that by using Proposition \ref{prop:generating-function}, we may write an alternate form for $K_{N}^{c}(x,y)$:
\begin{align}
    \label{eq:spiked-DPP-alt}
 \conjspikedkernel_{n,t}(x,y)=\dysonkernel_{n,t}(x,y)+\sqrt{\phi_t(y)} \frac{\wt H_n(y,t)}{c^{n}} 
\sum_{k=n}^{\infty} \sqrt{\phi_{t}(x)} \wt H_k(t,x)
 \left(\frac{c}{t}\right)^{k}\frac{1}{k!}
\end{align}.
\end{remark}

Recall that $t=\frac\sigma n$ for some $\sigma>0$. We will consider separately the cases $c<\sqrt\sigma$ and $c>\sqrt\sigma$.

\subsection{Analysis of the kernel \texorpdfstring{when $\sigma>c^2$}{for long time}}

Here we prove the first part of Theorem \ref{thm:spiked-limit}. Following the discussion preceding
\eqref{eq:rescdysonkernel}, the probability on the left hand side of \eqref{eq:convergence-spiked-subcritical} can be written as the Fredholm determinant of the kernel
\begin{align}
\label{eq:spiked-rescaled}
            n^{-2/3}\sigma^{-1/2} \conjspikedkernel_{n,\sigma/n} \rb{2\sqrt \sigma+\frac u{n^{2/3}\sigma^{1/2}},2\sqrt \sigma+\frac v{n^{2/3}\sigma^{1/2}}}.
\end{align}
We establish a uniform bound on the difference between the spiked and DBM kernels. From this bound it follows directly that the limit in \eqref{eq:spiked-rescaled} is dominated by the DBM kernel, which concludes the proof of the first part of Theorem \ref{thm:spiked-limit}.

\begin{proposition}
\label{prop:bound-subcritical}
       For $\sigma>c^2$,
\begin{align*}
   \sup_{x\in\bR} \sup_{y\in\bR} \abs{\conjspikedkernel_{n,t}(x,y)-\dysonkernel_{n,t}(x,y)} \leq \frac{\sqrt{e}}{(2\pi)^{3/4}(\sqrt\sigma-c)} n^{1/2}.
\end{align*}
\end{proposition}

\begin{proof}
Using \eqref{eq:spiked-DPP-alt} and the fact that $\phi_t(x)=\frac1{\sqrt t} \phi_1(x/\sqrt{t})$,
    \begin{align*}
    \conjspikedkernel_{n,t}(x,y) 
 & = \dysonkernel_{n,t}(x,y)+ \frac{t^{(n-1)/2}}{c^n}  \sqrt{n!} \Psi_n\rb{\frac y{\sqrt t}}  
\sum_{k=n}^{\infty} \frac1{\sqrt{k!}} t^{k/2} 
\Psi_k\rb{\frac x{\sqrt t}}
 \left(\frac{c}{t}\right)^{k} \\
 & = \dysonkernel_{n,t}(x,y)+ t^{-1/2}  \Psi_n\rb{\frac y{\sqrt t}}  
\sum_{k=n}^{\infty} 
\sqrt{\frac{n!}{k!}} t^{(n-k)/2} 
\Psi_k\rb{\frac x{\sqrt t}}c^{k-n}
\end{align*}
Set now $t=\frac\sigma n$. 
We  apply Proposition \ref{prop:Cramer} as well as Remark \ref{rmk:Stirling-bound} to get
\begin{align*}
        \abs{\conjspikedkernel_{n,t}(x,y)-\dysonkernel_{n,t}(x,y)} & \leq \sqrt{\frac{e}{\sqrt{2\pi}}}\cdot \frac1{\sqrt{2\pi}} \sqrt{\frac n\sigma}  
\sum_{k=n}^{\infty} 
e^{(k-n)/2}
\sqrt{\frac{\sqrt{n} n^n}{\sqrt{k} k^k}} \rb{\frac n\sigma}^{(k-n)/2} 
 c^{k-n}\\
 & = \frac{\sqrt{e}}{(2\pi)^{3/4}}\sqrt{\frac n\sigma}   
\sum_{k=n}^{\infty}   e^{(k-n)/2}
\sqrt{\frac{\sqrt{n} }{\sqrt{k} }}  \rb{\frac{n} k}^{k/2} 
 \rb{\frac c{\sqrt\sigma}}^{k-n}\\
 & =\frac{\sqrt{e}}{(2\pi)^{3/4}}
 \sqrt{\frac n\sigma}   
\sum_{\ell=0}^{\infty}   e^{\ell/2}
 \rb{\frac{n} {n+\ell}}^{(n+\ell+1/2)/2} 
 \rb{\frac c{\sqrt\sigma}}^{\ell}\\
 & \leq \frac{\sqrt{e}}{(2\pi)^{3/4}}
 \sqrt{\frac n\sigma}   
\sum_{\ell=0}^{\infty} \rb{\frac c{\sqrt\sigma}}^{\ell} =\frac{\sqrt{e}}{(2\pi)^{3/4}}
 \sqrt{n} 
 \frac1{\sqrt\sigma-c}
\end{align*}
as claimed.
For the last inequality, we have used the  inequality $(1-\ell/(n+\ell))^{(n+\ell+1/2)/2}\leq e^{-\ell/2}$.
\end{proof}



\subsection{Analysis of the kernel  \texorpdfstring{when $\sigma<c^2$}{for short time}}

We now turn to the second part of Theorem \ref{thm:spiked-limit}, which governs the `early' regime where $t=\frac\sigma n$ for $\sigma<c^2$.
The key ingredient is the following proposition, which constrains the location of the top particle to an area of size proportional to $n^{-1/2}$ around $\gamma$.


\begin{proposition}
\label{prop:particles-spiked}
    Let $\gamma=c+\frac\sigma c>2\sqrt\sigma$ and suppose that $\sigma<c^2$. For any interval $A$ let $\npartspiked(A)$ be the number of particles of the $n$-particle spiked DBM process at time $\frac\sigma n$  in A. Then the following three statements hold for any $\ep\in (0,\frac{\gamma-2\sqrt\sigma}4)$:
    \begin{enumerate}
        \item The probability of finding a particle to the right of $\gamma+\ep$ decays exponentially: there exists $C>0$ such that for all $n\in \bN$
        \begin{align*}
            \bP\set{\npartspiked(\rb{\gamma+\ep,\infty)}\geq 1}\leq Ce^{-(\ep^2 \wedge \delta/(4\sigma)) n}
        \end{align*}
        where $\delta = \frac{\gamma-2\sqrt\sigma}2>0$.
        \item $\bE\ab{\npartspiked(\gamma-\ep,\infty)}\to 1$ as $n\to\infty$.
        \item $\bE(\npartspiked(\gamma+un^{-1/2},\gamma+vn^{-1/2}))\to F_\Phi(\frac{c^2}{\sigma(c^2-\sigma)}\,v) - 
        F_\Phi(\frac{c^2}{\sigma(c^2-\sigma)}\,u)$,
        where $F_\Phi$ is the cumulative distribution function of a standard Gaussian.
    \end{enumerate}
\end{proposition}

\begin{proof}
Note that we have $\frac\gamma{2\sqrt\sigma}=\frac{(c^2+\sigma)/2}{c\sqrt\sigma}$, which is the arithmetic mean divided by the geometric mean of the quantities $\sigma$ and $c^2$. Therefore $\gamma>2\sqrt\sigma$ by the AM-GM inequality.

We further note that for any $u\in\bR$,
\begin{align}
\label{eq:sum-npartspiked}
    \bE\ab{\npartspiked(u,\infty)} & = 
    \quad \bE\ab{\npart(\rb{u,\infty)}} + 
    c^{-n} \int_{u}^\infty \wt H_n\rb{x,\frac\sigma n} \phi_t(x-c) \, dx  \\
    \notag
    & \quad - \frac1{c^{n}} \sum_{j=0}^{n-1} \int_{u}^\infty \rb{\frac{cn}\sigma}^j \frac1{j!} \wt H_n\rb{x,\frac\sigma n} \wt H_j\rb{x,\frac\sigma n} \phi_t(x)\, \d x
\end{align}
We will estimate each of the three summands in \eqref{eq:sum-npartspiked} in turn. For the first and the third term we will set $u=\gamma-\ep$ and show that they still go to zero. For the middle term, we will see that the integral concentrates exactly around $u=\gamma$, and therefore the integral from $\gamma+\ep$ on will be negligible.

We first appeal to Proposition \ref{prop:decay-points-unspiked}. By the assumption on $\epsilon$ we have 
$\gamma-\ep > 2\sqrt\sigma+\delta $ where $\delta = \frac{\gamma-2\sqrt\sigma}2>0$. Therefore we can bound the first term by $\frac {C_\delta}\delta \sqrt{8\sigma n} e^{-\delta^2n/4\sigma}$ .
To bound the second term,
\begin{align*}
    c^{-n} \int_{\gamma+\ep}^\infty \wt H_n\rb{x,\frac\sigma n} \phi_t(x-c) \, dx & = c^{-n} \sqrt{\frac n{2\pi\sigma}} \rb{\frac\sigma n}^{n/2} \int_{\gamma+\ep}^\infty  \fH_n\rb{\frac{x\sqrt n}{\sqrt\sigma}} e^{-n(x-c)^2/(2\sigma)}\, \d x\\
    & = 2\sqrt{\frac n{2\pi}} \rb{\frac{\sqrt\sigma}c}^n n^{n/2} \int_{(\gamma+\ep)/2{\sqrt\sigma}}^\infty \fH_n\rb{2\sqrt n a} e^{-n(2\sqrt\sigma a -c)^2/2\sigma}\, \d a\\
    & \leq \sqrt{\frac n\pi} \rb{\frac{\sqrt\sigma}c}^n \int_{(\gamma+\ep)/2{\sqrt\sigma}}^\infty \frac1{\sqrt{1-a^2+a\sqrt{a^2-1}}}\, e^{nR(a)}\, \d a,
\end{align*}
where $R(a)=-a^2-a\sqrt{a-1}+\arccosh(a)+\frac{2ac}{\sqrt\sigma}-\frac12(1+\frac{c^2}\sigma)$ and we have used Proposition \ref{prop:bound-uniform-H}. Note that the first term in the integrand is uniformly bounded above by a constant $\wt C$ that only depends on $\gamma$. Moreover $R'(a)=-2a-2\sqrt{a^2-1}+\frac{2c}{\sqrt\sigma}$ which has a unique zero at $a=\frac{\gamma}{2\sqrt\sigma}>1$ and $R''(a)=-2-\frac{2a}{\sqrt{a^2-1}}<-2$ for $a>1$. Therefore,
\begin{align*}
     c^{-n} \int_{\gamma+\ep}^\infty \wt H_n\rb{x,\frac\sigma n} \phi_t(x-c) \, dx &\leq \wt C \sqrt{\frac n\pi}\rb{\frac{\sqrt\sigma}c}^n \int_{(\gamma+\ep)/2{\sqrt\sigma}}^\infty e^{-n(a-\gamma/2\sqrt\sigma)^2}\, \d a\\
     & \leq \frac{\wt C}\ep \rb{\frac{\sqrt\sigma}c}^n \sqrt{\frac n\pi} e^{-n\ep^2}
\end{align*}
where  we have applied the Mills inequality (Lemma \ref{lem:mills}).

Finally we estimate the third term for $u=\gamma-\ep$. This goes very similarly to the first term (see Section \ref{sec:proof-third-term} in the Appendix) and we obtain
\begin{align}
\label{eq:third-term-spiked}
    \sum_{j=0}^{n-1}  c^{j-n} \rb{\frac n\sigma}^{(j-n)/2} \sqrt{\frac n{\sigma}} &  \int_{\gamma-\ep}^\infty \abs{\Psi_n\rb{\frac {x\sqrt n}{\sqrt\sigma}}} \abs{\Psi_j\rb{\frac {x\sqrt n}{\sqrt\sigma}}}\, dx \leq C e^{-n(1+\frac{\delta}{2\sqrt {\sigma}})/2}
\end{align}

Taking these points together we obtain $\bE\ab{\npartspiked(\gamma+\epsilon)}\leq Ce^{-(\ep^2 \wedge \delta/(4\sigma)) n}$. The same bound for the probability that $\npartspiked(\gamma+\epsilon)$ is no bigger than 1 now follows from Chebychev's inequality and we have proved the first part of the proposition.

     To prove the second part, recall that we have given exponential bounds on the first and the third summand of \eqref{eq:sum-npartspiked} even when $u=\gamma-\ep$. In other words,
     \begin{align*}
         \bP\rb{\npartspiked[\gamma-\ep,\infty)} & = \int_{\gamma-\ep}^\infty \frac{\wt H_n(x,{\sigma/n})}{c^n} \vp_{\sigma/n} (x-c) \, dx\  \rb{1+o(1)}
     \end{align*}
    In order to analyse the asymptotics of the latter expression, we will apply Laplace asymptotics, Proposition \ref{prop:laplace}.  To do so, let us re-write the left-hand side:
\begin{align}
   I_n:= c^{-n}\intop_{\gamma-\ep}^\infty \wt {H}\rb{x,\frac\sigma n} \vp_{\sigma/n}(x-c)\, dx & = \rb{\frac{\sqrt\sigma}{c}}^{n} n^{-n/2} \intop_{\gamma-\ep}^\infty \fH_n \rb{x\sqrt{\frac n\sigma}}\vp_{\sigma/n}(x-c)\, dx.
\end{align}
Take now $a=\cosh(\th)>1$. From Theorem \ref{thm:main-asymptotics} we have $n^{-n/2} \fH_n(2a\sqrt n)=e^{n S(a)} \kappa(a) \rb{1+o(1/n)}$ where
\begin{align*}
    S(a) & = 
    a^2 - a\sqrt{a^2-1}+\arccosh(a) \quad \quad \text{and}\quad\quad
      \kappa(\th)  = 
   \frac{e^{\theta/2}}{\sqrt{2\sinh(\th)}}.
\end{align*}
Thus,
\begin{align}
    I_n & = \rb{\frac{\sqrt\sigma}{c}}^{n}\frac {\sqrt n}{\sqrt{2\pi
    \sigma}} \intop_{\gamma-\ep}^\infty  \exp\set{n \ab{S\rb{\frac{x}{2\sqrt\sigma} }-\frac{(x-c)^2}{2\sigma}}}\kappa\rb{\frac{x}{2\sqrt\sigma}}\, dx \ \rb{1+o(1)}.
\end{align}
In order to apply Proposition \ref{prop:laplace} we define $f(x) = S\rb{\frac{x}{2\sqrt\sigma} }-\frac{(x-c)^2}{2\sigma}$. A direct calculation (see Appendix \ref{subsec:proof-critical-point}) shows that on $[\gamma-\ep,\infty)$
\begin{align}
\label{eq:spiked-critical-point}
f'(x) =0 \ \Leftrightarrow\ x=\gamma.
\end{align}
Thus by Proposition \ref{prop:laplace},
\begin{align*}
    I_n 
    & = \rb{\frac{\sqrt\sigma}c}^n \frac {1}{\sqrt{
    \sigma}} e^{n f(\gamma)}  \,\kappa\rb{ \frac{\gamma}{2\sqrt\sigma}}\sqrt{\frac{1}{\abs{f''(\gamma)}}}\, (1+o(1))
\end{align*}
A direct computation yields $f''(x)=\frac1{4\sigma} S''(\frac x{2\sqrt\sigma})-\frac1\sigma$. After some algebra, we obtain $S''(\frac\gamma{2\sigma})=-\frac{4\sigma}{c^2-\sigma}$ and therefore $f''(\gamma)=-\frac{c^2}{\sigma(c^2-\sigma)}<0$.

 Next, we compute the $\kappa$ term. Write $a=\frac{\gamma}{2\sqrt\sigma}$ and note that $a = \cosh(\ln(\lambda))$ where $\lambda=\frac{c}{\sqrt\sigma}>1$, so that $\th := \ln(\lambda)=\ln\rb{\frac{c}{\sqrt\sigma}}$. This implies $e^{\th/2}=\sqrt{\frac{c}{\sqrt{\sigma}}}$ and $\sinh(\th)=\frac12\rb{\frac c{\sqrt\sigma}-\frac{\sqrt\sigma}c}=\frac{c^2-\sigma}{2c\sqrt\sigma}$.
 Plugging this into the formula for $\kappa$ we obtain
 \begin{align*}
     \kappa\rb{ \frac{\gamma}{2\sqrt\sigma}} = \frac{\sqrt{c/\sqrt\sigma}}{\sqrt{2\frac{c^2-\sigma}{2c\sqrt\sigma}}} = \frac c{\sqrt{c^2-\sigma}}.
 \end{align*}
 Finally we need to compute $f\rb{\gamma}$: 
 \begin{align}
     f(\gamma) &= S\rb{a }-\frac{(\gamma-c)^2}{2\sigma}
     \\
     &=\th+\frac12 e^{-2\th\ast}-\frac{(\gamma-c)^2}{2\sigma}\\
     & = \ln\rb{\lambda}+\frac1{2\lambda^2}-\frac{(\gamma-c)^2}{2\sigma}\\
     & =\ln\rb{\frac c{\sqrt{\sigma}}}+\frac\sigma{2c^2}-\frac{(\sigma/c)^2}{2\sigma} =\ln\rb{\frac c{\sqrt{\sigma}}},
 \end{align}
 so $e^{nf(\gamma)}$ and $\rb{\frac{\sqrt\sigma}c}^n$ cancel out exactly. Putting everything together,
 \begin{align*}
     I_n & =  \frac {1}{\sqrt{
    \sigma}}  \,
    \frac c{\sqrt{c^2-\sigma}}
    \sqrt{\frac{1}{\frac{c^2}{\sigma(c^2-\sigma)}}} \rb{1+o(1)}
     =  \frac {1}{\sqrt{
    \sigma}}  \,
   \frac c{\sqrt{c^2-\sigma}}
    \sqrt{\frac{{\sigma(c^2-\sigma)}}{c^2}}\rb{1+o(1)}
    = \rb{1+o(1)}
 \end{align*}
 as claimed.

To prove the third point of the proposition, we use the discussion around \eqref{eq:zoomed_in_laplace_approx_eq},  which zooms in around the critical point $\gamma$. From this, we see that not only the law of large numbers result that  $\bP\rb{\npartspiked(\gamma-\ep,\infty]=1} \to 1$ but also the central limit theorem type result that the expectation $\bE\rb{\npartspiked(\gamma+\frac{u}{\sqrt{n}},\gamma+\frac{v}{\sqrt{n}}]} $ converges to $F_\Phi(\abs{f''(\gamma)}v)-F_\Phi(\abs{f''(\gamma)}u)$. where $F_\Phi$
is the cumulative distribution function for a standard Gaussian random variable. Note that multiplying by $\abs{f''(\gamma)}$ inside $F_\Phi$
corresponds to changing the variance from 1 to $\frac1{\abs{f''(\gamma)}}$
\end{proof}

\begin{remark}
    The above argument can be repeated after replacing the indicator $1_{[\gamma-\ep,\infty)}$ by a general bounded continuous function in order to prove weak convergence of $c^{-n}\wt{H}\rb{x,\frac{\sigma}n} \vp_{\sigma/n}(x-c)$ to $\delta_\gamma$.
\end{remark}

\section{Fourier identities, Gaussian integration-by-parts, and Edgeworth expansions }\label{sec:fourier}
		
In this section we derive some alternate integral formulas for the Hermite polynomials related to Fourier transforms that are proven starting from the Gaussian integral \eqref{eq:formula} by a simple change-of-variables. In particular Proposition \ref{prop:IBP-not-integrated} gives another integral formula for the Hermite polynomials $H_{\vec {n}}$ and we establish a Gaussian integration by parts formula, Theorem \ref{thm:int-by-parts}.

In 1 dimension, these formulas establish the well known facts that the Fourier transform of the product $\fH_n(x)e^{-\frac{1}{2}x^2}$ is  $(\imath u)^n e^{-\frac{1}{2}u^2}$ (up to constants depending on how one exactly defines the Fourier transform).

Since multiplication by $\imath u$ in Fourier space corresponds to differentiation, a corollary to this is that $\fH_n(x)e^{-\frac{1}{2}x^2} = \left( \frac{\d}{\d x} \right)^n \left( e^{-\frac{1}{2}x^2} \right)$, thereby recovering the classical derivative definition of the Hermite polynomials. This  allows one to do integration-by-parts $n$ times when integrating against a Gaussian variable to establish the identity $\mathbb{E}[ \left( (\frac{\d}{\d x})^n f \right) (Z)] = \mathbb{E}[\fH_n(Z)f(Z)]$ for any smooth enough function $f$.    

Lastly, we show how these Fourier type identities explain why the Hermite polynomials naturally appear in the Edgeworth expansion for sums of random variables, which are lower order corrections to the central limit theorem. See Section \ref{sec:1d-edgeworth} for the classic scalar Edgeworth expansion and Section \ref{sec:multi-d-edgeworth} for its multidimensional generalization.

Note that although all the formula really are related to Fourier transforms, we state and prove everything as straight integrals involving $e^{\imath xu}$ instead of defining $\cF$. This avoids the awkward situation where different authors have different normalizing constants/factors in their definition of the Fourier transform that makes it harder to translate between works. 

\subsection{Fourier formula for the Hermite polynomials}

	We begin by recalling some basic facts about the centred multidimensional Gaussian distribution with covariance $V$, whose density we denote by $\vp_V$:
	\begin{align}
		\label{eq:density-gaussian}
		\vp_V(x) := \sqrt{\frac{\det(V)}{(2\pi)^d}}\ \exp\set{-\frac12\, x^T V^{-1}x}.
	\end{align}
	We first recall an integral formula for $\vp_V$ based on the Fourier transform. While the results are standard, we give the proofs for convenience. 	\begin{lemma}
		\label{lem:formulaCF}
		For any $x\in\bR^d$ we have
		\begin{align}
			\vp_V(x) = \frac1{(2\pi)^d}\, \int_{\bR^d} \exp\set{i\,x\cdot z-\frac12\,z^TV z}\, dz
		\end{align}
	\end{lemma}
	\begin{proof}
		The Fourier transform of $\vp_{V^{-1}}$ is $x\longmapsto\exp\set{-\frac12 x^TVx}$ and therefore
		\begin{align*}
			\exp\set{-\frac12\,x^T V x}=\bE_{Z\sim N(0,V^{-1})} \ab{e^{i x\cdot Z}}
			= \frac1{\sqrt{(2\pi)^d \det(V)}}\, \int_{\bR^d} \exp\set{ix\cdot z - \frac12 z^T V z}\,dz
		\end{align*}
		Multiplying both sides by $\sqrt{\frac{\det(V)}{(2\pi)^d}}$ yields the result.
	\end{proof}
	
		\begin{lemma}
		\label{lem:shift-density}
		For any $a,z\in\bC^d$,
		\begin{align}
			\label{eq:shift-density}
			a\cdot z-\frac12 z^T Vz =  \frac12
			a^T V^{-1} a - \frac12(z-V^{-1}a)^T V (z-V^{-1}a)
		\end{align}
	\end{lemma}
	\begin{proof}
		This follows by rearranging the following identity
		\begin{align}
			\label{eq:Va-written-out}
			-\frac12 (z-V^{-1}a)^T V (z-V^{-1}a) & = a\cdot z - \frac12 z^T Vz - \frac12 a^T V^{-1} a
		\end{align}
		which can be verified by direct computation. 
	\end{proof}
    From this lemma, the following proposition follows directly by taking $\Sigma=V^{-1}$.
	\begin{proposition}
	    \label{prop:gaussian-shift-mean}
        For any $a, z\in\bC^d$ and a covariance matrix $\Sigma$,
        \begin{align*}
             \frac{\vp_\Sigma(z-\Sigma a)}{\vp_\Sigma (z)} & = e^{a\cdot z - \frac12 a^T \Sigma a}.
        \end{align*}
	\end{proposition}

    The following result corresponds to equation (6.5) in \cite{Sobel-tech-1963}.
	\begin{proposition}
		\label{prop:IBP-not-integrated}
		For $n\in\bN^d$ and $x\in\bR^d$ we have
		\begin{align}\label{eq:H_is_iu_to_the_n}
			H_{\vec n}(\vec{x};V) & = \frac1{(2\pi)^d} \frac{1}{\vp_V(\vec{x})} \int_{\bR^d} (\imath \vec{u})^{\vec{n}}\, \exp\set{-\frac12 \vec{u}^T V \vec{u}-\imath \,\vec{x}\cdot \vec{u}} \, d\,\vec{u}.
		\end{align}
	where $\vec{u}^{\vec{n}} := \prod_{j=1}^d u_j^{n_j}$ is shorthand for the product of powers.
    \begin{proof}
			We use \eqref{eq:shift-density} applied to $a=ix$ to get, recalling the notation $\vec x^{\vec n} := \prod_a x_a^{n_a}$,
			\begin{align*}
				H_{\vec{n}}(\vec{x};V) & = \frac{\det(V^{-1})}{(2\pi)^{d/2}}\, \int_{\bR^d} \rb{V^{-1}\vec x - \imath \vec z}^{\vec {n}} \exp\set{-\frac12 \vec z^T V \vec z}\, \d \vec z\\
				& = \frac{\det(V^{-1})}{(2\pi)^{d/2}}\, \int_{\bR^d} \rb{\imath \vec u}^{\vec {n}} \exp\set{-\frac12 ( \vec u-\imath V^{-1}\vec x)^T V (\vec u-\imath V^{-1}\vec x)}\, \d \vec u\\
				& = \frac{\det(V^{-1})}{(2\pi)^{d/2}}\, \int_{\bR^d} \rb{\imath \vec u}^{\vec {n}} \exp\set{-\frac12 \imath \vec x\cdot \vec u -\frac12 \vec u^TV \vec u + \frac12\vec u^T V^{-1} \vec u}\, \d \vec u.
			\end{align*}
			The last term in the exponential multiplied by the prefactor $\det(V^{-1})$ is exactly $\frac{(2\pi)^{-d/2}}{\vp_V(x)}$, which completes the proof.
		\end{proof}
		This identity can also be stated in terms of derivatives. We define the differential operator $\partial_{\vec{n}}$ of partial derivatives as follows: let $\vec{n}\in\bN^d$ and $f\colon\bR^d\longrightarrow\bR$. Then,
	\begin{align}
		\label{eq:defDerivative}
		\partial_{\vec {n}} f(x) & := \left(\frac{\partial^{n_1}}{\partial x_1^{n_1}} \ldots\frac{\partial^{n_d}}{\partial x_d^{n_d}}\right) f(x).
	\end{align}
    \begin{corollary}
        Let $\vec n \in\bN^d$ and let $n = n_1+\ldots+n_d$. The polynomials $H_{\vec{n}}(\vec{x};V)$  defined as in Definition \ref{def:dual-H} satisfy
    $$H_{\vec{n}}(\vec{x};V) \vp_V(\vec{x}) = (-1)^{n} \partial_{\vec {n}}  \varphi_V(\vec{x}) $$
    \end{corollary}
    \begin{proof}
The identity follows from the result of Proposition \ref{prop:IBP-not-integrated} by noting that the Fourier transform converts differentiation $\frac{\d}{\d x_j}$ into multiplication by $\imath u_j$ in Fourier space. More precisely, by integration by parts to move the $n$ derivatives over, one sees that for suitable test functions $f$ we have the identity
$$\intop_{\mathbb{R}^d} e^{\imath \vec{u} \cdot \vec{x}} \partial_{\vec {n}}  f(\vec{x}) \d \vec u = (-1)^n \intop_{\mathbb{R}^d} (\imath \vec{u})^{\vec{n}} e^{\imath \vec{u} \cdot \vec{x}} f(\vec{x}) \d \vec u ,$$
where we recall the notation $\vec x^{\vec n} := \prod_a x_a^{n_a}$. Along with the identity from Lemma \ref{lem:formulaCF} for the  Gaussian density $\vp_V(\vec x)$, we see that the integral in the RHS of \eqref{eq:H_is_iu_to_the_n}  corresponds exactly to the derivatives of $\vp_V(\vec{x})$, establishing the result. 
    \end{proof}

    \end{proposition}

    \subsection{Gaussian integration by parts}
    
	\begin{theorem}
		\label{thm:int-by-parts}
		Recall the notation for derivatives $\partial_{\vec {n}} f(x) $ from \eqref{eq:defDerivative}. Let $\vec n \in\bN^d$ and $X\sim\cN(0,V)$. For a suitable test function $f\colon\bR^d\longrightarrow \bR$,
		\begin{align}
			\label{eq:int-by-parts}
			\bE_{X\sim\cN(0,V)} \ab{\partial_{\vec{n}} f(X) } = \bE_{X\sim\cN(0,V)}\ab{H_{\vec{n}}(X,V) f(X)}.
		\end{align}
	\end{theorem}
	\begin{proof}
		First, by Lemma \ref{lem:formulaCF} 
		then by repeated integration by parts, and finally by the change of variables $z\longmapsto -z$,
		\begin{align*}
			\bE \ab{\partial_{\vec{n}} f(X) } & = \frac1{(2\pi)^d}\int_{\bR^d} \int_{\bR^d} \rb{\partial_{\vec{n}} f}(x)\, \exp\set{ix\cdot z-\frac12 z^TVz}\, dx\, dz\\
			& = \frac{1}{(2\pi)^d} \int_{\bR^d} \int_{\bR^d} f(x) (-iz)^{\vec{n}}\, \exp\set{ix\cdot z-\frac12 z^TVz}\, dx\, dz\\
			& =  \int_{\bR^d} f(x) \ab{\frac{1}{(2\pi)^d}\int_{\bR^d}  (iz)^{\vec{n}}\, \exp\set{-ix\cdot z-\frac12 z^TVz}\, dz}\,dx \\
			& =  \int_{\bR^d} f(x) H_{\vec{n}}(x,V)\,\vp_V(x) \,dx,
		\end{align*}
		using Proposition \ref{prop:IBP-not-integrated}. This last expression is exactly the right-hand side of \eqref{eq:int-by-parts}.
	\end{proof}

\subsection{Cumulants and joint cumulants}

One common application of the Hermite polynomials is in computing higher order corrections to the central limit theorem. These expansion are given in terms of the cumulants of the random variables in question.

\begin{definition}
Let $X$ be a $\mathbb{R}$-valued random variable.
    The \emph{cumulants} $\{\kappa_j\}_{j=1}^\infty$ of $X$ are the coefficients of the characteristic function $\phi_X(u)$ of $X$,
    $$\phi_X(u) := \mathbb{E}\ab{e^{\imath u X}} = \exp\rb{\frac{\kappa_1}{1!}u + \frac{\kappa_2}{2!}u^2  + \frac{\kappa_3}{3!}u^3 + \ldots} $$
    or equivalently:
    $$\kappa_k := j! \ab{u^j} \ln\rb{ \phi_X(u) } = \left. \rb{\frac{\d}{\d u}}^j \ln \rb{\phi_X(u) } \right|_{u=0}$$ For vector valued random variables in $\mathbb{R}^d$, the cumulants are indexed by $d$  indices $\kappa_{k_1,k_2,\ldots,k_d}$  so that 
     $$\phi_{\vec{X}}\rb{\vec{u}} = \bE\ab{e^{\imath \vec{u} \cdot \vec{X} }}
     = \exp\left( \sum_{j_1,j_2,\ldots,j_d} \frac{\kappa_{j_1,j_2,\ldots,j_d}}{j_1!j_2!\cdots j_d!} u_1^{j_1} u_2^{j_2} \cdots {u_d}^{j_d}\right).
    $$
\end{definition}

\begin{proposition}
The cumulants can also be written as combinations of the moments of the underlying random variables. They key formula is the joint cumulants for a collection of $m$ random variables $\kappa(X_1,\ldots,X_m)\in \mathbb{R}$ given by summing over all partitions $\mathfrak{P}(\{1,\ldots,m\} $ as follows

\[
\kappa(X_1,\dots,X_n)
:=\sum_{\pi \in \mathfrak{P}(\{1,\ldots,n\})}
 (|\pi|-1)!\,(-1)^{|\pi|-1}
 \prod_{B\in\pi}\mathbb{E}\!\left[\prod_{i\in B}X_i\right].
\]

where the set of all partitions $\mathfrak{P}(\{1,\ldots,m\})$ is all the ways to divide this set into disjoint blocks

\[
\mathfrak{P}(\{1,\ldots,m\})
:= \left\{
\pi = \{B_1,\ldots,B_k\} \ \middle|\ 
\begin{array}{l}
B_1,\ldots,B_k \subseteq \{1,\ldots,m\},\ B_j \neq \varnothing, \\
B_i \cap B_j = \varnothing \ \text{for}\ i \neq j, \\
\bigcup_{j=1}^k B_j = \{1,\ldots,m\}
\end{array}
\right\}.
\]
Then, the cumulants for an $\mathbb{R}$-valued random variable are simply $\kappa_j = \kappa(\underbrace{X,X,\ldots,X}_{j\text{ times}})$  and for multi-dimensional random variables $\vec{X} = (X_1,\ldots,X_d)\in \mathbb{R}^d$, we have $$\kappa_{j_1,\ldots,j_d} = \kappa(\underbrace{X_1,\ldots,X_1}_{j_1 \text{ times}},\ldots,\underbrace{X_d,\ldots,X_d}_{j_d \text{ times}} )$$
\end{proposition}
\begin{example}\label{ex:4-cumulant}
Suppose we have four mean zero random variables $A,B,C,D \in \mathbb{R}$. Since the random variables are mean zero, the only moments that are non-zere in the joint cumulant formula are the ones where the four variables are divided either into one block of 4 or else two blocks of 2. Hence we have:
$$\kappa(A,B,C,D) = \mathbb{E}[ABCD] - \mathbb{E}[AB]\mathbb{E}[CD] - \mathbb{E}[AC]\mathbb{E}[BD] - \mathbb{E}[AD]\mathbb{E}[BC] $$
From this it follows that the 4th cumulant for a mean zero $\mathbb{R}$ random variable is given in terms of the 4th and 2nd moments, $\mu_p = \mathbb{E}[X^p]$, 
$$\kappa_4 = \mu_4 - 3\mu_2^2$$
For a 2D random variable $\vec{X}=(X_1,X_2)$  whose compoents are means zero with joint moments $\mu_{j_1,j_2}=\mathbb{E}[X_1^{j_1} X_2^{j_2}]$  we have that the cumulants of order 4, $\kappa_{j_1,j_2}$  for which $j_1+j_2 = 4$ are
\begin{align}
\kappa_{4,0} =\kappa(X_1,X_1,X_1,X_1) &= \mu_{4,0} - 3\mu_{2,0}^2 \\
\kappa_{3,1} = \kappa(X_1,X_1,X_1,X_2) &= \mu_{1,3} - 3\mu_{1,1}^2 \\
\kappa_{2,2} = \kappa(X_1,X_1,X_2,X_2) &= \mu_{2,2} - \mu_{0,2}\mu_{2,0} -  2\mu_{1,1}^2 
 \end{align}
 and analgouously for $\kappa_{1,3},\kappa_{0,4}$.
\end{example}

\subsection{Edgeworth expansion in one dimension}\label{sec:1d-edgeworth}

\begin{theorem}\label{thm:1d_edgeworth}

Let  $X_1,X_2,\ldots$  be an i.i.d. sequence of  $\mathbb{R}$-valued random variables which are mean zero $\mathbb{E}[X]=0$ and have variance $\si^2 = \mathbb{E}[ X^2]$.  Let $S_n = \frac{1}{\sqrt{n}} \sum_{j=1}^n X_j $  be the normalized sum of $n$ such random variables. Let $k \geq 3$ be the first value for which the higher cumulant $\kappa_{k}$ is non-zero and finite. Suppose also that the characteristic function $\phi_X$ has the property that $\intop_{-\infty}^\infty|\phi_X(u)|^v \d u < \infty$ for some power $v \geq 1$.   Then the leading order correction to the central limit theorem is of size $\frac{1}{\sqrt{n}^{k-2}}$ and given by the Hermite polynomial $\fH_{k}(x)$  as follows:

$$ \rho_{S_n}(x) =  \frac{1}{\sqrt{2\pi \sigma^2}}e^{-\frac{x^2}{2\sigma^2}} + \frac{1}{\sqrt{n}^{k-2}} \frac{\kappa_{k}}{k! \sigma^{k}} \fH_k\left(\frac{x}{\sigma} \right)\frac{1}{\sqrt{2\pi \sigma^2}}e^{-\frac{x^2}{2\sigma^2}}  + O\left( \frac{1}{\sqrt{n}^{k-1}} \right)$$

\end{theorem}
	\begin{proof}
    The proof is given by analysing the characteristic function for $S_n$ and using the inversion formula to get the density. First notice that since the first non-zero cumulant of $X$ is $\kappa_k$ , we have the characteristic function
    $$\phi_X(u)=\exp\left(-\frac{1}{2}\sigma^2 u^2 + \frac{\kappa_k}{k!} (\imath u)^k + \epsilon(u) \right) $$
where $\epsilon(u)=O(u^{k+1})$  as $u \to 0$ represents all the higher order terms beyond $u^k$ in the characteristic function. Notice that since $S_n$ is a sum of iid copies of $X$, the characteristic function for $S_n$  can be written in terms of the charactersitic function for $X$, namely  $\phi_{S_n}(u) = \phi_{X}(\frac{u}{\sqrt{n}})^n $ , so we have:
    \begin{align} \phi_{S_n}(u) &= \exp\left( -n \frac{1}{2}\sigma^2 (\frac{u}{\sqrt{n}})^2 + \frac{\kappa_k}{k!} n (\frac{u}{\sqrt{n}}) ^k + n \epsilon(\frac{u}{\sqrt{n}}) \right) \\ 
    &= \exp\left( - \frac{1}{2}\sigma^2 u^2 \right) \exp\left( \frac{1}{\sqrt{n}^{k-2}} \frac{\kappa_k}{k!}  (\imath u)^k  + n \epsilon\big(\frac{u}{\sqrt{n}}\big) \right)
    \end{align}
We now use the approximation $\exp(x)=1+x+O(x^2)$ as $x\to 0$ to approximate the exponential. To make our bounds precise, let $u$ range in some compact set. Since the error term $\epsilon(u)=O(u^{k+1})$ as $u \to 0$, we see that $n \epsilon(u/\sqrt{n}) = O(\frac{1}{\sqrt{n}^{k-1}})$ for $u$ in a compact set, so we get the approximation
\begin{equation}\label{eq:phi-approx} 
\phi_{S_n}(u) = \exp(-\half \sigma^2 u^2) + \frac{1}{\sqrt{n}^{k-2}}\frac{\kappa_k}{k!}(\imath u)^k \exp(-\half \sigma^2 u^2) + O\Big(\frac{1}{\sqrt{n}^{k-1}}\Big)
\end{equation}
    Finally, by the technical assumption that $\abs{\phi}^v$ is integrable, we can use the characteristic/Fourier inversion formula for the density $\rho_S(x) = \frac{1}{2\pi} \intop_{-\infty}^\infty e^{-\imath u x} \phi_{S_n}(u) \d u$ to recover the density from this (see \cite{feller1971probability2} Chapter XV.3 for a detailed proof of this inversion formula). By virtue of the formula for the Hermite polynomials, Proposition \ref{prop:IBP-not-integrated}, we can exactly integrate the terms of the approximation for $\phi_S$ in \eqref{eq:phi-approx}, as follows:
\begin{align} \frac{1}{2\pi} \intop_{-\infty}^\infty e^{-\imath u x} \left\{ \exp(-\half \sigma^2 u^2) \right\} \d u &= \frac{1}{\sqrt{2\pi\sigma^2}} e^{-\half \frac{x^2}{\sigma^2}} \\
\frac{1}{2\pi} \intop_{-\infty}^\infty e^{-\imath u x} \left\{ (\imath u)^k\exp(-\half \sigma^2 u^2) \right\} \d u &= \frac{1}{\sqrt{2\pi\sigma^2}} \frac{1}{\sigma^k} \fH_k\left(\frac{x}{\sigma}\right) e^{-\half \frac{x^2}{\sigma^2}} \\
\end{align}
    To complete the proof, it remains only to argue that the $O(\frac{1}{\sqrt{n}^{k-1}})$ error term in the approximation for $\phi_S$ in \eqref{eq:phi-approx} (which holds uniformly for $u$ in a compact set), translates to a $O(\frac{1}{\sqrt{n}^{k-1}})$ error in the density $\rho_S$ after integrating with the inversion formula. This error analysis is argued by using the techincal assumption that $\intop_{-\infty}^\infty |\phi_{S_n}(u)|^v \d u < \infty $ which allows one to approximate the inversion formula integral to arbitarty accuracey using $u$ in a compact set; we refer to \cite{feller1971probability2}  Chapter XVI.2 for the detailed verification of this error analysis.
    \end{proof}

\begin{remark}
Instead of obtaining just the leading order correction, one can obtain an entire series expansion in powers of $n$ known as the \emph{Edgeworth expansion}. This is achieved by simplifying using more terms in the approximation to the characteristic function (i.e. more cumulants) and also more terms in in the exponential approximation. For example, let us take $k=3$ so that the leading order correction is $\frac{\kappa_3}{3! \sqrt{n}} \fH_3(\frac{x}{\sigma})$.  By expanding up to order $u^4$ and using the higher order correction $\exp(x)=1+x+\frac{1}{2}x^2+O(x^3)$ , one will find that the next order correction is $\frac{1}{n}\left( \frac{\kappa_4}{4!\sigma^4} \fH_4(\frac{x}{\sigma}) + \frac{1}{2}\big(\frac{\kappa_3}{3!\sigma^3}\big)^2\fH_6(\frac{x}{\sigma}) \right)$ where the $\fH_6$ comes from the square term in the exponential which gives an $\frac{1}{2}(\frac{1}{3!}(\imath z)^3)^2$ in the series expansion for $\phi_S$. We refer to \cite{feller1971probability2} Chapter XVI for a detailed exposition.
\end{remark}

\begin{example}
The quantity $S_n = \frac{1}{\sqrt{n}} \sum_{j=1}^n W_i \sqrt{2} \operatorname{ReLU}\left( G_i \right)$ where $G_i,W_i$ are iid Gaussians and $\text{ReLU}(x)=\max\{x,0\}$ arises in the theory of neural networks, where $n $ is the number of neurons in a single layer of the network. See the workshop article \cite{nica2024improving} for more details on how this random variable is related to neural networks. $S_n$  is a sum of the underlying random variable $X = W\sqrt{2}\operatorname{ReLU}(G)$, which has mean zero, unit variance, $\mathbb{E}[S_n^2] = 1$ and third cumulant equal to zero. The first non-zero higher cumulant is when $k=4$, which has $\kappa_4 = \mathbb{E}[W^4 \cdot \sqrt{2}^4 \cdot \operatorname{ReLU}(G)^4] - 3 = 3\cdot 4\cdot \frac{1}{2} \cdot 3 - 3 = 15$. Hence Theorem \ref{thm:1d_edgeworth} gives the approximation 
\begin{equation}\label{eq:1d_edgeworth}\rho_{S_n}(x)= \frac{1}{\sqrt{2\pi}}e^{-\frac{x^2}{2}} + \frac{1}{n} \frac{15}{24} \fH_k\left( x \right)\frac{1}{\sqrt{2\pi}}e^{-\frac{x^2}{2}}  + O\left( \frac{1}{n^{3/2}} \right)\end{equation}
While this correction of size $1/n$ is small in a single layer, typical 
neural networks are comprised of stacking together many layers. Therefore, small errors in the approximation accumulate as the depth of the network grows, meaning that the Gaussian approximation will fail for deep neural networks of sufficient depth. The correction provided here is a theoretical foothold into a better understanding of these deep neural networks.

\begin{figure}[!ht]
    \centering
    \includegraphics[width=0.95\linewidth]{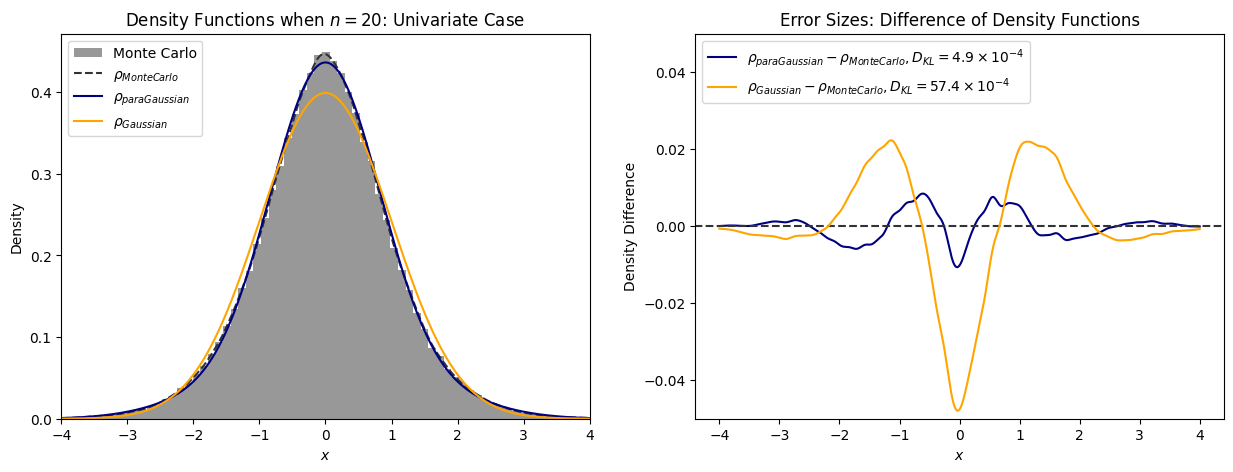}
    \caption[Example of an Edgeworth expansion correction to the Central Limit Theorem]{Example of the 1D neural network distribution when $n=20$. We compare the pure CLT Gaussian approximation $\rho_{Guassian}$ and the approximation from \eqref{eq:1d_edgeworth} which labelled here as $\rho_{\text{para-Gaussian}}$, the "para-Gaussian" distribution, to mean it is a near-Gaussian approximation that includes the effect of the $1/n$ sized 4th cumulant and $\fH_4$. We also compare both to Monte Carlo samples of the random variable. The left plot shows the raw distributions, while the right plot is the difference between the theoretical approximations and the Monte Carlo estimate for the density $\rho_{\text{MonteCarlo}}$. This example is also discussed in the workshop article \cite{nica2024improving}.  }
    \label{fig:1d-edgeworth}
\end{figure}
\end{example}

\subsection{Multi-dimensional Edgeworth expansion}\label{sec:multi-d-edgeworth}
The Edgeworth expansion technique works in multiple dimensions for random vectors using the multi-variable Hermite polynomials $\wt H(\vec{x},V)$. This technique does not seem to be widely cited in the literature; we only find an unpublished technical note \cite{Sobel-tech-1963} with the statment of this result. We believe a possible impediment to this application was dealing with the multi-variable Hermite polynomials, which are easier to work with by use of our methods in this article. In particular, the combinatorial method of Section \ref{sec:combinatorics}  allows one to quickly calculate the Hermite polynomials even when the covariance structure is complicated.

\begin{theorem}

Let  $\vec{X}_1,\vec{X}_2,\ldots$  be an i.i.d. sequence of  $\mathbb{R}^d$-valued random vectors with zero mean
and $d\times d$ co-variance matrix $\Sigma$, i.e. $\Sigma_{ab} = \mathbb{E}[ X_a X_b]$. Let 
$$\vec{S}_n = \frac{1}{\sqrt{n}} \sum_{j=1}^n \vec{X}_j $$
be the normalized sum of the first $n$ random vectors.  Define $k \geq 3$ to be the first value for which there is a non-zero cumulant of order $k$, i.e. there $\exists k_1,k_2,\ldots,k_d$ with $k_1+\ldots+k_d= k$ so that $\kappa_{k_1,k_2,\ldots,k_d}$ is non-zero.  Suppose also the characteristic function $\phi_X$ of $X_1$  is in $L^v(\mathbb{R}^d)$ for some $v \geq 1$.  Then the leading order correction to the central limit theorem is of size $\frac{1}{\sqrt{n}^{k-2}}$ and given by the multivariable Hermite polynomial of order $k$:

\begin{align*} \rho_{\vec{S}_n}(\vec{x}) =   \frac{1}{\sqrt{n}^{k-2}} \sum_{\substack{k_1,\ldots,k_d\\ k_1+k_2+\ldots+k_d = k}} \frac{\kappa_{k_1, k_2,\ldots,k_d}}{k_1!k_2!\cdots k_d!} \wt H_{k_1,\ldots,k_p}(\vec{x},\Sigma) \frac{1}{\sqrt{2\pi \det{\Sigma}}}e^{-\vec{x} \Sigma^{-1} \vec{x}}  \\ 
+\frac{1}{\sqrt{2\pi \det{\Sigma}}}e^{-\vec{x} \Sigma^{-1} \vec{x}} + O\left( \frac{1}{\sqrt{n}^{k-1}} \right)\end{align*}

\end{theorem}

The proof is analogous to the one-dimensional proof of Theorem \ref{thm:1d_edgeworth}, using the expansion of the characteristic polynomial as cumulants and then appealing to Proposition \ref{prop:IBP-not-integrated}

\begin{example}\label{ex:2d-edgeworth}
Let $\theta \in (0,\pi)$ be given and consider the two random variables, which are written in terms of collection of iid random $\mathcal{N}(0,1)$ Gaussian random variables, $W_i,G_i,Z_i$. 
\begin{align}
A_n &= \frac{1}{\sqrt{n}} \sum_{j=1}^n W_i \sqrt{2} \operatorname{ReLU}\Big( G_i \Big) \\
B_n &= \frac{1}{\sqrt{n}} \sum_{j=1}^n W_i \sqrt{2} \operatorname{ReLU}\Big( \cos(\theta) G_i + \sin(\theta) Z_i \Big) 
\end{align}
where $\text{ReLU}(x)=\max\{x,0\}$. The vector $(A_n,B_n) \in \mathbb{R}^2$  is the joint distribution of 2 outputs from a single layer of a neural network with $n$ neurons when the inputs are correlated with correlation $\cos(\theta)$. Understanding how neural networks behave for correlated inputs is important because of the depth degeneracy problem whereby the correlation tends to $1$ as the depth of the network increases, see \cite{DepthDegeneracy}.

The joint distribution of $(A_n,B_n)$  can be written using the multi-variable Edgeworth expansion in terms of the functions $\bar{J}_{j,k}(\theta)$  which describe the expected powers of correlated ReLU-Gaussians, namely
$$\bar{J}_{j,k}(\theta) := \mathbb{E}_{G,Z \sim \mathcal{N}(0,1)}\Big[ \big( \sqrt{2}\text{ReLU}(G)\big)^j \big(\sqrt{2} \text{ReLU}(\cos(\theta) G + \sin(\theta) Z) \big)^k \Big],$$
Explicit formulas for the $J$  functions have been calculated in \cite{DepthDegeneracy}. The relevant formulas needed for the Edgeworth expansion are
\begin{align}
\bar{J}_{1,1}(\theta) &= 2 \frac{\sin\theta + (\pi-\theta)\cos\theta}{2\pi} \\
\bar{J}_{1,3}(\theta) &= 4 \frac{2(\cos\theta+1) + \sin^2\theta\cos\theta}{2\pi} \\
\bar{J}_{2,2}(\theta) &= 4 \frac{(\pi-\theta)(2\cos^2\theta+1)+3\sin\theta\cos\theta}{2\pi} 
\end{align}
From this, and the 4th order cumulant formulas from Example \ref{ex:4-cumulant}, one calculates that the covariance matrix and cumulants as:

\begin{align}
\Sigma &= \begin{bmatrix}
1 & \bar{J}_{1,1}(\theta) \\
\bar{J}_{1,1}(\theta) & 1
\end{bmatrix} \\
\kappa_{4,0} = \kappa_{0,4} &= 15 \\
\kappa_{3,1} = \kappa_{1,3} & =  \bar{J}_{1,3}(\theta) - 3\bar{J}_{1,1}(\theta)^2 \\
\kappa_{2,2} &= \bar{J}_{2,2}(\theta) - 1 - \bar{J}_{1,1}(\theta)^2
\end{align}
and the mixed Hermite polynomials can be computed as in Example \ref{ex:4-hermite} to yield
\begin{align}
\wt H_{4,0}(x_1,x_2,\Sigma) & = x_1^4 - 6x_1^2 + 3 \\
\wt H_{3,1}(x_1,x_2,\Sigma) &= x_1^3 x_2 - 3 x_1 x_2 - 3\bar{J}_{1,1}(\theta)x_1^2 + 3\bar{J}_{1,1}(\theta) \\
\wt H_{2,2}(x_1,x_2,\Sigma) &= x_1^2 x_2^2 - 4\bar{J}_{1,1}(\theta) x_1 x_2 - x^2_1 -  x^2_2 + 1 + 2\bar{J}_{1,1}(\theta), 
\end{align}
and anagously for $\wt H_{0,4}, \wt H_{1,3}$ by symmetry. Finally, the correction to the Gaussian density depends on the sum of the 5 terms

\begin{equation}\label{eq:paraGaussian} \rho_{(A_n,B_n)}(x_1,x_2) \approx \Big(1 + \frac{1}{n} \sum_{j=0}^4\frac{\kappa_{j,4-j}}{j! (4-j)!}\wt{H}_{j,4-j}(x_1,x_2;V)\Big)  \frac{1}{\sqrt{2\pi \det{\Sigma}}}e^{-\vec{x} \Sigma^{-1} \vec{x}}\end{equation}
This density approximation is compared to the ordinary Gaussian approximation and to Monte Carlo simulations in Figure \ref{fig:two} 

\end{example}
	\begin{figure}[!ht]
\includegraphics[width=\textwidth]{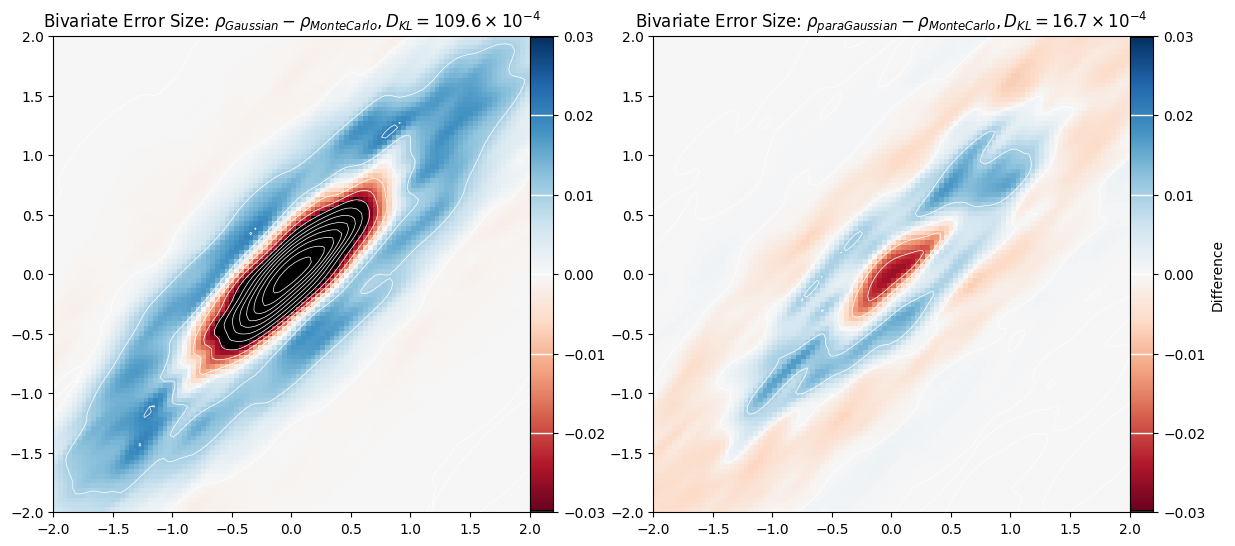}
\caption[Example of multivariable Edgeworth expansion]{ Comparing the error sizes of the pure Gaussian debsity approximation and the ``para-Gaussian'' density approximation, which includes Hermite polynomials of order 4 in  \eqref{eq:paraGaussian} from Example \ref{ex:2d-edgeworth}. For this experiment,  $\theta = 0.55$, $n=20$, and $\rho_{MonteCarlo}$ is estimated from $2^{18}$ Monte Carlo samples by using Kernel Density Estimation. As a single number encapsulation of how accurate the approximation is, we also compute the  KL divergence $D_{KL}=D_{KL}(\rho_{Approx},\rho_{MonteCarlo})$ from the approximation to the Monte Carlo estimate. This example is discussed in more detail in the workshop article \cite{nica2024improving}.  }
\label{fig:two}
\end{figure}

\section{More applications}\label{sec:misc}

\subsection{Exponential generating function}


\begin{proposition}
        \label{prop:generating-function}
        Let $x\in\bR$ and $t>0$. Recall the notation $\vp_t(x) = \frac{1}{\sqrt{2\pi t}}e^{-\frac{1}{2t} x^2}$. For any $\alpha\in\bR$,
        \begin{align*}
            \sum_{n=0}^\infty \wt H_n(x,t) \frac{\alpha^n}{n!} & = \exp\set{\alpha x - \frac{\alpha^2t}{2}} = \frac{\vp_t(x-t \alpha)}{\vp_t(x)}.
            \intertext{In particular, setting $t=1$, one has}
            \sum_{n=0}^\infty \fH_n(x)\frac{\alpha^n}{n!} &= e^{\alpha x - \half \alpha^2 } \\
        \end{align*}
    \end{proposition}
    \begin{proof}
        For the first line we first use the Gaussian integral formula \eqref{eq:formula} and then change the order of the expectation and sum, and noticing an exponetial sum $e^c = \sum_{k=0}^\infty \frac{c^k}{k!} $ appears. Using also $\mathbb{E}[e^{\imath c Z}] = e^{-\half c^2}$ where $Z \sim \mathcal{N}(0,1)$ a standard Gaussian random variable, we compute
        \begin{align*}
            \sum_{n=0}^\infty \wt H_n(x,t) \frac{\alpha^n}{n!} & = \sum_{n=0}^\infty \bE\ab{
            \frac{\rb{\alpha\rb{x+i\sqrt{t}Z}}^n}{n!}} = \bE\ab{\exp\set{\alpha x + i\alpha\sqrt{t}Z}}
             =\exp\set{\alpha x - \frac{\alpha^2t}{2}}
        \end{align*}
        as claimed. The second equality follows by Proposition \ref{prop:gaussian-shift-mean}.
    The result for $\fH_n(x)$ follows by setting $t=1$ since $\fH_n(x)=\wt H(x,1)$.
    \end{proof}

    \begin{remark}
    By the exponential generating function, we can recover $\wt H_n(x,t)$ as the $n$-th derivative
$$\wt H_n(x,t)  = \left.\frac{\partial^n}{\partial \alpha^n}\right|_{\alpha = 0}  e^{\alpha x - \half \alpha^2 t} $$
This identity can can then give an alternate proof of the orthogonality of the polynomials $\{ \wt H_n(x,t) \}_{n=1}^\infty$ w.r.t to the weight $\mathcal{N}(0,t)$ by pulling the derivatives out and using $\mathbb{E}[e^{ (\alpha + \beta) X}] = e^{\half (\alpha + \beta)^2 t}$ when $X \sim \mathcal{N}(0,t)$ to compute

\begin{align*}
\mathbb{E}_{X \sim \mathcal{N}(0,t)}\left[\wt H_n(X, t) \cdot \wt H_m(X, t)\right] 
&= \left. \frac{\partial^n}{\partial \alpha^n} \right|_{\alpha=0} 
  \left. \frac{\partial^m}{\partial \beta^m} \right|_{\beta=0} 
  \, \mathbb{E} \left[ 
    e^{\alpha X - \half \alpha^2 t} 
    e^{\beta X - \half \beta^2 t} \right] \\
    &= \left. \frac{\partial^n}{\partial \alpha^n} \right|_{\alpha=0} 
  \left. \frac{\partial^m}{\partial \beta^m} \right|_{\beta=0} 
  \,  
    e^{ \alpha \beta t } \\
    &= \left. \frac{\partial^n}{\partial \alpha^n} \right|_{\alpha=0} 
  (\alpha t)^m 
  = \begin{cases} n! t^n &\text{ if }n=m \\ 0&\text{ if } n\neq m\end{cases}
\end{align*}

    \end{remark}

	\subsection{Brownian martingales from the Hermite polynomials}
	\label{subsec:martingale}

    If $B_t$ is a standard Brownian motion, then one can create new polynomial martingales from $B_t$ by plugging into the Hermite polynomial $\wt  H_n(B_t,t)$. The first few are
    
    $$\wt H_1(B_t,t) = B_t, \quad \wt H_2(B_t,t) = B^2_t - t, \quad \wt H_3(B_t,t) = B^3_t - 3 B_t t, \quad \wt H_4(B_t,t) = B^4_t - 6B_t^2 t + 3t^2 $$

 In combination with tools like the Doob optional stopping theorem, these martingales are useful for calculating properties of Brownian motion like expected hitting times.
 
 One can intuitively see these are martingales by starting with the exponential martingale $M_t(\alpha) = e^{\alpha B_t} /\mathbb{E}[e^{\alpha B_t}] = e^{\alpha B_t - \half \alpha^2 t}$, and noticing this is precisely the exponential generating function of the Hermite polynomials $e^{x \alpha - \half \alpha^2 t} = \sum_{n=0}^\infty \frac{ \alpha^n}{n!} \wt H_n(x,t)$, from Proposition \ref{prop:generating-function}. Therefore, $\wt H_n(B_t,t)$ is also a martingale, because can be realized as a derivative of the martingale $M_t(\alpha)$ as $\alpha \to 0$

$$\wt H_n(B_t,t) =  \left.\frac{\partial^n}{\partial \alpha^n} \right|_{\alpha = 0}M_t(\alpha)  = \left.\frac{\partial^n}{\partial \alpha^n}\right|_{\alpha = 0}  e^{\alpha B_t - \half \alpha^2 t} $$

This derivative formula shows that $\wt H_n(B_t,t)$ is a limit of linear combinations of $M_t(\alpha)$, and since $M_t(\alpha)$ is a martingale, $\wt H_n(B_t,t)$ should be too. 

However, checking the techinical details to make sure the derivative works (which is an $\alpha \to 0$ limit) can be annoying. The Gaussian integral formula provides an alternative method in terms of complex Brownian motion $B_t + \imath X_t$ for two independent (real) Brownian motions $B_t$ and $X_t$.
    
	\begin{proposition}
		\label{prop:Martingale}
		Fix $n\in\bN$ and let $X_t$ be a one-dimensional Brownian motion. The collection of random variables $\mg n t = \wt H_n(B_t,t)$ defines a martingale $t\longmapsto \mg n t$.
	\end{proposition}
	
	\begin{proof}
		Let $\cF_\ast$ be the filtration generated by the Brownian motion $B$. By \eqref{eq:HExpectation} we can write
		\begin{align*}
			\mg n t & = \bE\ab{\rb{B_t + i X_t}^n\vert\cF_t}=\bE\ab{Z_t^n\vert\cF_t},
		\end{align*}
		where $X$ and $B$ are independent Brownian motions, so that $Z_t=B_t+i X_t$ is a complex Brownian motion. By Theorem 7.18 of \cite{LeGall}, holomorphic functions of conformal martingales are  martingales. Therefore $Z_t^n$ is a martingale and so is $\mg n t$.
	\end{proof}
	
	\begin{remark}
		In terms of $H_n$, Proposition \ref{prop:Martingale} states that $t\longmapsto t^n H_n(B_t,t)$ is a martingale for each $n\in \bN$.
	\end{remark}
\begin{remark}
Since $\fH_n(B_t,t)$ is a Brownian martingale, by Ito's lemma one has $\left( \partial_t + \frac{1}{2}\partial_{xx} \right) \fH_n(x,t) = 0$, which is a non-trivial identity but can be alternatively proven by the recurrence relationship Lemma \ref{lem:recurrence-with-integration-by-parts}. Since $\fH_n(x,0)=x^n$, this PDE shows that $\fH_n(x,t)$ can be written as solutions of the backwards heat equation $$ \fH_n(x,t) = e^{-t \frac{1}{2} \partial_{xx}} x^n $$
where $g(x,t) = e^{-t \frac{1}{2}\partial_{xx}} f(x)$  is the solution to the PDE $\partial_t g = - \frac{1}{2}\partial_{xx} g(x,t)$ with initial data $g(x,0)=f(x)$. In this case, one can also solve the PDE by doing a formal Taylor Series expansion for $\exp$ as only finitely many terms are non-zero against the polynomials $x^n$. For example, when $n=4$ one has:
$$\fH_n(x,t) = \left(1 - t\frac{1}{2}\partial_{xx} + \frac{1}{2!} t^2 (\frac{1}{2}\partial_{xx})^2 + \ldots \right) x^4 = x^4 - 6tx^2 + 3t^2$$
This is yet another way to derive the sum formula $\fH_n(x,t) =\sum_{k=0}^{\lfloor{n/2}\rfloor} \frac{n!}{2^k \cdot k! \cdot (n - 2k)!}(-1)^k t^k x^{n-2k}$, which is equivalent to the combinatorial formula from Section \ref{sec:combinatorics} . 
\end{remark}



\subsection{Identities proven with Gaussian integration-by-parts}

There are many identites satisfied by the Hermite polynomials. Many of them have novel proofs using the Gaussian integral and tricks like the Gaussian integration by parts $\mathbb{E}_{Z\sim\mathcal{N}(0,1)}[Zg(Z)]=\mathbb{E}_{Z\sim\mathcal{N}(0,1)}[g'(Z)]$  for any smooth function $g$. This is the $n=1$ case of the more general Hermite Gaussian integration by parts formula Theorem \ref{thm:int-by-parts} since $H_1(x)=x$. This identity is sometimes filed under the name ``Stein's lemma'' even though Stein's lemma is the reverse implication that concludes that a variable $Z$ that satisfies this $\forall g$ must be Gaussian. 

\begin{lemma}\label{lem:recurrence-with-integration-by-parts}
    The Hermite polynomials satisfy the recurrence:
    $$\fH_{n+1}(x)= x\fH_{n}(x) - n\fH_{n-1}(x)$$
\end{lemma}
\begin{proof}
Write $\fH_{n+1}(x) = \mathbb{E}[(x+\imath Z)^{n+1}] = x\mathbb{E}[(x+\imath Z)^n] + \imath \mathbb{E}[Z(x+\imath Z)^{n}]$ but by the Gaussian integration by parts formula, we take the derivative to find $\mathbb{E}[Z(x+\imath Z)^{n}] = \imath n \mathbb{E}[(x+\imath Z)^{n-1}]$, and the result follows.

\end{proof}
\begin{lemma}\label{lem:deriv-of-H}
The derivative of the Hermite polynomial $\fH_n$ is related to the previous Hermite polynomial $\fH_{n-1}$: $$\fH_n^\prime(x) = n \fH_{n-1}(x)$$
\end{lemma}
\begin{proof}
This follows by differentiating under the integral in the Gaussian integral formula
$$\fH_n^\prime(x)=\frac{d}{dx}\mathbb{E}[(x+\imath Z)^n]=\mathbb{E}[\frac{d}{dx}(x+\imath Z)^n]=n\mathbb{E}[(x+\imath Z)^{n-1}]=n\fH_{n-1}(x)$$
\end{proof}

\begin{proposition} \label{prop:Christofel-Darboux}
	\textbf{Christoffel-Darboux formula}	For any $x\in\bR$
		\begin{align}
		\sum_{k=0}^{n-1}\frac{1}{k!}\fH_{k}(x)\fH_{k}(y) &=\frac{1}{(n-1)!}\frac{\fH_{n}(x)\fH_{n-1}(y)-\fH_{n-1}(x)\fH_{n}(y)}{x-y} \\
        &=\frac{1}{n!}\frac{\fH_{n}(x)\fH^\prime_{n}(y)-\fH^\prime_{n}(x)\fH_{n}(y)}{x-y}
		\end{align}
\end{proposition}
	\begin{proof}
		The second equality follows from the first by Lemma \ref{lem:deriv-of-H}.  To prove the first equality, it is possible to do a simple proof by induction induction using the recursiveness formula Proposition \ref{prop:recursive}, $\fH_{n+1}(x)=x\fH_n(x)-n\fH_{n-1}(x)$, which is a standard method of proof. However, we will instead give here a proof using the Gaussian integral formula $\fH_n(x)=\mathbb{E}_{\mathcal{N}(0,1)}[(x+\imath Z)^n]$ and the Gaussian integration by parts formula. We use the letters $L,R\sim\mathcal{N}(0,1)$  so that $\fH_n(x) = \mathbb{E}[(x+\imath L)^n]$ and $\fH_n(y)=\mathbb{E}[(y+\imath R)^n]$  are distinguished. Now consider the telescoping-type manipulation: 
        \begin{align*}
\fH_n(x)\fH_n(y) &= \mathbb{E}\bigl[ (x+\imath L)^n (y+\imath R)^n \bigr]\\ 
  &= \frac{1}{x-y} \mathbb{E}\bigl[  x(x+\imath L)^n (y+\imath R)^n - (x+\imath L)^n y(y+\imath R)^n \bigr] \\
  &= \frac{1}{x-y} \mathbb{E}\bigl[ (x+\imath L)^{n+1} (y+\imath R)^n  - (x+\imath L)^n (y+\imath R)^{n+1} \\ &\quad -\imath L (x+\imath L)^n (y+\imath R)^n + \imath R (x+\imath L)^n (y+\imath R)^n \bigr]
\end{align*}
We now apply the Gaussian integration by parts to the variable $L$ to see that $\mathbb{E}[L (x+\imath L)^n] = \imath n \mathbb{E}[(x+\imath L)^{n-1}]$  and analogously for $R$. We can hence replace this part of the expression and recognize the appearance of $\fH_{n+1},\fH_n,\fH_{n-1}$  appearing when we take the expectation
\begin{align}
\fH_n(x)\fH_n(y)
  = \frac{1}{x-y} \mathbb{E}&\bigl[ (x+\imath L)^{n+1} (y+\imath R)^n  - (x+\imath L)^n (y+\imath R)^{n+1} \\ & + n (x+\imath L)^{n-1} (y+\imath R)^n - n (x+\imath L)^n (y+\imath R)^{n-1} \bigr] \\
  = \frac{1}{x-y} &\bigl( \fH_{n+1}(x)\fH_n(y) - \fH_{n}(x)\fH_{n+1}(y) \\
  &+n \fH_{n-1}(x)\fH_n(y) -n \fH_{n}(x)\fH_{n-1}(y) 
\end{align}
Finally, dividing through by $n!$  we obtain
$$\frac{1}{n!} \fH_n(x)\fH_n(y) = \frac{\fH_{n+1}(x)\fH_n(y) - \fH_n(x)\fH_{n+1}(y)}{n!(x-y)} - \frac{\fH_{n}(x)\fH_{n-1}(y) - \fH_{n-1}(x)\fH_{n}(y)}{(n-1)!(x-y)}$$
  and the result follows by summing over $n$ and recognizing the telescoping sum.      
	\end{proof}

\subsection{The expected characteristic function of the GUE random matrix}

    An application of the combinatorial formula is to provide a direct proof that the expected characteristic function of the GUE is precisely the Hermite polynomial. The traditional proof of this fact is by induction using the recursive formula for determinants and for Hermite polynomials.
    
\begin{proposition}
Fix $n \in \mathbb{N}$ and let $M$ be the $n\times n$ GUE matrix with independent $\mathcal{N}(0,1)$ entries on the diagonal, independent complex entries of the form $\mathcal{N}(0,\half) + \imath \mathcal{N}(0,\half)$ above the diagonal and set to be Hermitian $M_{ij} = \overline{M}_{ji}$ below the diagonal. Then the expected characteristic function of $M$ is the $n$-th Hermite polynomial $\fH_n$ :
\begin{align}
\mathbb{E}\Big[ \det(x\Id_n - M) \Big] &= \fH_{n}(x)
\end{align}
If the entries are scaled so that the variance of each entry is $t$, then we get the Hermite polynomial $\wt H_n(x,t)$:
\begin{align}
\mathbb{E}\Big[ \det(x\Id_n - \sqrt{t} M) \Big] &= \wt H_{n}(x,t)
\end{align}
\end{proposition}
\begin{proof}
We prove the result for $\wt H_n$, the result for $\fH_n$ follow by setting $t=1$. The proof follows by observing the determinant can be expanded into a combinatorial sum that exactly coincides with the interpretation from By the permutation expansion for the determinant, we have:
\begin{align}
\mathbb{E}\Big[ \det(x\Id_n - \sqrt{t} M) = \sum_{\sigma \in S_n} \sgn (\sigma) \mathbb{E} \Big[ \prod_{j=1}^n \Big( x \mathbf{1}\{j = \sigma(j)\} - \sqrt{t} M_{j \sigma(j)}  \Big) \Big],
\end{align}
where $S_n$ is the symmetric group of permutations. We now claim that the only permutations $\sigma \in S_n$ that have a non-zero contribution to the sum are precisely the permutations $\sigma$ whose cycle decomposition consists of only 1-cycles (fixed points) and 2-cycles (disjoint transpositions). Firstly, notice that the expectation splits as a product over the cycles of $\sigma$ by independence because disjoint cycles cannot contain repeated indices so have independent components. Now notice the that the the expected value of any cycle length 3 or greater is zero: $\mathbb{E}[M_{j_1 j_2} M_{j_2 j_3} M_{j_3 j_1}] = 0$ since all of these entries are independent and mean zero. 2-cycles of $\sigma$ have a non-zero contribution by the Hermitian property of the matrix $\mathbb{E}[\sqrt{t} M_{j_1 j_2} \cdot \sqrt{t}M_{j_2 j_1}]=\mathbb{E}[t|M_{j_1 j_2}|^2]=t$. 1-cycles (aka fixed points of $\sigma$) have a contribution through the identity matrix term even though the matrix entry is mean 0, $\mathbb{E}\left[ x \mathbf{1}\{j = \sigma(j)\} - \sqrt{t} M_{j \sigma(j)}  \right] = x$ when $\sigma(j)=j$.  Hence we arrive at:
\begin{align}
\mathbb{E}\Big[ \det(x\Id_n - \sqrt{t} M)\Big] = \sum_{\substack{\sigma \in S_n\\ \sigma \text{ has no cycles }\geq 3}} \sgn (\sigma) t^{|\{\text{2-cycles of }\sigma\}|} x^{|\{ \text{fixed points of }\sigma\}|},
\end{align}
Since $\sigma$ has no cylces of length $\geq 3$, we see that $\sgn(\sigma)$ is simply $(-1)^{|\{ \text{2-cycles of }\sigma\}|}$. Hence, factoring the $-1$ in with $t$, we get
\end{proof}
\begin{align}
\mathbb{E}\Big[ \det(x\Id_n - \sqrt{t} M) \Big] = \sum_{\substack{\sigma \in S_n\\ \sigma \text{ has no cycles }\geq 3}} (-t)^{|\{\text{2-cycles of }\sigma\}|} x^{|\{ \text{fixed points of }\sigma\}|},
\end{align}
Notice that such a permutation $\sigma$ can be interpreted as a choice of variables and pairing for the Hermite polynomial $\wt H_n$ as in \eqref{eq:H_sum}; every fixed point of $\sigma$ is a choice of the variable $x$ and every 2-cycle $(a, b)$ is a pairing of two $\imath \sqrt{t} Z$ variables at positions $a$ and $b$ to give a contribution of $\mathbb{E}[(\imath \sqrt{t} Z)(\imath \sqrt{t}Z)]=-t$ from the Wick/Isserlis theorem. With this point of view, we now recognize this as precisely the combinatorial formula for $\wt H_n(x,t)$ as desired. 

	\appendix

	\section{Proof of Lemma \ref{lem:cos-weak-convergence}}
	\label{sec:proof-lem-cos-weak-convergence}
	
	For the proof it will be helpful to define some functions. Let
	\begin{align*}
		g(x)=x\sqrt{1-x^2}\quad\text{ and }\quad P(x) = \frac{1}{2}\left(\arccos(x)+\frac\pi2\right).
	\end{align*}
	Then we have $G_n(x) = g(x) - \frac1n P(x)$. We further define $H_n(x)=\cos^2(n\,G_n(x)\bigr)$ so that $h_n(x) = 2H_n(C_n x)$.
	
	Since $C_n=\sqrt{1+\frac1{2n}}\to 1$, a change of variables argument confirms that weak convergence of $h_n$ to 1 is equivalent to weak convergence of $H_n$ to $\frac12$. 
	We do this in two steps:
	\begin{enumerate}[label=Step \arabic*., leftmargin=*, align=left]
		\item Prove that 
		\begin{align}
			\label{eq:to-be-shown-appendix}
			\int_{-1}^1 \rb{H_n(x)-\frac12}\phi(x) dx\longrightarrow 0
		\end{align} as $n\to\infty$ for every test function $\phi\in C^2([-1,1])$.
		\item Use density of $C^2$ in the dual of $L^p$ to conclude weak convergence.
	\end{enumerate}

	\paragraph{\textbf{Step 1.}} By the double angle and addition formulae,
	\begin{align*}
		H_n (x)-\frac12 = \frac12 \cos(2n G_n(x)) =  \frac12 \cos(2n g(x))\cos(2P(x)) - \frac12\sin(2ng(x))\sin(2P(x)).
	\end{align*}
	Since $\cos\circ P,\sin\circ P\in C^2([-1,1])$ and the latter space is closed under multiplication, \eqref{eq:to-be-shown-appendix} follows from the following lemma.
	
	\begin{lemma}
		\label{lem:convergence-exp-g}
		For any $\phi\in C^2([-1,1])$,
		\begin{align}
			\label{eq:reformulated-appendix}
			\int_{-1}^1 e^{2in g(x)}\,\phi(x)dx \xrightarrow[n\to\infty]{} 0.
		\end{align}
	\end{lemma}
	
	\paragraph{\textbf{Step 2.}} The space $C^2([-1,1])$ is dense in the dual $L^p([-1,1])^\ast \cong L^q([-1,1])$ of $L^p([-1,1])$. Therefore the fact that \eqref{eq:to-be-shown-appendix} holds for all $\phi \in C^2([-1,1])$ implies weak convergence of $H_n$ to $\frac12$.

	It remains to show Lemma \ref{lem:convergence-exp-g}.
	


	\begin{proof}[Proof of Lemma \ref{lem:convergence-exp-g}]
		Since $g'(x)=\frac{1-2x^2}{\sqrt{1-x^2}}$, the function $g$ has critical points at $x_j=\pm \frac1{\sqrt2}$. Note that 
		\begin{align*}
			g''(x)=\frac{x(2x^2-3)}{(1-x^2)^{3/2}}
		\end{align*}
		stays bounded near $x_1$ and $x_2$ (only blows up at $\pm 1$).
		Fix now some $\delta\in (0, 1-\frac1{\sqrt2})$ and define
		\begin{align}
			U_j = [x_j-\delta,x_j+\delta] \quad(j\in\set{1,2}),\quad\quad U = U_1\cup U_2
		\end{align}
		and let $I_1$, $I_2$ and $I_3$ be the three connected components of $[-1,1]\setminus U$. Then there exists $C>0$ such that $g'(x)\geq C$ for $x\in I:=I_1\cup I_2\cup I_3$. We will bound the integrals over the $I$ and the $U$ intervals separately.
		
		\paragraph{Integral over the $I$ intervals.} Let $I_j= [a_j,b_j]$ (for example $a_1=-1$ and $b_1 = x_1-\delta$). Then for each $j\in\set{1,2,3}$,
		\begin{align*}
			\int_{I_j} e^{i2ng(x)}\,\phi(x) dx
			&=\Bigl[\frac{\phi}{i2ng'}\,e^{i2ng}\Bigr]_{a_j}^{b_j}
			-\int_{I_j}e^{i2nG_n}
			\,\frac{d}{dx}\!\Bigl(\frac{\phi}{i2ng'}\Bigr)\,dx\\
			& =\frac1n \set{ \ab{\frac{\phi}{i2G_n'}\,e^{2ing}}_{-1}^1
				-\frac1{2i}\int_{I_j}e^{i2ng}
				\rb{\frac{\phi''(x)g(x) - \phi'(x)g'(x)}{g'(x)^2}} \,dx},
		\end{align*}
		which is $O(1/n)$ since $G'$ is bounded below outside of $U_\delta$
		
		\paragraph{Integral over the $U$ intervals}. Since $g'(x_j)=0$ we have
		the Taylor expansion 
		\begin{align*}
			g(x)&=g(x_{j})
			+\tfrac12g''(x_j)(x-x_{j})^2+R(x),
			\quad |R(x)|\le C|x-x_{j}|^3.
		\end{align*}
		With the change of variable $t=\sqrt{n}\,(x-x_{j})$, we get
		\begin{align*}
			\int_{U_j} e^{2ing(x)}\phi(x) dx & = e^{2in g(x_j)} \frac1{\sqrt n} \underbrace{\int_{-\delta\sqrt n}^{\delta\sqrt n} e^{i g''(x_j)t^2+2inR(t)}\phi\rb{x_j+\frac t{\sqrt n}}\, dt}_{A_j}
		\end{align*}
		It remains to show that the integral $A_j$ on the right hand side is bounded. To see this, note that the terms inside the exponential are purely imaginary, so that
		\begin{align*}
			\abs{A_j} & \leq \int_{-\delta\sqrt{n}}^{\delta\sqrt n} \abs{\phi\rb{x_j+\frac t{\sqrt n}}}\, dt = \int_{I_j}\abs{\phi(y)}\, dy\leq \norm{\phi}_1.
		\end{align*}
		In summary, we have partitioned the interval into the five sub-intervals $I_1,U_1,I_2,U_2,I_3$ and shown that the integral of $e^{2ing(x)}\phi(x)$ over each of the elements of the partition is $O\rb{n^{1/2}}$. This completes the proof.
		
	\end{proof}
	

\section{Proof of the Mills inequality for the Gaussian integral}
\label{sec:proof-Mills}

    We first observe that 
    \begin{align*}
        \intop_A^\infty e^{-B(x-c)^2} \d x & = \frac1{\sqrt{2B}} \intop_{\sqrt {2B}(A-c)}^\infty e^{-t^2/2}\, \d t
    \end{align*}
    which follows by making the change of variables $t=\sqrt {2B} (x-c)$, so $\d x= \frac {\d t}{\sqrt B}$.
    Therefore it suffices to prove the case when $A=\lambda$, $B=\frac12$ and $c=0$, namely,
    $$ \frac{1}{\lambda + \frac{1}{\lambda}}e^{-\frac{1}{2}\lambda^2} \leq \intop_\lambda^\infty e^{-\frac{1}{2}x^2} \d x \leq \frac{1}{\lambda} e^{-\frac{1}{2}\lambda^2},$$
which is the form in which the Mills inequality is usually stated.

The upper bound is found by using the estimate $t/\lambda>1$ in the range of integration:
    \begin{align*}
        \intop_{\lambda}^{\infty}e^{-t^{2}/2}dy	\leq	\intop_{\lambda}^{\infty}\frac{t}{\lambda}e^{-t^{2}/2}dy
	=	\frac{1}{\lambda}\left[-e^{-t^{2}/2}\right]_{\lambda}^{\infty}
	=	\frac{1}{\lambda}e^{-\lambda^{2}/2}
    \end{align*}

    To prove the lower bound, we use a remarkable anti-derivative \cite{McKean2012}. First note that 
    ${t^{4}+2t^{2}-1}<{t^{4}+2t^{2}+1}$ for any $y\in\bR$. Therefore
    \begin{align*}
        \intop_{\lambda}^{\infty}e^{-t^{2}/2}dy	\geq	\intop_{\lambda}^{\infty}e^{-t^{2}/2}\left(\frac{t^{4}+2t^{2}-1}{t^{4}+2t^{2}+1}\right)dy.
        \end{align*}
        Next comes the remarkable anti-derivative: a direct calculation shows that
        \begin{align*}
            \frac{\d}{\d t}\rb{-\frac{t}{t^{2}+1}e^{-t^{2}/2}} & = \frac{t^{4}+2t^{2}-1}{t^{4}+2t^{2}+1}\,e^{-t^{2}/2}.
        \end{align*}
        It follows that
    \begin{align*}
	\intop_{\lambda}^{\infty}e^{-t^{2}/2}dy \geq	\left[-\frac{t}{t^{2}+1}e^{-t^{2}/2}\right]_{\lambda}^{\infty}
	=	\frac{\lambda}{\lambda^{2}+1}e^{-\lambda^{2}/2}=\frac{1}{\lambda + \frac{1}{\lambda}}e^{-\lambda^{2}/2}.
    \end{align*}


	\section{Extra material for the Dyson Brownian motion construction}

	\subsection{Proof of Proposition \ref{prop:Watermelon-dist}}
	\label{subsec:proof-Watermelon-dist}
	
	In order to prove Proposition \ref{prop:Watermelon-dist}, we will use the following lemma.
	
	\begin{lemma}
		\label{lem:B_NI_from_0_ep}Let $\vec{B}(\cdot)$ be a Brownian motion.
		For any fixed $\vec{x}$ and $t$, we have in the limit $\ep\to0$:
		\[
		\p\left(\left\{ \vec{B}(t)=\vec{x}\right\} \cap NI_{t}\given{\vec{B}(0)=\vec{0}^{\ep}}\right)=\left(\frac{\ep}{t}\right)^{\binom{n}{2}}V(x_{1},\ldots,x_{n})\prod_{i=1}^{n}\vp_{t}(x_{i})+o(\ep^{\binom{n}{2}})
		\]
	\end{lemma}

	The proof of Lemma \ref{lem:B_NI_from_0_ep} relies on the Karlin--Macgregor formula \cite{KarlinMcGregor}:
	
	\begin{theorem}
		Suppose $Z_{1}(\cdot),\ldots,Z_{n}(\cdot)$ are i.i.d. random processes
		on a state space $\Si$ with transition probabilities:
		\[
		\ph_{t}(x,\d y) = \p\left(Z(t)=y\vert{Z(0)=x}\right)
		\]
		
		Let $\vec{x}=\left(x_{1},\ldots,x_{n}\right)$ be a vector of starting
		points and $\vec{y}=\left(y_{1},\ldots,y_{n}\right)$ be vector of
		end points. Suppose the random process has the following property:
		\[
		\forall\text{non-identity permutations }\si\in S_{n},\text{ }\text{there are no non-intersecting walks from }\vec{x}\text{ to }\left(y_{\si(1)},\ldots,y_{\si(n)}\right)
		\]
		More precisely, we require $\forall\si\in S_{n}$, $\si\neq Id$
		that:
		\[
		\p\left(\left\{ \vec{Z}(t)=\left(y_{\si(1)},y_{\si(2)},\ldots,y_{\si(n)})\right)\right\} \cap NI_{t}\given{\vec{Z}(0)=\vec{x}}\right)=0
		\]
		Then for any $\vec{x}\in\bZ^{n}$ and $\vec{y}\in\bZ^{n}$, we have:
		\[
		\p\left(\left\{ \vec{Z}(t)=\vec{y}\right\} \cap NI_{t}\given{\vec{Z}(0)=\vec{x}}\right)=\det\left[\ph_{t}(x_{i},y_{j})\right]_{i,j=1}^{n}
		\]
	\end{theorem}

	\begin{proof}[Proof of Lemma \ref{lem:B_NI_from_0_ep}]
		By the Karlin-Macgregor theorem, the probability is given by
		\begin{align*}
			\p\left(\left\{ \vec{X}(t)=\vec{x}\right\} \cap NI_{t}\given{\vec{X}(0)=\vec{0}^{\ep}}\right) & =\det_{i,j=1}^{n}\left[\vp_{t}\left(x_{i}-\ep(j-1)\right)\right]
		\end{align*}
		
		Now expand:
		\[
		\vp_{t}\left(x_{i}-\ep(j-1)\right)=\frac{1}{\sqrt{2\pi t}}\exp\left(-\frac{(x_{i}-(j-1)\ep)^{2}}{2t}\right)=\frac{1}{\sqrt{2\pi t}}\exp\left(-\frac{x_{i}^{2}}{2t}\right)\exp\left(-\frac{(j-1)^{2}\ep^{2}}{2t}\right)\exp\left(\frac{x_{i}(j-1)\ep}{t}\right)
		\]
		
		By the rules for determinants when every row/column is multiplied
		by a constant factor, we have:
		\begin{align*}
			\p\left(\left\{ \vec{X}(t)=\vec{x}\right\} \cap NI_{t}\given{\vec{X}(0)=\vec{0}^{\ep}}\right) & =\det_{i,j=1}^{n}\left[\vp_{t}\left(x_{i}-\ep(j-1)\right)\right]\\
			& =\left(\frac{1}{\sqrt{2\pi t}}\right)^{n}\prod_{i=1}^{n}\exp\left(-\frac{x_{i}^{2}}{2t}\right)\prod_{j=1}^{n}\exp\left(-\frac{(j-1)^{2}\ep^{2}}{2t}\right)\det\left[\exp\left(\frac{x_{i}(j-1)\ep}{t}\right)\right]
		\end{align*}
		
		Since $\lim_{\ep\to0}\exp\left(-\frac{(j-1)^{2}\ep^{2}}{2t}\right)=1$,
		this term will be negligible in the limit. It remains to expand out
		the remaining determinant. By serendipity, this determinant is exactly
		a Vandermonde determinant!
		\begin{align*}
			\det_{i,j=1}^{n}\left[\exp\left(\frac{x_{i}(j-1)\ep}{t}\right)\right] & =\det_{i,j=1}^{n}\left[\exp\left(\frac{x_{i}\ep}{t}\right)^{j-1}\right]\\
			& =V\left(\exp\left(\frac{x_{1}\ep}{t}\right),\ldots,\exp\left(\frac{x_{n}\ep}{t}\right)\right)\\
			& =\prod_{i<j}\left(\exp\left(\frac{x_{i}\ep}{t}\right)-\exp\left(\frac{x_{j}\ep}{t}\right)\right)\\
			& =\prod_{i<j}\left(\frac{\ep}{t}x_{i}-\frac{\ep}{t}x_{j}+O\left(\frac{\ep^{2}}{t^{2}}\right)\right)\\
			& =\left(\frac{\ep}{t}\right)^{\binom{n}{2}}V(x_{1},\ldots,x_{n})+o(\ep^{\binom{n}{2}})
		\end{align*}
		from which the result follows.
	\end{proof}
	
	
	\begin{proof}[Proof of Proposition \ref{prop:Watermelon-dist}] 
		By decomposing the interval $[0,t^\ast]$ into $[0,t]$ and $[t,t^\ast]$ and using stationarity of Brownian motion, we obtain
		\begin{align*}
			\p\left(\vec{W}(t)=\vec{x}\right) & =\lim_{\ep\to0}\p\left(\vec{W}^{\ep}(t)=\vec{x}\right)\\
			& =\lim_{\ep\to0}\frac{\p\left(\left\{ \vec{B}(t)=\vec{x}\right\} \cap NI_{t}\given{\vec{B}(0)=\vec{0}^{\ep}}\right)\p\left(\left\{ \vec{B}(t)=\vec{x}\right\} \cap NI_{t^{\ast}-t}=\vec{x}\given{\vec{B}(0)=\vec{0}^{\ep}}\right)}{\p\left(\left\{ \vec{B}(t^{\ast})=\vec{0}^{\ep}\right\} \cap NI_{t^{\ast}}\given{\vec{B}(0)=\vec{0}^{\ep}}\right)}
		\end{align*}
		We now apply Lemma \ref{lem:B_NI_from_0_ep} to each factor and obtain
		\[
		\p\left(\vec{W}(t)=\vec{x}\right)=\lim_{\ep\to0}\frac{\left(\frac{\ep}{t}\right)^{\binom{n}{2}}V(x_{1},\ldots,x_{n})\prod_{i=1}^{n}\vp_{t}(x_{i})\cdot\left(\frac{\ep}{t^{\ast}-t}\right)^{\binom{n}{2}}V(x_{1},\ldots,x_{n})\prod_{i=1}^{n}\vp_{t^{\ast}-t}(x_{i})}{\left(\frac{\ep}{t^{\ast}}\right)^{\binom{n}{2}}V(0,\ep\ldots,\ep(n-1))\prod_{i=1}^{n}\vp_{t^{\ast}}(\vec{0}^{\ep})}
		\]
		Since $V(0,\ep\ldots,\ep(n-1))=\ep^{\binom{n}{2}}V(0,1,\ldots,n-1)=\ep^{\binom{n}{2}}\prod_{k=0}^{n-1}k!$
		the powers of $\ep$ in numerator and denominator cancel, which allows us to take the $\ep\to0$ limit.
	\end{proof}

	\begin{remark}
		In Proposition \ref{prop:Watermelon-dist} we applied Lemma \ref{lem:B_NI_from_0_ep}
		with $\vec{x}=\vec{0}^{\ep}$, but the Lemma is only proven for fixed
		$\vec{x}$ (not $\vec{x}$ depending on $\ep$). It is a simple exercise
		to adapt the proof of the lemma to compute the asymptotic form of
		$\p\left(\left\{ \vec{X}(t)=\vec{0}^{\ep}\right\} \cap NI_{t}\given{\vec{X}(0)=\vec{0}^{\ep}}\right)$
		.
	\end{remark}
	
	\subsection{Determinantal formula}
	
	\begin{proof}[Proof of Proposition \ref{prop:expectation-determinant}]
		Let $\Xi_{k}$ be the random measure on $B^{k}\subset\bR^{k}$
		which keeps track of subsets of size $k$ from $S$:
		\[
		\Xi_{k}=\sum_{\substack{A\subset S\\
				\abs A=k
			}
		}\sum_{\si\in S^{n}}\de_{\left(X_{\si(1)},\ldots,X_{\si(k)}\right)}.
		\]
		Expanding out the product on the left-hand side we obtain
		\begin{align*}
			\prod_{X\in S}\left(1+\ph(X)\right) & =  \sum_{k=0}^{\infty}\sum_{\substack{A\subset S\\
					\abs A=k
				}
			}\prod_{X\in A}\ph(X)
			=\sum_{k=0}^{\infty}\frac{1}{k!}\intop_{B^{k}}\prod_{i=1}^{k}\ph(x_{i})\ \d\Xi_{k}(\vec{x})
		\end{align*}
		We claim that $\bE[\Xi_k(U)]=\int_U\rho_k(x)\, dx$ for any $U\subseteq\bR^k$. Indeed, this follows immediately for $U=\prod_{j=1}^k U_j$ and any $U_j\subseteq\bR$, and the general case can then be deduced from the monotone class theorem. Thus we obtain
		\begin{align*}
			\bE\ab{\prod_{X\in S}\left(1+\ph(X)\right)} & = \sum_{k=0}^{\infty}\frac{1}{k!}\intop_{B^{k}}\prod_{i=1}^{k}\ph(x_{i})\ \det\rb{K\rb{x_i,x_j}} \d x_1\ldots\d x_k
		\end{align*}
		and the result follows by multiplying the $\phi(x_j)$ inside the determinant.
	\end{proof}

    \subsection{Proof of Proposition \ref{prop:spiked-DPP}}
	\label{app:proof-spiked-DPP}

    Proceeding similarly to the `unspiked' case, i.e. by taking the limit of $\epsilon$-of-room watermelons  (see also \cite{Johansson2001CMP}), the spiked DBM process is a biorthogonal ensemble of the form \eqref{eq:biorthog_ensemble} with $\tilde \psi_j(x) = x^j$ and
    \begin{align}
        \notag
        \tilde\phi_j(x) & = \begin{cases}
            x^j \varphi_t(x)\quad & \text{if } j<n\\
            \varphi_t(x-c) & \text{if } j=n.
        \end{cases}
    \end{align}
    recalling the notation $\varphi_t$ for the Gaussian density of variance $t$.
    By applying row and column operations, we can replace $\phi_j$ by any monic polynomial of degree $j$. We will choose $\phi_j= \wt H_j(\cdot,t)$ for some $t$ to be determined later.
    In order to apply Theorem \ref{thm:bi-orthog-thm}, we need to compute the Gram matrix. As for the unspiked case, $G_{j\ell}=$  whenever whenever $j\ne \ell$ and  $j,\ell<n$. Moreover, for any $j\in [n]$, by Proposition \ref{prop:generating-function},
    \begin{align*}
        G_{Nj} & = \intop_{-\infty}^\infty \phi_t(x-c)\wt H_j(x,t)\d x
         =\intop_{-\infty}^\infty \left(\sum_{k=0}^{\infty}\wt H_k(x,t) \left(\frac{c}{t}\right)^{k}\frac{1}{k!}\right)\wt H_{j}(x,t)\phi_t(x)\d x\\
        & =\sum_{k=0}^{\infty} \frac{1}{k!} \left(\frac{c}{t}\right)^{k} \intop_{-\infty}^\infty  \wt H_k(x,t) \wt H_{j}(x,t)\phi_t(x)\d x\\&=\left(\frac{c}{t}\right)^{j}\frac{1}{j!}\left(t^{j}j!\right)= c^j.
    \end{align*}
    Using the inversion formula for block diagonal matrices, we obtain
    \begin{align*}
G^{-1}=\left[\begin{array}{ccccc}
 &  &  & \vdots\\
 & diag\left(\frac{1}{t^{i}i!}\right) &  & \vdots & \vec{0}\\
 &  &  & \vdots\\
\ldots & \ldots & \ldots & .\\
 & -c^{-n}\left(\frac{c}{t}\right)^{j}\frac{1}{j!} &  &  & \frac{1}{c^{n}}
\end{array}\right]
\end{align*}


It follows that the spiked DBM is a DPP with kernel
\begin{align*}
\wt\spikedkernel_{n,t}(x,y)&=\sum_{i,j=0}^{n}\tilde{\ph}_{i}(x)\tilde{\ps}_{j}(y)\left[\tilde{G}^{-t}\right]_{ij}\\
    &=\sum_{i=0}^{n-1}\tilde{\ph}_{i}(x)\tilde{\ps}_{i}(y)\frac{1}{t^{i}i!}+\frac{1}{c^{n}}\tilde{\ph}_{n}(x)\tilde{\ps}_{n}(y)-\frac{1}{c^{n}}\tilde{\ps}_{n}(y)\sum_{j=0}^{n-1}\left(\frac{c}{t}\right)^{j}\frac{1}{j!}\tilde{\ph}_{i}(x)\\
    &=\sum_{i=0}^{n-1}\frac{\wt H_i(x,t) \wt H_i(y,t)}{t^{i}i!}\vp_{t}(x)+\frac{1}{c^{N}}\wt H_n(y,t)\vp_{t}(x-c)-\frac{1}{c^{N}}p_{t}^{N}(y)\sum_{j=0}^{N-1}\left(\frac{c}{t}\right)^{j}\frac{1}{j!}p_{t}^{j}(x)\vp_{t}(x).
\end{align*}
Proposition \ref{prop:spiked-DPP} now follows by conjugating $\wt\spikedkernel$ with the function $\sqrt{\phi_t}$.

\subsection{Proof of \texorpdfstring{\eqref{eq:third-term-spiked}}{the estimate of the third term}}
\label{sec:proof-third-term}

We have
\begin{align*}
    \sum_{j=0}^{n-1}  c^{j-n} \rb{\frac n\sigma}^{(j-n)/2} \sqrt{\frac n{\sigma}} &  \int_{\gamma-\ep}^\infty \abs{\Psi_n\rb{\frac {x\sqrt n}{\sqrt\sigma}}} \abs{\Psi_j\rb{\frac {x\sqrt n}{\sqrt\sigma}}}\, dx\\
    & \leq \sum_{j=0}^{n-1}  c^{j-n} \rb{\frac n\sigma}^{(j-n)/2} \sqrt{\frac n{\sigma}}  \int_{\gamma-\ep}^\infty \frac {C_\delta}{(nj)^{1/4}} e^{-(n+j)(x/(2\sqrt{\sigma})- 1)^2/2} \, dx\\
    & \leq \sqrt{\frac n{\sigma}}  \sum_{j=0}^{n-1}   \rb{\frac {c^2 n}\sigma}^{(j-n)/2} \frac {C_\delta}{(nj)^{1/4}} 
    \frac1{(n+j)(\gamma-\ep-2\sqrt\sigma)} e^{-(n+j)\frac{\gamma-\ep}{2\sqrt \sigma}/2}\\
 & \leq  n^{-3/4} \frac {C_\delta}{\sqrt{\sigma}\delta}  \sum_{j=0}^{n-1}   \rb{\frac {c^2 n}\sigma}^{(j-n)/2} 
 e^{-(n+j)(1+\frac{\delta}{2\sqrt \sigma})/2}\\
 & =  n^{-3/4} \frac {C_\delta}{\sqrt{\sigma}\delta} e^{-n(1+\frac{\delta}{2\sqrt {\sigma}})/2}  \rb{\frac {c^2 n}{\sigma}}^{-n/2}   \sum_{j=0}^{n-1}   \rb{ \sqrt{\frac {c^2 n}\sigma} e^{-(1+\frac{\delta}{2\sqrt \sigma})/2} }^{j} \\
 & \leq C e^{-n(1+\frac{\delta}{2\sqrt {\sigma}})/2} 
\end{align*}
as claimed.

\subsection{Proof of \texorpdfstring{\eqref{eq:spiked-critical-point}}{the critical point}}
\label{subsec:proof-critical-point}
For $x>\gamma-\ep>2\sqrt \sigma$,
\begin{align*}
    f'(x) &= \frac1{2\sqrt\sigma} S'\rb{\frac x{2\sqrt\sigma}}-\frac{x-c}{\sigma}\\ 
    & =  \frac1{\sqrt\sigma} \rb{\frac x{2\sqrt\sigma}-\sqrt{\rb{\frac x {2\sqrt\sigma}}^2-1}}-\frac{x-c}{\sigma}\\
    & = \frac{2c-x}{2\sigma}-\frac1{2\sigma}\sqrt{x^2-4\sigma}.
\end{align*}
Thus $f'(x) =0$ if and only if $2c-x=\sqrt{x^2-4\sigma}$  iff $4c^2-4cx+x^2=x^2-4\sigma$. Thus, the only critical point is $x_\ast = c+\frac\sigma c$.

    \printbibliography

@article{Withers2000,
  author    = {C. S. Withers},
  title     = {A simple expression for the multivariate Hermite polynomials},
  journal   = {Statistics \& Probability Letters},
  volume    = {47},
  number    = {2},
  pages     = {165--169},
  year      = {2000},
  publisher = {Elsevier}
}

@book{LeGall,
    author ={Le Gall, Jean-Fran{\c c}ois},
    title = {Brownian Motion, Martingales, and Stochastic Calculus},
    publisher = {Springer},
    year = {2016},
    series = {Graduate Text in Mathematics},
volume = {274}
}

@article{PolynomialCombinatorics,
author = {Labelle, Jacques and Yeh, Yeong-nan},
year = {1989},
month = {02},
pages = {},
title = {The Combinatorics of Laguerre, Charlier, and Hermite Polynomials},
volume = {80},
journal = {Studies in Applied Mathematics},
doi = {10.1002/sapm198980125}
}

@article{HermiteCombinatorics,
author = {Drake, Dan},
copyright = {2008 Elsevier Ltd},
issn = {0195-6698},
journal = {European journal of combinatorics},
language = {eng},
number = {4},
pages = {1005-1021},
publisher = {Elsevier Ltd},
title = {The combinatorics of associated Hermite polynomials},
volume = {30},
year = {2009},
}

@book{durrettPTE,
  author       = {Rick Durrett},
  title        = {Probability: Theory and Examples, 4th Edition},
  publisher    = {Cambridge University Press},
  year         = {2010},
  isbn         = {9780511779398},
  bibsource    = {dblp computer science bibliography, https://dblp.org}
}

@techreport{Sobel-tech-1963,
    author = {Milton Sobel},
    title = {Multivariate Hermite polynomials, Gram--Charlier expansions and Edgeworth expansions},
    institution = {University of Minnesota},
    year = {1963}
}

@book{AGZ, 
place={Cambridge}, 
series={Cambridge Studies in Advanced Mathematics}, 
title={An Introduction to Random Matrices}, 
publisher={Cambridge University Press}, 
author={Anderson, Greg W. and Guionnet, Alice and Zeitouni, Ofer}, 
year={2009}, 
collection={Cambridge Studies in Advanced Mathematics}
}

@book{Szego1975,
author = {Szeg{\"o}, G{\'a}bor},
year = {1975},
title = {Orthogonal polynomials},
volume = 4,
address = {Providence, Rhode Island},
publisher = {American Mathematical Society},
isbn = {0-8218-1023-5}
}

@article{BBP,
  author    = {Baik, Jinho and Ben Arous, G{\'e}rard and P{\'e}ch{\'e}, Sandrine},
  title     = {Phase Transition of the Largest Eigenvalue for Nonnull Complex Sample Covariance Matrices},
  journal   = {The Annals of Probability},
  volume    = {33},
  number    = {5},
  pages     = {1643--1697},
  year      = {2005}
}

@book{griffithsQM,
  author    = {David J. Griffiths},
  title     = {Introduction to Quantum Mechanics},
  publisher = {Pearson Education},
  year      = {2018},
  edition   = {3rd},
  isbn      = {9781107189638}
}

@article{JonesOConnell,
    author = {Liza Jones and Neil O'Connell},
    title = {Weyl Chambers, Symmetric Spaces and Number-Variance Saturation} ,
    journal = {Alea},
    year = {2006},
volume = {2},
pages = {91--118}
}

@article{Dyson,
    author = {F. J. Dyson},
    title = {A Brownian-motion model for the eigenvalues of a random matrix} ,
    journal = {J. Mathematical Physics},
    year = {1962},
volume = {3},
pages = {1191--1198}
}

@article{Biane1994,
    author = {Philippe Biane},
    title = {Quelques propri{\'e}t{\'e}s du mouvement Brownien dans un cone},
    journal = {Stochastic Processes and Applications},
    year = 1994,
    volume= {53},
    issue = {2}
}

@incollection{Borodin-HbRMT,
  author    = {Alexei Borodin},
  title     = {Determinantal Point Processes},
  booktitle = {The Oxford Handbook of Random Matrix Theory},
  editor    = {Gernot Akemann and Jinho Baik and Philippe Di Francesco},
  publisher = {Oxford University Press},
  year      = {2011},
  chapter   = {11},
  pages     = {231--249}
}

@article{McKean,
    author ={Henry P. McKean},
    title = {Fredholm determinants},
    journal = {Central European Journal of Mathematics},
    year = {2011},
    volume= {p},
    number = {2},
    pages = {205 -- 243}
}

@article{CepaLepingle,
author = {E. C{\'e}pa and D. L{\'e}pingle},
title = {Diffusing particles with electrostatic repulsion},
journal = {Probab. Theory Related Fields},
volume = {107},
number = {4}, 
pages =  {429--449},
    year = {1997}
}

@article{MillsRatio1941,
author = {Robert D. Gordon},
title = {{Values of Mills' Ratio of Area to Bounding Ordinate and of the Normal Probability Integral for Large Values of the Argument}},
volume = {12},
journal = {The Annals of Mathematical Statistics},
number = {3},
publisher = {Institute of Mathematical Statistics},
pages = {364 -- 366},
year = {1941}
}

@misc{McKean2012,
  author       = {Henry P. McKean},
  title        = {Lecture communication},
  howpublished = {Course: Probability Limit Theorems 2},
  note         = {New York University},
  year         = {2012}
}

@article{Chiani,
title = {Distribution of the largest eigenvalue for real Wishart and Gaussian random matrices and a simple approximation for the Tracy–Widom distribution},
journal = {Journal of Multivariate Analysis},
volume = {129},
pages = {69-81},
year = {2014},
author = {Marco Chiani}
}

@article{KarlinMcGregor,
    author = {Karlin, S. and McGregor, G.},
    title = {Coincidence probabilities},
    journal = {Pacific J. Math},
    year = {1959},
    volume = {9},
pages = {1141--1164}
}

@inproceedings{
nica2024improving,
title={Improving the Gaussian Approximation in Neural Networks: Para-Gaussians and Edgeworth Expansions},
author={Mihai Nica and Janosch Ortmann},
booktitle={NeurIPS 2024 Workshop on Mathematics of Modern Machine Learning},
year={2024},
url={https://openreview.net/forum?id=92q7WV4od7}
}

@article{DepthDegeneracy,
  author  = {Cameron Jakub and Mihai Nica},
  title   = {Depth Degeneracy in Neural Networks: Vanishing Angles in Fully Connected ReLU Networks on Initialization},
  journal = {Journal of Machine Learning Research},
  year    = {2024},
  volume  = {25},
  number  = {239},
  pages   = {1--45},
  url     = {http://jmlr.org/papers/v25/23-0350.html}
}

@article{Johansson2001CMP,
    author = {Kurt Johansson},
    title = {Universality of the local spacing distribution in certain ensembles of Hermitian Wigner matrices},
    journal = {Communications in Mathematical Physics},
    year = {2001},
    volume = {215},
pages = {683--705}
}

@article{Indritz1961,
author = {Jack Indritz},
title = {An inequality for Hermite polynomials},
journal = {Proceedings of the American Mathematical Society},
year = {1961},
volume = {12},
pages = {981--983}
}

@book{feller1971probability2,
  author    = {Feller, William},
  title     = {An Introduction to Probability Theory and Its Applications, Volume II},
  edition   = {2},
  publisher = {John Wiley \& Sons Inc.},
  address   = {New York},
  year      = {1971},
  pages     = {xxiv+669},
}

@incollection{HolcombVirag2019,
  author       = {Diane Holcomb and B{\'a}lint Vir{\'a}g},
  title        = {A Short Introduction to Operator Limits of Random Matrices},
  booktitle    = {Random Matrices},
  editor       = {Alexei Borodin and Ivan Corwin and Alice Guionnet},
  series       = {IAS/Park City Mathematics Series},
  volume       = {26},
  publisher    = {American Mathematical Society},
  year         = {2019},
  doi          = {10.1090/pcms/026}
}

@article{Ledoux2004,
  author       = {Michel Ledoux},
  title        = {Differential Operators and Spectral Distributions of Invariant Ensembles from the Classical Orthogonal Polynomials. The Continuous Case},
  journal      = {Electronic Journal of Probability},
  volume       = {9},
  year         = {2004},
  pages        = {no.\ 7, 177--208},
  doi          = {10.1214/EJP.v9-191},
  url          = {http://ejp.ejpecp.org/article/view/191},
  issn         = {1083-6489}
}

@article{Krasikov,
    author = {Ilia Krasikov},
    title = {New bounds on the Hermite polynomials},
    year = {2004},
archivePrefix = {arXiv},
       eprint = {math/0401310}
}

@article{Robbins-Stirling,
    author = {Robbins, Herbert},
    title = {A Remark on Stirling's Formula}, 
    volume = {62},
    number = {1},
    pages = {26–29},
    journal = {The American Mathematical Monthly},
    year =  {1955}
}

@article{MollVignat,
    author = {Vignat, Christophe and Moll, Victor H.},
    title = {A probabilistic approach to some binomial identities},
    journal = {Elemente der Mathematik},
    year = {2014},
    volume = {69},
    pages = {1--12}
}

@book{janson,
    author = {Svante Janson},
    title = {Gaussian Hilbert Spaces},
    publisher = {Cambridge University Press},
    year = {1997},
series = {Cambridge Tracts in Mathematics},
volume = {129}
}

@book{Temme,
    author = {N.M. Temme},
    title = {Special Functions. An introduction to the Classical Functions of Mathematical Physics},
    publisher = {Wiley},
    year = {1996},
    address = {New York}
}

@article{Appell,
    author = {Appell, Paul},
    title = {Sur les fonctions hypergéométriques de plusieurs variables, les polynômes d'Hermite et autres fonctions sphériques dans l'hyperespace},
    journal = {Mémorial des sciences mathématiques},
    year = {1925},
    volume = 3
}

@book{AKF,
    author = {P.Appell and J.Kampé de Fériet},
    title = {Fonctions Hypergéométriques et Hypersphériques; Polynômes d’Hermite},
    publisher = {Gauthier-Villars},
    year = {1926},
    address = {Paris}
}

@book{ismael,
    author = {M. Ismail},
    title = {Classical and Quantum Orthogonal Polynomials in One Variable},
    publisher = {Cambridge University Press},
    year = {2005}
}
    
\end{document}